\documentclass{amsart}
\usepackage{amsmath}
\usepackage{amssymb}
\usepackage{enumitem}
\newcommand{\limss}[2]{\displaystyle \lim_{{#1} \rightarrow {#2}}\mbox{ }}

\newcommand{\NN}{\mathbb{N}}
\newcommand{\RR}{\mathbb{R}}
\newcommand{\QQ}{\mathbb{Q}}

\newcommand{\UU}{\mathbb{U}}
\newcommand{\ZZ}{\mathbb{Z}}
\newtheorem{thm}{Theorem}
\newtheorem{Theorem}{Theorem}[section]
\newtheorem{Proposition}[Theorem]{Proposition}
\newtheorem{Lemma}[Theorem]{Lemma}
\newtheorem{Fact}[Theorem]{Fact}
\newtheorem{Corollary}[Theorem]{Corollary}
\newtheorem{Remark}[Theorem]{Remark}
\newtheorem{defin}[Theorem]{Definition}
\newenvironment{defi}{\begin{defin} \rm}{\end{defin}}
\newtheorem{defins}[Theorem]{Definitions}
\newenvironment{defis}{\begin{defins} \rm}{\end{defins}}
\newtheorem{example}[Theorem]{Example}

\newtheorem{notat}[Theorem]{Notation}
\newenvironment{nota}{\begin{notat} \rm}{\end{notat}}
\newtheorem{notats}[Theorem]{Notations}

\title{Model theory of divisible abelian cyclically ordered groups and minimal c.\ o.\ g.}
\author{G. Leloup}
\address{G. Leloup, 
Laboratoire Manceau de Math\'ematiques, 
LMM EA 3263, 
Le Mans Universit\'e, 
avenue Olivier Messiaen, 
72085 LE MANS CEDEX, 
FRANCE. 
ORCID: 0000-0003-4034-9402}
\email{gerard.leloup-At-univ-lemans.fr}
\date{November 13, 2021}
\begin{document}
\begin{abstract} We make available some results about model theory cyclically ordered groups. 
We start with a classification of complete theories of divisible abelian cyclically ordered groups. 
Then we look at the cyclically ordered groups where the only parametrically 
definable subsets are finite unions of singletons and open intervals, and those where the definable subsets are finite union of 
singletons and c-convex subsets, where being c-convex is the analogue of being convex in the linearly ordered case. 
\end{abstract}
\maketitle
%
%
\section{Introduction.}
\indent 
Cyclically ordered groups continue to interest mathematicians since they arise naturally in lower-dimensional 
topology (see, for example \cite{B15}) and in studying groups which act on a circle (e.\ g.\ \cite{C04}). 
In these papers, they are called circularly ordered groups. 
By studying the topological structure of the space of circular orders on a group, it was 
proved that this space is either finite or uncountable (\cite{CMR18}). 
This generalizes a result of Fran\c{c}ois Lucas which proved that there are $2^{\aleph_0}$ 
non isomorphic cyclic orders on the additive group $(\QQ,+)$ of rational numbers, by giving a description of these 
cyclic orders. In fact Lucas proved that there are 
$2^{\aleph_0}$ elementary classes of discrete divisible cyclic orders on $\QQ$ 
(we will generalize this proposition here). 
This result was part of a paper that he wanted to publish, however, Lucas died before 
achieving this task.
In this paper he used the spines of Schmitt of 
ordered abelian groups to study the model theory divisible abelian cyclically ordered groups. He also 
had studied different notions of minimality concerning cyclically 
ordered groups. This work may apply to 
multiplicative groups of fields since the multiplicative group of any field is cyclically orderable 
(\cite[Proposition 5.8]{Le}). We make available these results here. The present paper is written in order to make clear 
which results come from Lucas's preprint. 
To reach an uninitiated audience we have added 
cyclically ordered groups reminders. 
Section \ref{sec3} is similar to the original work of Lucas, however, the proofs have been drafted in detail 
and gaps have been filled. 
Section \ref{n4} has been completely overhauled, so that the theorems 
are written in terms of the cyclically ordered groups, instead of their unwounds, 
and without requiring the definitions of Schmitt (\cite{Sc1}, \cite{Sc2}). 
Furthermore, we obtain a classification of complete theories of 
divisible abelian cyclically ordered groups. 
Section \ref{sec5} has been revisited, and the proofs have been carefully written. 
This led to several errors being corrected. Thus the statements of lemmas have been changed. 
However we attribute to Lucas the authorship of these lemmas. 
\subsection{Cyclically ordered groups.}\label{subsec11} 
Let $G$ be a group and $R$ be a ternary relation on $G$. 
We say that $(G,R)$ is a {\it cyclically ordered group}, or that $R$ is a {\it cyclic order} on $G$, 
if the following holds. \\ 
$R$ is strict: for every  $g_1$, $g_2$, $g_3$, $R(g_1,g_2,g_3) \Rightarrow g_1\neq g_2\neq g_3
\neq g_1$.\\
$R$ is cyclic: 
for every  $g_1$, $g_2$, $g_3$, $R(g_1,g_2,g_3) \Rightarrow R(g_2,g_3,g_1)$.\\
An element $g$ being fixed, $R(g,\cdot,\cdot)$ 
induces a linear order relation on the set $G\backslash\{g\}$. \\
$R$ is compatible: for every $g_1$, $g_2$, $g_3$, $h$, $h'$,  
$R(g_1,g_2,g_3) \Rightarrow R(hg_1h',hg_2h',hg_3h')$ (\cite{R}). \\[2mm]
\indent 
For example, the group of complex numbers of norm $1$ is a cyclically ordered group by setting 
$R(\mbox{exp}(i\theta_1),\mbox{exp}(i\theta_2),\mbox{exp}(i\theta_3))$ if and only if either 
$0\leq\theta_1<\theta_2<\theta_3<2\pi$ or 
$0\leq\theta_2<\theta_3<\theta_1<2\pi$ or $0\leq\theta_3<\theta_1<\theta_2<2\pi$ (e.\ g.\ \cite{Sw59}). 
This condition is equivalent to saying that $0\leq\theta_{\sigma(1)}<\theta_{\sigma(2)}<\theta_{\sigma(3)}<2\pi$, 
for some permutation $\sigma$ in the alternating group $A_3$ of degree $3$. \\
\setlength{\unitlength}{1mm}
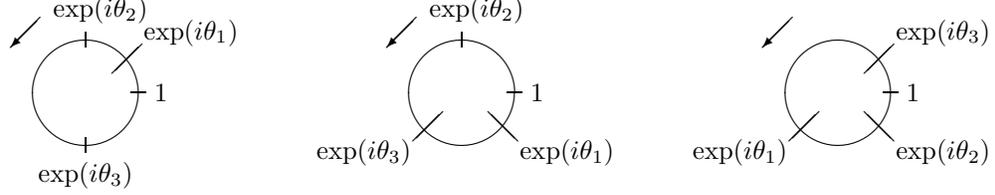
\begin{picture}(200,25)(0,0)
\put(20,13){\circle{14}}
\put(20,19){\line(0,1){2}}
\put(23.5,15.5){\line(1,1){4}}
\put(26,13){\line(1,0){2}}
\put(20,5){\line(0,1){2}}
\put(22,24){\makebox(0,0){$\mbox{exp}(i\theta_2)$}}
\put(34,21){\makebox(0,0){$\mbox{exp}(i\theta_1)$}}
\put(30,13){\makebox(0,0){$1$}}
\put(20,2){\makebox(0,0){$\mbox{exp}(i\theta_3)$}}
\put(14,23){\vector(-1,-1){4}}
\put(70,13){\circle{14}}
\put(70,19){\line(0,1){2}}
\put(76,13){\line(1,0){2}}
\put(73.5,10.5){\line(1,-1){4}}\put(67.5,10.5){\line(-1,-1){4}}
\put(72,24){\makebox(0,0){$\mbox{exp}(i\theta_2)$}}
\put(80,13){\makebox(0,0){$1$}}
\put(84,5){\makebox(0,0){$\mbox{exp}(i\theta_1)$}}
\put(57,5){\makebox(0,0){$\mbox{exp}(i\theta_3)$}}
\put(64,23){\vector(-1,-1){4}}
\put(120,13){\circle{14}}
\put(126,13){\line(1,0){2}}
\put(123.5,15.5){\line(1,1){4}}
\put(123.5,10.5){\line(1,-1){4}}
\put(117.5,10.5){\line(-1,-1){4}}
\put(107,5){\makebox(0,0){$\mbox{exp}(i\theta_1)$}}
\put(134,21){\makebox(0,0){$\mbox{exp}(i\theta_3)$}}
\put(134,5){\makebox(0,0){$\mbox{exp}(i\theta_2)$}}
\put(130,13){\makebox(0,0){$1$}}
\put(114,23){\vector(-1,-1){4}}
\end{picture}
\indent 
We denote by $\mathbb{U}$ the cyclically ordered group $\{\mbox{exp}(i\theta)\; :\; \theta\in [0,2\pi[\}$ 
of unimodular complex numbers, and by 
$T(\mathbb{U})$ its torsion subgroup. Note that every finite cyclically ordered group is cyclic 
(\cite[Theorem 1]{Zh}, \cite[Lemma 3.1]{BS18}), hence it embeds in $T(\mathbb{U})$. 
 We know that, for every positive integer $n$, 
$T(\UU)$ contains a unique subgroup of $n$ elements, which is the subgroup of $n$-th roots of $1$. 
We denote it by $T(\UU)_n$. \\[2mm]
\indent The language of cyclically ordered groups is $\{\cdot,R,e,^{-1}\}$, where the first predicate 
stands for the group operation, $R$ for the ternary relation, $e$ for the group identity and $^{-1}$ 
for the inverse function. When we consider abelian cyclically ordered  groups we will also use the usual symbols 
$+,0,-$. If necessary, the unit element of $G$ will be denoted by $e_G$ (or $0_G$) in order to avoid 
any confusion. For every $g_0$, $g_1,\dots ,g_n$ in $G$, we will let $R(g_0,g_1,\dots,g_n)$ stand for 
$\bigvee_{0\leq i<j<k\leq n}R(g_i,g_j,g_k)$. One can check that this is also equivalent to 
$\bigvee_{0<i<n}R(g_0,g_i,g_{i+1})$. \\[2mm]
\indent 
A {\it c-homomorphism} from a cyclically ordered group $(G,R)$ to a cyclically ordered group $(G',R')$ 
is a group homomorphism $f$ such that for 
every $g_1$, $g_2$, $g_3$ in $G$, if $R(g_1,g_2,g_3)$ holds and $f(g_1)\neq f(g_2)\neq 
f(g_3)\neq f(g_1)$, then $R'(f(g_1),f(g_2),f(g_3))$ holds (e.\ g.\ \cite{JPR}).\\[2mm]
\indent A linearly ordered group is cyclically ordered by the relation given by: $R(g_1,g_2,g_3)$ if, 
and only if, either $g_1<g_2<g_3$ or $g_2<g_3<g_1$ or $g_3<g_1<g_2$ (e.\ g.\ \cite{Sw}, \cite{BS18}). \\
\indent 
The cyclic order defined on a linearly ordered group as above is called a 
{\it linear cyclic order}. One also says that the group is {\it linearly cyclically ordered}. 
A cyclically ordered group which is not linearly ordered will be called a {\it nonlinear} 
cyclically ordered group. \\[2mm]
\indent A proper subgroup $H$ of a cyclically ordered group is said to be {\it c-convex} if it 
does not contain a non unit element of order $2$ and it satisfies:
$$\forall h \in H\; \forall g\in G\;(R(h^{-1},e,h) \;\&\; R(e,g,h))\Rightarrow g\in H \mbox{ (\cite{JPR})}.$$
\indent For every proper normal c-convex subgroup $H$ of a cyclically ordered group $G$, one 
can define canonically a cyclic order on $G/H$ such that the canonical epimorphism is 
a c-homomorphism. \\[2mm]
\indent 
If $(G,R)$ is a nonlinear cyclically ordered group, then 
it contains a largest c-convex proper subgroup $l(G)$ 
(\cite[Corollaries 3.6 and 4.7, Lemmas 4.2 and 4.6]{JPR}, \cite{ZP}). 
It is a normal subgroup of $G$. The restriction to $l(G)$ of the cyclic order is a linear cyclic order. 
The subgroup $l(G)$ is called the {\it linear part} of $(G,R)$. 
We denote by $<$ the linear order induced by $R$ on $l(G)$. 
If the cyclic order is linear, then we let $l(G)=G$\\
\indent The cyclically ordered group $G/l(G)$ embeds in a unique way in $\UU$ (\cite[Lemma 5.1, Theorem 5.3]{JPR}). 
We denote by $U(G)$ the image of 
$G/l(G)$ in $\UU$, and for $g\in G$, $U(g)$ denotes the image of $g\!\cdot\!l(G)$. \\[2mm]
\indent An abelian group $G$ is {\it $n$-divisible}, for some positive integer $n$, 
if for each $g\in G$ there is some $h\in G$ satisfying $nh=g$, this $h$ is not always unique. 
The group $G$ is {\it divisible} if it is $n$-divisible for each integer $n$. \\
\indent An abelian cyclically ordered group $(G,R)$ which is $n$-divisible and 
contains a subgroup isomorphic to the group of $n$-th roots of $1$ in the field of complex number 
is said to be {\it c-$n$-divisible}. 
It is called {\it c-divisible} if it is c-$n$-divisible for every positive integer $n$. This equivalent to 
being divisible and containing 
a subgroup c-isomorphic to the group $T(\UU)$ of torsion elements of $\UU$. 
\begin{Proposition}\label{sup}(\cite[Proposition 2.26]{Le})\label{aprop27}
Let $n\in \NN\backslash \{0\}$. An abelian cyclically ordered group $(G,R)$ is c-$n$-divisible  
if, and only if, $\mbox{uw}(G)$ is $n$-divisible.
\end{Proposition}
\indent We say that an abelian cyclically ordered group
$(G,R)$ is {\it discrete} if there exists $g\neq 0$ such that for all $h$ in $G$ we have $\neg R(0,h,g)$. This element 
will be denoted by $1_G$. We say that $(G,R)$ is {\it dense} if it is not discrete.  
\subsection{Outline of the content} 
Section \ref{In} is dedicated to cyclically ordered groups reminders and to properties that we will use 
in the former sections.  \\[2mm] 
\indent One easily sees that the theory of divisible abelian cyclically ordered groups is not complete, 
the cyclic order does not need to be linear and the group can have torsion elements 
(\cite[Corollary 6.13]{GiLL}). In Section \ref{sec3} we prove the 
following theorem.  
\begin{thm}\label{thm31}(Lucas)
The theory TcD of c-divisible abelian cyclically ordered groups: 
\begin{enumerate}
\item is complete and model-complete 
\item
has the amalgamation property
\item
admits the quantifiers elimination
\item
has a prime model $T(\mathbb{U})$
\item
is the model completion of the theory of abelian cyclically ordered groups 
(we prove that each abelian cyclically ordered group admits a c-divisible hull). 
\end{enumerate}
\end{thm}
\indent In Section \ref{n4} we get a characterization of elementary equivalence of 
divisible abelian cyclically ordered groups by mean of the following family of 
c-convex subgroups. Let $(G,R)$ be a divisible abelian cyclically ordered group. 
For every prime $p$ such that $G$ is not c-$p$-divisible, 
we let $H_p$ be the greatest $p$-divisible convex subgroup of the ordered group $l(G)$. Overwise, 
we let $H_p=G$. 
\begin{Proposition}\label{n128}(Lucas) Let $(G,R)$ be a divisible nonlinear abelian 
cyclically ordered group, and 
$p$ be a prime such that the cyclically ordered quotient group $G/H_p$ is discrete. 
Then $H_p\subsetneq l(G)$ and for every $r\in\NN\backslash\{0\}$ 
there exists an integer $k\in\{1,\dots, p^r-1\}$ which satisfies: for every $g\in l(G)$, $h\in G$ such that 
$g+H_p=1 _{G/H_p}$ and $g=p^rh$, we have
$\displaystyle{U(h)=e^{\frac{2ik\pi}{p ^r}}}$. The integer $k$ is determined by the first-order theory of $(G,R)$. 
\end{Proposition}
\indent The integer $k$ defined above is denoted $f_{G,p}(r)$. \\[2mm]
\indent Let $(G,R)$ be a divisible abelian cyclically ordered group. 
We denote by $\mathcal{CD}(G)$ the family $\{\{0_G\},\; G,\;H_p\;:\; p \;{\rm prime}\}$. It is equipped with 
the linear preorder relation $\subseteq$.
\begin{thm}\label{thm221} Let $(G,R)$ and $(G',R')$ be divisible nonlinear abelian cyclically ordered groups.  
\begin{itemize}
\item If, for every prime $p$, $G/H_p$ is dense, then  $(G,R)\equiv (G',R')$ if, 
and only if, the mapping $H_p\mapsto H_p'$ induces an isomorphism from $\mathcal{CD}(G)$ onto 
$\mathcal{CD}(G')$. 
\item 
If, for some prime $p$, $G/H_p$ is discrete, then  $(G,R)\equiv (G',R')$ if, and only if, the following holds: \\ 
the mapping $H_p\mapsto H_p'$ induces an isomorphism from $\mathcal{CD}(G)$ onto 
$\mathcal{CD}(G')$, \\ 
for every prime $p$, $G/H_p$ is discrete if, and only if, $G'/H_p'$ is discrete, 
and if this holds, then the mappings $f_{G,p}$ and $f_{G',p}$ are equal. \end{itemize}
\end{thm}
\indent Following Lucas, the proof of this theorem  is based on the works of  Schmitt (\cite{Sc1}, \cite{Sc2}).\\
\indent Then we describe the cyclic orders on the additive group $\QQ$ (Proposition \ref{prop418}). 
This gives rise to a description of the functions $f_{G,p}$ (Proposition \ref{lm241}) and the following theorem. 
\begin{thm}\label{uncount} Let $\Gamma$ be a divisible abelian group. 
\begin{enumerate}
\item 
It is cyclically orderable if, and only if,
its torsion subgroup $T(\Gamma)$ is isomorphic to a divisible subgroup of $T(\UU)$. 
\item If $\Gamma$ is isomorphic to a divisible subgroup of $T(\UU)$, then it admits uncountably many cyclic orders, 
but they are pairwise c-isomorphic. 
\item If $\Gamma$ contains non-torsion elements, and its torsion subgroup is not trivial, 
then it admits uncountably many pairwise non-c-isomorphic cyclic orders. If $T(\Gamma)$ is isomorphic to 
$T(\UU)$, then all of them are pairwise elementarily equivalent. 
\item If $\Gamma$ is torsion-free, then there are uncountably many pairwise non elementarily equivalent cyclic orders 
on $\Gamma$. 
\end{enumerate}
\end{thm}
\indent Next, we prove that there is a one-to-one correspondence between the set of partitions of the set 
of prime numbers, equipped with an arbitrary linear order, and the set of families $\mathcal{CD}(G)$, 
where $G$ is a divisible abelian cyclically ordered group. 
\\[2mm]
\indent Finally, the minimality is studied in Section \ref{sec5}. We start with definitions. 
\begin{defis} Let $(G,R)$ be a cyclically ordered group. 
If $I$ is a subset of $G$, then $I$ is an {\it open interval} if there are $g\neq g'$ in $G$ such that 
$I=\{ h\in G\; :\; R(g,h,g')\}$. We denote this open interval by $I(g,g')$.\\
A subset $J$ is {\it c-convex} if either it has only one point, or, 
for each $g\neq g'$ in $J$, either $I(g,g')\subseteq J$ or $I(g',g)\subseteq J$. 
\end{defis}                      
\begin{Remark}\label{rk51}\begin{itemize}
\item A c-convex proper subgroup of a cyclically ordered group is a proper subgroup which is a c-convex subset. 
\item A c-convex subset is not 
necessarily a singleton or an open interval: the linear part of a cyclically ordered group 
is a c-convex subset and if it is nontrivial, then it is not an 
open interval. Indeed, if $g\notin l(G)$, then any open interval bounded by $g$ contains an element $g'$ such that 
$U(g')=U(g)$, hence $g'\notin l(G)$. If $g\in l(G)$, then $g$ doesn't belong to the open intervals bounded by itself. 
\item The bounded open intervals of $(l(G),<)$ are open intervals of $(G,R)$, since for $h$ and $g<g'$ in $l(G)$ we have 
$g<h<g'\Leftrightarrow R(g,h,g')$. 
\end{itemize} 
\end{Remark}
\begin{defis} Let $(G,R)$ be an infinite cyclically ordered group. 
We say that $(G,R)$ is:
\begin{enumerate}
\item
{\it cyclically minimal} if each definable subset of $G$ is a finite union of singletons and open intervals.
\item
{\it strongly cyclically minimal} if for each $(G',R')$ which is elementary e\-qui\-va\-lent to $(G,R)$, 
$(G',R')$ is cyclically minimal.
\item
{\it weakly cyclically minimal} if each definable subset of $G$ is a finite union of c-convex subsets.
\end{enumerate}
\end{defis}
\indent 
These notions correspond respectively to the notions of o-minimal, strongly o-minimal and weakly 
o-minimal in the linearly ordered case. 
In the linearly ordered case the three notions coincide and are satisfied in $(G,R)$ if, 
and only if, $(G,R)$ is abelian and divisible. For the equivalence between o-minimal and strongly o-minimal see 
\cite{KnPiSt} and \cite{PiSt}. The equivalence with weakly o-minimal follows from \cite{Di} corollary 1.3. \\
\indent
Clearly if $(G,R)$ is strongly cyclically minimal, then it is cyclically minimal, and if it is cyclically minimal, then it is weakly 
cyclically minimal. \\
\indent Note that in a finite structure every subset is a finite union of singletons. 
\begin{thm}\label{th513}(\cite[Theorem 5.1]{MSt}) 
A cyclically ordered group is cyclically minimal if, and only if, it is abelian and c-divisible.
\end{thm}
\begin{Corollary}\label{cor15}(Lucas) A cyclically ordered group is strongly cyclically minimal 
if, and only if, it is cyclically minimal.
\end{Corollary}
\begin{proof} Indeed, the theory of c-divisible abelian cyclically ordered groups is complete. 
\end{proof}
\indent The characterization of weakly cyclically minimal cyclically ordered groups requires to define the 
construction of cyclically ordered groups by mean of the lexicographic product (\cite{Sw}). \\
\indent 
Let $(G,R)$ be a cyclically ordered group and $\Gamma$ be a linearly ordered group.  
One can define a cyclic order on 
$G\times \Gamma$ as follows:  $R'((g_1,x_1),(g_2,x_2),(g_3,x_3))$ holds if, and only if, 
\begin{itemize}
\item either $g_1=g_2=g_3$ and $x_{\sigma(1)}<x_{\sigma(2)}<x_{\sigma(3)}$ for some 
$\sigma$ in the alternating group $A_3$ of degree $3$ 
\item or $g_{\sigma(1)}=g_{\sigma(2)}\neq g_{\sigma(3)}$ and 
$x_{\sigma(1)}<x_{\sigma(2)}$ for some $\sigma\in A_3$ 
\item or $R(g_1,g_2,g_3)$. 
\end{itemize}
\begin{defi}\label{def124}
Let $(G,R)$ be a cyclically ordered group and $\Gamma$ be a linearly ordered group. 
We let $G\overrightarrow{\times} \Gamma$ denote the cyclically ordered group defined 
above. It is called the 
{\it lexicographic product} of $(G,R)$ and $(\Gamma,\leq)$. 
\end{defi}
\begin{thm}\label{th514}(Lucas) 
A cyclically ordered group $(G,R)$ is weakly cyclically minimal if, and only if, 
either $(G,R)$ is abelian and c-divisible, or $l(G)$ is infinite, abelian and divisible and 
there exist a positive integer $n$ such that $G\simeq T(\UU)_n\overrightarrow{\times}l(G)$, where $T(\UU)_n$ is the 
subgroup of $n$-th roots of $1$ in $T(\UU)$. 
\end{thm}
\indent The author thanks Viktor Verbosvskiy and Melissa Nalbandiyan for their comments. 
\section{Properties of cyclically ordered groups.}\label{In}
\indent Before to list properties of cyclically ordered groups, we point out that 
in \cite{Z82} an equivalent definition of cyclically ordered groups by mean of a mapping 
from $G\times G\times G$ to $\{-1,0,1\}$ appears. Later, the conditions of this definition are made simpler and 
the cyclic orderings are called circular orderings.  
A \textit{compatible circular order} on $G$ is a mapping $c$ from $G\times G\times G$ to $\{-1,0,1\}$ 
which is a cocycle, that is for all $g_1$, $g_2$, $g_3$, $g_4$ in $G$ we have 
$c(g_2, g_3, g_4)-c(g_1, g_3, g_4)+c(g_1, g_2, g_4)-c(g_1, g_2, g_3)=0$ (e.\ g.\ \cite{J89}). Now, one can check 
that $R$ is a compatible cyclic order on $G$ if, and only if, the mapping $c$ defined by 
$c(g_1,g_2,g_3)=1$ if $R(g_1,g_2,g_3)$ holds, $c(g_1,g_2,g_3)=-1$ if $R(g_3,g_2,g_1)$ holds, and $c(g_1,g_2,g_3)=0$ 
if $g_1=g_2$ or $g_2=g_3$ or $g_3=g_1$, is a compatible circular order on $G$. The reader can find a similar proof 
in \cite[Lemma 2.3]{BS18}. 
Many papers about circularly ordered groups do not refer to the classical ones about cyclically ordered groups 
(except \cite{CMR18} for example), and several results are re-proved. 
\subsection{Wound-round cyclically ordered groups.}\label{subsect21}
If $(\Gamma,\leq)$ is a linearly ordered  group and $z\in \Gamma$, $z>e$,  
is a central and cofinal element of $\Gamma$, 
then the quotient group $\Gamma/\langle z\rangle$ can be cyclically ordered by setting 
$R(x_1\!\cdot\!\langle z\rangle,x_2\!\cdot\!\langle z\rangle,x_3\!\cdot\!\langle z\rangle)$ if, and only if, 
there are $x_1',\;x_2',\;x_3'$  
such that \\
\indent $x_1\!\cdot\!\langle z\rangle=x_1'\!\cdot\!\langle z\rangle$, $x_2\!\cdot\!\langle z\rangle=
x_2'\!\cdot\!\langle z\rangle$, 
$x_3\!\cdot\!\langle z\rangle=x_3'\!\cdot\!\langle z\rangle$ and \\ 
\indent either 
 $e\leq x_1'<x_2'<x_3'<z $ or $e\leq x_2'<x_3'<x_1'<z$ or $e\leq x_3'<x_1'<x_2'<z$. \\ 
\indent With this cyclic order, $\Gamma/\langle z\rangle$ is called the 
{\it wound-round} cyclically ordered 
group associated to $\Gamma$ and $z$. \\
\indent If $\Gamma=\ZZ$ and $n$ is a positive integer, then we get a natural cyclic order on the group 
$\ZZ/n\ZZ$. This cyclically ordered group is c-isomorphic to the the subgroup $T(\UU)_n$ of $n$-th roots of 
$1$ in the cyclically ordered group $T(\UU)$. \\
\indent Assume that $\Gamma$ is a lexicographic product of linearly ordered groups $\ZZ\overrightarrow{\times}\Gamma_1$, 
and $z=(1,e)$. For any $x\in \Gamma_1$ and $n\in\ZZ$, we have $(0,e)<(n,x)<(1,e)=z \Leftrightarrow 
n=0$ and $x>e$ or $n=1$ and $x<e$. So, one can check that the cyclically ordered group $\Gamma/\langle z\rangle$ 
is linearly cyclically ordered, and isomorphic to the ordered group $\Gamma_1$. \\
\indent A  theorem of Rieger (\cite{R}, \cite[Theorem 21 p.\ 64]{F}) states that for every  
cyclically ordered group $(G,R)$ 
there exists a linearly ordered  group $({\rm uw}(G),\leq_R)$ and a positive element $z_G\in {\rm uw}(G)$ 
which is central and cofinal such that $(G,R)$ is c-isomorphic to the cyclically ordered  group 
${\rm uw}(G)/\langle z_G\rangle$. The structure $({\rm uw}(G),\leq_R,z_G)$ is uniquely defined, up to isomorphism. 
The group $({\rm uw}(G),\leq_R)$ is called the {\it unwound} of $(G,R)$. \\
\indent For example, the unwound of $\UU$ is $\RR$, the unwound of $(G,R)$ is isomorphic to $\ZZ$ if, and only if, 
$G$ is finite. \\
\indent For further purposes, we give the definition of the order relation and of the group operation on 
${\rm uw}(G)$. \\
\indent 
The underlying set of ${\rm uw}(G)$ is $\mathbb{Z}\times G$. \\
\indent The relation $\leq_R$ is defined by 
$(n,g)\leq_R (n',g')$ if, and only if, either $n<n'$ or $[n=n'$ and $(g=e$ or $ R(e,g,g')$ or $g=g')]$. \\
\indent Turning to the group operation, $(n,g)\! \cdot \! (n',g')$ is equal to 
$(n+n',gg')$ if either $g=e$ or $g'=e$ or $R(e,g,gg')$ holds, and it is equal to 
$(n+n'+1,gg')$ if $g\neq e$ and $gg'=e$ or $R(e,gg',g)$ holds. \\
\setlength{\unitlength}{1mm}
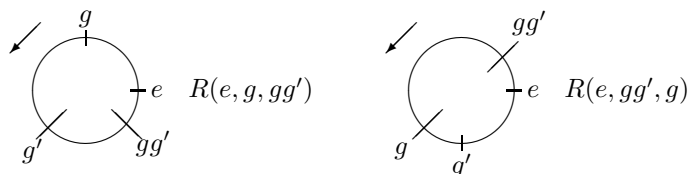
\begin{picture}(200,22)(0,0)
\put(20,10){\circle{14}}
\put(20,16){\line(0,1){2}}
\put(26,10){\line(1,0){2}}
\put(23.5,7.5){\line(1,-1){4}}
\put(17.5,7.5){\line(-1,-1){4}}
\put(20,19.5){\makebox(0,0){$g$}}
\put(29.5,10){\makebox(0,0){$e$}}
\put(29,3){\makebox(0,0){$gg'$}}
\put(13,2){\makebox(0,0){$g'$}}
\put(14,19){\vector(-1,-1){4}}
\put(42,10){\makebox(0,0){$R(e,g,gg')$}}
\put(70,10){\circle{14}}
\put(73.5,12.5){\line(1,1){4}}
\put(76,10){\line(1,0){2}}
\put(70,2){\line(0,1){2}}
\put(67.5,7.5){\line(-1,-1){4}}
\put(79.5,10){\makebox(0,0){$e$}}
\put(79,19){\makebox(0,0){$gg'$}}
\put(70,0){\makebox(0,0){$g'$}}
\put(62,2){\makebox(0,0){$g$}}
\put(64,19){\vector(-1,-1){4}}
\put(92,10){\makebox(0,0){$R(e,gg',g)$}}
\end{picture}
\indent Furthermore, $z_G=(1,e)$. \\
\indent For $x\in {\rm uw}(G)$, we denote by $\bar{x}$ its 
image in ${\rm uw}(G)/\langle z_G\rangle\simeq G$. \\
\indent 
There is a set embedding of $G$ in $\{0\}\times G$ such that for any $g,g'$ in $G \backslash \{e\}$, 
$({\rm uw}(G),\leq_R)$ satisfies $(0,g)<_R (0,g')$ if, and only if, $(G,R)$ satisfies $R(e,g,g')$.\\[2mm] 
\indent We end this subsection with a lemma which will be useful in the proof of Proposition \ref{prop211} 
and in subsection \ref{nsec33}. 
For every abelian group $\Gamma$ and every prime $p$, we denote by $[p]\Gamma$ the number of classes 
modulo $p$ in $\Gamma$ (without distinguishing between infinites). 
Note that $\Gamma$ being divisible is equivalent to $[p]\Gamma=1$ for every prime $p$. 
\begin{Lemma}(\cite[Lemma 4.16]{Le})\label{pclass} 
Let $(G,R)$ be an abelian cyclically ordered group. If $G$ contains a $p$-torsion element, 
then $[p]{\rm uw}(G)=[p]G$. Otherwise, $[p]{\rm uw}(G)=p[p]G$. 
\end{Lemma}
\subsection{Linear part of a cyclically ordered group.}\label{subsec21}
Recall that if $(G,R)$ is a nonlinear cyclically ordered group, then its linear part $l(G)$ is its largest c-convex 
proper subgroup and if the cyclic order is linear, then we let $l(G)=G$.\\[2mm] 
\indent 
Let $(G,R)$ be a cyclically ordered group. 
An element $g$ is said to be {\it positive} if $R(e,g,g^2)$ holds. This condition is equivalent to 
$R(e,g,g^{-1})$. 
We denote by $P$ the union of $\{ e\}$ 
and of the set of positive elements of $(G,R)$ (\cite{Z82}). We say that $P$ is the {\it positive cone} 
of $(G,R)$. Both of $P$ and $P^{-1}$ are c-convex subsets of $G$, and $g\in P\cup P^{-1}$ if, and only if, 
either $g=e$ or $g^2\neq e$. 
The set $P\cap l(G)$ is the positive cone of the linear order on $l(G)$ (\cite[p.\ 547]{ZP}). \\
\indent 
One can prove that the map from $G$ in ${\rm uw}(G)$ defined by $f(g)=(0,g)$ if $ g\in P$ and 
$f(g)=(-1,g)$ if not, when restricted to a c-convex proper subgroup $H$, is a group homomorphism and an order 
isomorphism between $H$ and $f(H)$. Furthermore, $f(H)$ is a convex subgroup of ${\rm uw}(G)$. 
\begin{nota}\label{nota27} Let $(G,R)$ be a cyclically ordered group. 
We set $G_{uw}={\rm uw}(G)$, and if $H$ is a c-convex proper subgroup of $G$, then  set $H_{uw}=f(H)$. 
\end{nota}
\indent 
The epimorphism from ${\rm uw}(G)$ to ${\rm uw}(G)/\langle z_G\rangle\simeq G$,  
when restricted to $H_{uw}$, is the inverse of the isomorphism $f$ from $H$ onto $H_{uw}$. 
It follows that the  set of c-convex proper subgroups of $(G,R)$ is linearly ordered 
by inclusion and when $R$ is nonlinear we have an order isomorphism between the 
inclusion-ordered set of c-convex proper subgroups of $(G,R)$ and the inclusion-ordered set of proper 
convex subgroups of $({\rm uw}(G),\leq_R)$. 
Therefore, the set of c-convex proper subgroups is closed under arbitrary
unions and intersections. For each $g\in G$ there is a smallest c-convex proper subgroup of 
$(G,R)$ containing $g$. 
Note that $l(G)_{uw}$ is the largest proper convex subgroup of $({\rm uw}(G),\leq_R))$, since  
it is the largest convex subgroup which doesn't contain the cofinal element $z_G$. 
Hence ${\rm uw}(G)/l(G)_{uw}$ embeds in $\RR$. \\
\indent Recall that for every proper normal c-convex subgroup $H$ of a cyclically ordered group $G$, one 
can define canonically a cyclic order on $G/H$ such that the canonical epimorphism is 
a c-homomorphism. The linear part of $G/H$ is isomorphic to $l(G)/H$, and its unwound is 
isomorphic to ${\rm uw}(G)/H_{uw}$. \\[2mm]
\indent 
A cyclically ordered  group $(G,R)$ is said to be {\it c-archimedean} if for every $g$ 
and every $h$ there is an integer $n$ such that $R(e,g^n,h)$ is not satisfied (\cite{Sw59}, \cite[Section 3.3]{BS18}).
\begin{Fact}\label{fact17} \begin{enumerate}
\item $(G,R)$ is c-archimedean if, and only if, it can be embedded in $\UU$. This in turn holds 
if, and only if, $R$ is not linear and $G$ has no nontrivial c-convex proper subgroup (\cite[p.\ 162]{Sw}).
\item $(G,R)$ is nonlinearly cyclically ordered and is c-archimedean if, and only if, its unwound $({\rm uw}(G),\leq_R)$ is 
ar\-chi\-me\-dean. 
This follows from (1) and the order isomorphism between the inclusion-ordered set of 
c-convex proper subgroups of $(G,R)$ and the inclusion-ordered set of convex subgroups of 
$({\rm uw}(G),\leq_R)$, in the nonlinear case.\end{enumerate} \end{Fact}
\indent 
Let $g$, $h$ be elements of a cyclically ordered  group $(G,R)$. 
If $1\neq U(g)\neq U(h)\neq 1$, then $R(e,g,h)$ holds 
in $(G,R)$ if, and only if, $R(1,U(g),$ $U(h))$ holds in $U(G)$. 
Now, if $U(g)=1\neq U(h)$ (so $g\in l(G)$), 
then $R(e,g,h)$ holds if, and only if, $g>e$ in $l(G)$, and 
$R(e,h,g)$ holds if, and only if, $g<e$ in $l(G)$ (\cite[Remark 5.6]{GiLL}). 
\subsection{Dense, discrete, divisible abelian cyclically ordered groups.} 
One can prove that an abelian cyclically ordered group $(G,R)$ is dense if, and only if, 
$$\forall (g_1,g_2) \in G\times G\; g_1 \neq g_2\Rightarrow (\exists h\in G\; R(g_1,h,g_2)).$$
\indent 
Let $(G,R)$ and $(G',R')$ be two abelian cyclically ordered groups such that $(G',R')$ is a substructure of $(G,R)$.
We say that $(G',R')$ is {\it dense in} $(G,R)$ if 
$$\forall (g_1,g_2) \in G\times G\;  (\exists h \in G \;
R(g_1,h,g_2)) \Rightarrow (\exists h' \in G'\; R(g_1,h',g_2)).$$
\indent The group $\mathbb{U}$ is dense, a subgroup of $\UU$ is dense if, and only if, it is infinite 
(this is also equivalent to saying that it is dense in $\mathbb{U}$). In particular, each divisible subgroup of
$\mathbb{U}$ is dense. \\
\indent 
For every positive integer $n$, $T(\UU)_n$ is discrete, and $1_{T(\UU)_n}=e^{2i\pi/n}$. 
We give the proof of the following 
lemma for completeness. 
\begin{Lemma}\label{lm12} The abelian cyclically ordered group 
$(G,R)$ is discrete if, and only if, its unwound $({\rm uw}(G), \leq_R)$ is a discrete linearly ordered group. 
More generally, let $H$ be a c-convex subgroup of $(G,R)$. Then the cyclically ordered group $G/H$ is discrete 
if, and only if, the linearly ordered group ${\rm uw}(G)/H_{uw}$ is discrete. 
\end{Lemma}
\begin{proof} If $l(G)$ is nontrivial, then one can see that $(G,R)$ is discrete if, and only if, 
$l(G)$ is a discrete linearly ordered group. If this holds, then $1_G$ is the smallest positive element of $l(G)$. 
Since $l(G)$ embeds in a convex subgroup of $({\rm uw}(G),\leq_R)$, it 
follows that $(G,R)$ is discrete if, and only if, $({\rm uw}(G),\leq_R)$ is discrete. If $l(G)$ is trivial, then, 
by Fact \ref{fact17}, $(G,R)$ embeds in $\UU$. Now, the unwound of $\UU$ is $\RR$ and the unwound of $(G,R)$ 
is isomorphic to $\ZZ$ if, and only if, 
$G$ is finite. It follows 
that $(G,R)$ is discrete if, and only if, its unwound is a discrete linearly ordered group. \\ 
\indent Now, let $H$ be a c-convex subgroup of $G$. We saw before Fact \ref{fact17} that the unwound of 
$G/H$ is ${\rm uw}(G)/H_{uw}$. Therefore $G/H$ is discrete 
if, and only if, ${\rm uw}(G)/H_{uw}$ is discrete. 
\hfill \end{proof}
\begin{Remark}\label{r}\begin{itemize}
\item A divisible abelian cyclically ordered group is not necessarily dense, as shows the following example of Lucas. 
Let $G$ be the wound-round $(\mathbb{Q} \overrightarrow{\times} \mathbb{Z}) /\langle(\alpha,1)\rangle$ with $\alpha$ positive, 
where $\overrightarrow{\times}$ denotes the lexicographic product. 
The ordered group $\mathbb{Q} \overrightarrow{\times} \mathbb{Z}$ is discrete with smallest positive 
element $(0,1)$. By Lemma \ref{lm12}, $(G,R)$ is discrete. Let $n\in \NN\backslash \{0\}$ and 
$(x,m) \in \QQ \overrightarrow{\times}\ZZ$. We have: 
$$\displaystyle{n\left( \frac{x+(n-m)\alpha}{n},1\right)=(x+(n-m)\alpha,n)=(x,m)+(n-m)(\alpha,1)}.$$ 
\indent Hence $(x,m)$ is $n$-divisible modulo $(\alpha,1)$. So $(G,R)$ is divisible. 
\item It can happen that $(G,R)$ is divisible and ${\rm uw}(G)$ is not. This holds in the previous example. 
One can also take $\Gamma=\mathbb{Q}+\mathbb{Z} \pi $ in $\mathbb{R}$ and 
$G=\Gamma/\langle\pi \rangle$.
\item If $(G,R)$ is divisible, then so is $U(G)$ (since $h^n=g\Rightarrow U(h)^n=U(g)$). Therefore, if 
$(G,R)$ c-archimedean and divisible, then $G\simeq U(G)$ is either dense or trivial. Hence it is not discrete. 
\end{itemize}
\end{Remark}
\begin{Lemma}\label{cdiv}(Lucas) Let $n\in \NN\backslash \{0,1\}$. An abelian cyclically ordered 
group $(G,R)$ has an element of finite order $n$ if, and only if, $z_G$ is $n$-divisible within ${\rm uw}(G)$.
\end{Lemma}
\begin{proof} Assume that there is $x\in {\rm uw}(G)$ such that $nx=z_G$. Then in $G$ we have $n\bar{x}=0_G$. 
Now, for every $n'\in \{1,\dots,n-1\}$ we have $0_{{\rm uw}(G)}<n'x<z_G$, hence $n'\bar{x}\neq 0_G$. 
Conversely, assume that there is $x\in {\rm uw}(G)$ such the order of $\bar{x}$ is $n$, i.e.\ 
$n\bar{x}=0_G$ and, for every $n'\in \{1,\dots,n-1\}$, $n'\bar{x}\neq 0_G$. So there is $k\in \ZZ$ such that 
$nx=kz_G$. Let $d$ be the gcd of $n$ and $k$, and say, $n=dn'$, $k=dk'$. Therefore, 
we have $n'x=k'z_G$, so $n'\bar{x}=0_G$. By minimality of $n$, we have $n'=n$, so $d=1$. Therefore, 
there exist $u$, $v$ in $\ZZ$ such that $uk+vn=1$. 
Hence $z_G=ukz_G+vnz_G=unx+vnz_G=n(ux+vz_G)$. Consequently, $z_G$ is $n$ divisible. 
\end{proof}
\begin{Lemma}\label{finiteodense}(Lucas) If a divisible abelian cyclically ordered group has an 
element whose order is finite, then it is dense.
\end{Lemma}
\begin{proof} 
Let $(G,R)$ a divisible abelian cyclically ordered group and $p$ be a prime such that $G$ 
contains an element $g$ of torsion $p$. If $G$ is c-archimedean, then it is infinite so it is dense. 
Let $G>0$ in $l(G)$ and $h$ such that $ph=g$. 
By \cite[Proposition 3]{Zh}, $T(G)$ c-embeds in $T(\UU)$, hence it contains a subgroup isomorphic 
to $T(U)_p$. Therefore, by multiplying $h$ by a $p$-torsion element we can assume that $h\in l(G)$. 
So we have $0<h<g$ in $l(G)$. Hence $(G,R)$ is dense. 
\end{proof} 
\begin{Corollary}\label{cor210} A discrete divisible abelian cyclically ordered group is torsion-free. 
\end{Corollary} 
\subsection{C-regular cyclically ordered groups.}\label{n3}
We start with the definition of regular linear groups. They were introduced in \cite{RoZ}, \cite{Za}, but 
later equivalent definitions were given (\cite{Co62}, \cite{DP11}). Let $n$ be a positive integer. 
An abelian linearly ordered group is  {\it $n$-regular} if it satisfies 
$\forall x_1,\dots,\forall x_n \; 0<x_1<\dots <x_n\Rightarrow \exists y\; (x_1\leq ny \leq x_n).$\\  
\indent A useful criterion is the following. A linearly ordered abelian group 
$(\Gamma,\leq)$ is  $n$-regular if, and only if, for each proper convex subgroup $C$, the quotient group 
$\Gamma/C$ is  $n$-divisible (\cite{Co62}). The group  
$(\Gamma,\leq)$ being  $n$-regular is also equivalent to: 
for each prime number $p$ dividing $n$, $(\Gamma,\leq)$ is $p$-regular (\cite{RoZ}, \cite{Za}). 
We say that the group 
$(\Gamma,\leq)$ is {\it regular} if it is  $n$-regular for each $n$. If $(\Gamma,\leq)$
is divisible or archimedean, then it is regular. \\
If $(\mbox{uw}(G),\leq_R)$ is regular, then we say that $(G,R)$ is 
{\it c-regular}. This is also equivalent to: for every c-convex proper subgroup $C$, 
$G/C$ is divisible and to: either $G$ is c-archimedean or $l(G)$ is regular and $U(G)$ is divisible and 
contains a subgroup isomorphix to $T(\UU)$ (\cite[Theorem 3.5]{Le}). \\[2mm]
\indent Let $(G,R)$ be a divisible abelian cyclically ordered group. 
If $(G,R)$ is linearly cyclically ordered, then it is a regular linearly 
ordered group, since it is divisible. By \cite[Corollary 3.15]{Le}, if $(G,R)$ is nontrivial and linearly ordered, 
then it is not c-regular. It follows that $(G,R)$ is linearly ordered if, and only if, it is regular (as linearly 
ordered group) and not c-regular. This is first-order definable. If this holds, then it is 
elementarily e\-qui\-va\-lent to the linearly ordered group $\QQ$ (\cite[Theorem 4.3.2]{Ro56}). \\
\indent A dense divisible nonlinear abelian cyclically ordered group 
is not necessarily c-regular, as shows the following example of Lucas. Let $(G,R)$ be the wound-round 
$((\mathbb{Q}+\pi \mathbb{Z})\overrightarrow{\times} \mathbb{Q})/\langle (\pi,0)\rangle$. It is divisible. 
We have ${\rm uw}(G)=(\mathbb{Q}+\pi \mathbb{Z})\overrightarrow{\times} \mathbb{Q}$. 
Now, $((\mathbb{Q}+\pi \mathbb{Z})\overrightarrow{\times} \mathbb{Q})/(\{0\}\overrightarrow{\times}\QQ)$ 
is isomorphic to $\mathbb{Q}+\pi \mathbb{Z}$, so it is not divisible. Hence 
${\rm uw}(G)$ is not regular, therefore, $(G,R)$ is not c-regular. \\
\indent Note that a c-divisible cyclically ordered group is dense and c-regular, since ${\rm uw}(G)$ is divisible. \\
\indent Now, the next theorem shows that the theory of c-divisible abelian 
cyclically ordered groups is complete. 
\begin{Theorem}\label{thm215}
Two dense c-regular divisible abelian cyclically ordered groups are elementarily e\-qui\-va\-lent if, 
and only if, they have c-isomorphic torsion subgroups. 
\end{Theorem}
\begin{proof} This is a consequence of \cite[Theorem 1.9]{Le}, which states that any two dense c-regular 
abelian cyclically ordered groups are elementarily e\-qui\-va\-lent if, and only if, they have c-isomorphic 
torsion subgroups and, for every prime $p$, the same number of classes of congruence modulo $p$. 
\end{proof}
\indent Our Theorem \ref{thm221} involves dense and discrete divisible abelian cyclically ordered 
groups. However, the discrete case also follows from \cite{Le}. We turn to this case in the remainder of this subsection. 
\begin{Lemma}\label{discrete}
Let $(G,R)$ be a discrete abelian nonlinear cyclically ordered group. Then either $G$ is finite, 
or $U(1_G)=1$.
\end{Lemma}
\begin{proof} If $U(G)$ is infinite, then it is a dense cyclically ordered group. Therefore, 
if $l(G)=\{ 0_G\}$, then $(G,R)$ is discrete if, and only if, it is finite. 
If $l(G)$ is infinite, then $1_G\in l(G)$ (see proof of Lemma \ref{lm12}). Hence $U(1_G)=1$.  \end{proof}
\begin{Proposition}\label{prop211}(Lucas)
Each discrete divisible abelian cyclically ordered group is c-regular.
\end{Proposition}
\begin{proof} Let $(G,R)$ be a discrete divisible abelian cyclically ordered group and 
$n\in \NN\backslash \{0\}$. The cyclic order is not linear, since every linearly ordered divisible 
abelian group is dense. Since $G$ is torsion-free (Corollary \ref{cor210}), and divisible, by 
Lemma \ref{pclass} we have $[p]{\rm uw}(G)=p$. 
Since $(0,1_G)$ is the smallest positive element of 
$({\rm uw}(G),\leq_R)$, the classes $(0,0_G),(0,1_G),(0,2\!\cdot\! 1_G), \dots, (0,(p-1)\!\cdot\! 1_G)$ are pairwise 
distinct. Therefore, the classes of ${\rm uw}(G)$ modulo $p$ are the classes of 
 $(0,0_G),(0,1_G),(0,2\!\cdot\! 1_G), \dots, (0,(p-1)\!\cdot\! 1_G)$. Consequently,  
${\rm uw}(G)/\langle (0,1_G)\rangle$ is $p$-divisible. Since $\langle (0,1_G)\rangle$ is the smallest 
convex subgroup of $({\rm uw}(G),\leq_R)$, 
it follows that the quotient of ${\rm uw}(G)$ by any proper convex subgroup is $p$-divisible, and that 
$({\rm uw}(G),\leq_R)$ is $p$-regular. Consequently $(G,R)$ is c-regular. 
\end{proof}
\indent The classification of complete theories of discrete c-regular abelian cyclically ordered groups 
follows from \cite[Theorem 4.33 and Corollary 4.34]{Le}. 
If $(G,R)$ is a c-regular discrete abelian cyclically ordered group, then for every $p$ and $n$ there 
exists a unique $k$ such that the formula $\exists g\; R(0_G,g,2g,\dots,(p^n-1)g)\;\&\;p^ng=k1_G$ 
holds. The element $g$ in this formula is such that                                    
the $jg$'s where $1\leq j\leq p^n-1$ ``do not turn full the circle'', but $p^n g$ ``turns full the circle'', and it is 
equal to $1_G$. Let us denote this $k$ by $\varphi_{G,p}(n)$. \\
\indent For every c-regular discrete abelian cyclically ordered group
$(G,R)$ and every prime $p$ there exists a unique mapping $\varphi_p$ from $\NN\backslash \{0\}$ to $\{0,\dots,p-1\}$ 
such that for every $n\in \NN\backslash \{0\}$, 
$\varphi_{G,p}(n)=\varphi_p(1)+p\varphi_p(2)+\cdots+p^{n-1}\varphi_p(n)$ 
(\cite[Lemma 4.29]{Le}). The sequence of the $\varphi_p$'s is called 
the {\it cha\-rac\-te\-ris\-tic sequence} of $(G,R)$. \\
\indent For every sequence of $\varphi_p$'s there exists a 
discrete c-regular abelian cyclically ordered group $(G,R)$ such that this sequence is the characteristic sequence of 
$(G,R)$ (\cite[Proposition 4.36]{Le}). \\ 
\indent Finally, $(G,R)$ is divisible if, and only if, for 
every $p$ we have $\varphi_p(1)\neq 0$ (\cite[Proposition 4.36]{Le}). 
\begin{Theorem}\label{thm213}(\cite[Corollary 4.34]{Le}) 
Any two discrete divisible abelian cyclically ordered groups are e\-le\-men\-ta\-ri\-ly 
e\-qui\-va\-lent if, 
and only if, they have the same characteristic sequences. 
\end{Theorem}
\indent 
Since there are $2^{\aleph_0}$ distinct characteristic 
sequences of discrete divisible abelian cyclically ordered groups, there are $2^{\aleph_0}$ elementary classes of 
discrete divisible abelian cyclically ordered groups. This was originally proved by Lucas, 
using constructions of cyclic orders on $\QQ$. \\[2mm]
\indent Note that if $\varphi_p(1)\neq 0$, then $p$ does not divide $\varphi_{G,p}(n)$, hence 
$\varphi_{G,p}(n)$ has an inverse modulo $p^n$. 
We see that the function $f_{G,p}$ satisfies a similar property. In order to keep the consistency of the paper, 
we start by proving the existence of $f_{G,p}(n)$. 
\begin{Lemma}\label{rk240} Let $(G,R)$ be a dense divisible abelian cyclically ordered group. 
Assume that there is a prime $p$ such that $G/H_p$ is discrete. \begin{enumerate}
\item We have $H_p\subsetneq l(G)$, and for every $n\in\NN\backslash\{0\}$ 
there exists an integer $f_{G,p}(n)\in\{1,\dots, p^n-1\}$ which satisfies: for every $g\in l(G)$, $h\in G$ such that 
$g+H_p=1_{G/H_p}$ and $g=p^nh$, we have
$\displaystyle{U(h)=e^{\frac{2if_{G,p}(n)\pi}{p^n}}}$. 
\item For every positive integers $m$, $n$, 
$f_{G,p}(n)$ is the remainder of the euclidean division 
of $f_{G,p}(m+n)$ by $p^{n}$. It follows that $p$ and $f_{G,p}(n)$ are coprime. 
\end{enumerate} 
\end{Lemma}
\begin{proof} Note that if $G/H_p$ is discrete, then $H_p\neq G$, so $G$ is not c-$p$-divisible. In particular, $G$ 
does not contain $p$-torsion elements. 
Therefore, for every $g\in G$, and every positive integer $n$, there is a unique $h$ such that $g=p^nh$. \\
\indent (1) By the definition of $H_p$ we have 
$H_p\subseteq l(G)$. Now, $G/l(G)\simeq U(G)$, which is dense since it is divisible and nontrivial. 
Hence if $G/H_p$ is discrete, then $H_p\neq l(G)$ (if $(G,R)$ is discrete, then $l(G)$ is nontrivial 
(Remark \ref{r}) and $H_p=\{0_G\}$). \\
\indent Now, we assume that $G/H_p$ is discrete. Let $g\in G$ be such that $g+H_p=1_{G/H_p}$, 
and $h$ be the unique element of $G$ such that $g=p^nh$. Note that $1_{G/H_p}$ belongs to the 
positive cone of $G/H_p$, hence $g$ belongs to the positive cone of $G$. 
Since $H_p\subsetneq l(G)$, the cyclically 
ordered group $G/H_p$ is not c-archimedean, and its linear part is isomorphic to $l(G)/H_p$. Since  
$1_{G/H_p}$ belongs to its linear part, we have $g\in l(G)$, and $g>0_G$ in $l(G)$. 
Since $p^nh\in l(G)$, we have $U(p^nh)=1$. Hence $U(h)$ is a $p^n$-th root of $1$ in $\UU$. 
Assume that $h\in l(G)$. Since $H_p$ is a convex subgroup of $l(G)$ and $p^nh\notin H_p$, we have 
$h\notin H_p$. 
By properties of linearly ordered groups we have $0_G<h<g$.  
This contradicts the minimality of $g+H_p$. Hence $h\notin l(G)$, that is $U(h)\neq 1$. 
It follows that there exists 
$k\in \{1,\dots,p^n-1\}$ such that $\displaystyle{U(h)=e^{\frac{2ik\pi}{p^n}}}$. \\
\indent Let $g'\in G$ be such that $g'+H_p=1_{G/H_p}$, and $h'$ be the unique element of $G$ such 
that $g'=p^nh'$. Then $g-g'\in H_p$, and $p^n(h-h')=g-g'$. Since $H_p$ is $p$-divisible and $h-h'$ is the unique 
element of $G$ such that $p^n(h-h')=g-g'$, $h-h'$ belongs to $H_p$. Therefore $U(h-h')=1$, hence 
$U(h')=U(h+h'-h)=U(h)+U(h'-h)=U(h)$. It follows that $k$ doesn't depend of the choice of $g$. We let 
$f_{G,p}(n)=k$. \\
\indent 
(2) Let $g\in G$ such that $g+H_p=1_{G/H_p}$, $m$, $n$ in $\NN\backslash \{0\}$ and $h$, $h'$ be the elements of $G$ 
such that $g=p^{n} h=p^{m+n}h'$. 
Then $h=p^mh'$. It follows that $U(h)=U(p^mh')$, so $\frac{2f_{G,p}(n)\pi}{p^{n}}$ is congruent to 
$\frac{p^m2f_{G,p}(m+n)\pi}{p^{m+n}}$ modulo $2\pi$. This in turn is equivalent to saying that 
$f_{G,p}(m+n)-f_{G,p}(n)$ belongs to $p^{n}\ZZ$. Say, $f_{G,p}(m+n)=ap^{n}+f_{G,p}(n)$. Since $f_{G,p}(n)<p^{n}$ 
and $f_{G,p}(m+n)>0$, we have $a>0$. Hence $f_{G,p}(n)$ is the remainder of the euclidean division 
of $f_{G,p}(m+n)$ by $p^{n}$. \\ 
\indent By Proposition \ref{n128}, $f_{G,p}(n)\geq 1$, hence, since $f_{G,p}(1)\in \{1,\dots, p-1\}$, 
$p$ and $f_{G,p}(1)$ are coprime. If $n\geq 2$, then $p^{n-1}$ divides $f_{G,p}(n)-f_{G,p}(1)$. Therefore, 
$p$ does not divide $f_{G,p}(n)$. 
\end{proof}
\indent  
If $(G,R)$ is a discrete divisible abelian cyclically ordered group, then for every prime $p$ we have 
$H_p=\{0_G\}$. Hence $G/H_p\simeq G$. The  
the proof of \cite[Proposition 4.36]{Le}, shows that for every  $p$, $n$ the integer $\varphi_{G,p}(n)$ is 
the inverse of $f_{G,p}(n)$ modulo $p^n$. 
Now, Theorem \ref{thm213} can be reformulated in the following way. 
\begin{Theorem}\label{nth1} Let $(G,R)$ and $(G',R')$ be two discrete divisible abelian 
cyclically ordered groups. Then they are elementarily equivalent if, and only if, for every prime $p$ 
the functions $f_{G,p}$ and $f_{G',p}$ are equal. 
\end{Theorem}
\section{C-divisible abelian cyclically ordered groups.}\label{sec3}
\indent Recall that an abelian cyclically ordered group is c-divisible if it is divisible and 
its torsion subgroup is c-isomorphic to $(T(\UU),\cdot)$.
This is also equivalent to $\mbox{uw}(G)$ being divisible.  \\[2mm] 
\noindent 
{\it Proof of Theorem \ref{thm31}}. (1) The fact that $TcD$ is complete follows from Theorem \ref{thm215}. Now, 
Let $(G_1,R_1)$ and $(G_2,R_2)$ be two c-divisible abelian cyclically ordered groups such that $G_1$ is a 
subgroup of $G_2$. Recall that by Lemma \ref{finiteodense} a c-divisible abelian cyclically ordered group is dense. 
By \cite[Theorem 1.8]{Le}, $(G_1,R_1)$ is an elementary subextension of $(G_2,R_2)$ if, and only if, 
$G_1$ is pure in $G_2$ and, for every prime $p$, $[p]G_1=[p]G_2$. Now, if $(G_1,R_1)$ and $(G_2,R_2)$ are 
c-divisible, then 
$G_1$ is pure in $G_2$, and for every prime $p$ we have $[p]G_1=[p]G_2=1$. It follows that $(G_1,R_1)$ is an 
elementary substructure of $(G_2,R_2)$. So, $TcD$ is model complete. \\
\indent (2) Let $(G,R)$ be a c-divisible abelian cyclically ordered group (so it is dense). 
We already saw that the ordered group ${\rm uw}(G)/l(G)_{uw}$ embeds in $\RR$. 
Since ${\rm uw}(G)$ is abelian torsion-free and divisible, it is a $\QQ$-vector space. 
Hence there is a subspace $A$ of ${\rm uw}(G)$ such that the restriction to $A$ 
of the canonical epimorphism $x\mapsto x+l(G)_{uw}$ is one-to-one, and so ${\rm uw}(G)=A\oplus l(G)_{uw}$. 
Note that $A$ is also a divisible abelian subgroup. 
The canonical epimorphism is order pre\-ser\-ving because 
$l(G)_{uw}$ is a convex subgroup of ${\rm uw}(G)$. Hence its restriction to $A$ is order preserving, and 
every positive element of $A$ is greater than $l(G)_{uw}$. It follows that the ordered group ${\rm uw}(G)$ is isomorphic to 
the lexicographic products $A\overrightarrow{\times}l(G)_{uw}$ and $({\rm uw}(G)/l(G)_{uw})
\overrightarrow{\times}l(G)_{uw}$. Let $(G_1,R_1)$ and $(G_2,R_2)$ be two c-divisible abelian cyclically ordered groups, 
$\varphi_1$ be a c-embedding of $(G,R)$ in $(G_1,R_1)$ and $ \varphi_2$ be a c-embedding of $(G,R)$ in 
$(G_2,R_2)$. These c-embeddings 
induce embeddings $\varphi_{1uw}$ of $({\rm uw}(G),\leq_R)$ in 
$({\rm uw}(G_1),\leq_{R_1})$ and $\varphi_{2uw}$ of $({\rm uw}(G),\leq_R)$ in 
$({\rm uw}(G_2),\leq_{R_2})$, where $\varphi_{1uw}(z_G)=z_{G_1}$ and $\varphi_{2uw}(z_G)=z_{G_2}$. 
The restriction of 
$\varphi_{1uw}$ (resp.\ $\varphi_{2uw}$) to $l(G)$ is an embedding of $l(G)$ in $l(G_1)$ (resp.\ $l(G_2)$). 
Since the theory of divisible abelian linearly ordered groups has the amalgamation property, 
there are 
an abelian linearly ordered group $L$ and embeddings $\Phi_1$ of $l(G_1)$ in $L$ and $\Phi_2$ of $l(G_2)$ in $L$ 
such that the restrictions of $\Phi_1\circ \varphi_{1uw}$ and $\Phi_2\circ \varphi_{2uw}$ are equal. Now, 
since ${\rm uw}(G_1)/l(G_1)_{uw}$ and ${\rm uw}(G_2)/l(G_2)_{uw}$ are archimedean, there is a unique 
embedding $\Psi_1$ (resp.\ $\Psi_2$) of ${\rm uw}(G_1)/l(G_1)_{uw}$ (resp.\ ${\rm uw}(G_2)/l(G_2)_{uw}$) 
in $\RR$ such that $\Psi_1(z_{G_1}+l(G_1)_{uw})=1$ (resp.\ $\Psi_2(z_{G_2}+l(G_1)_{uw})=1$). Hence we 
get embeddings of $({\rm uw}(G_1),\leq_{R_1})$ and $({\rm uw G_2},\leq_{R_2})$ in 
$\RR\overrightarrow{\times} L$ such that 
$z_{G_1}$ and $z_{G_2}$ has the same image $Z$, which is positive and cofinal. These embeddings induce 
c-embeddings $\varphi_1'$ of $(G_1,R_1)$ in $(\RR\overrightarrow{\times} L)/\langle Z\rangle$ and 
$\varphi_2'$ of $(G_2,R_1)$ in $(\RR\overrightarrow{\times} L)/\langle Z\rangle$ such that 
$\varphi_1'\circ\varphi_1=\varphi_2'\circ\varphi_2$. \\
\indent (3)  Since $TcD$ is model-complete and has the amalgamation property, it admits the quantifiers elimination. \\
\indent (4) Let $(G,R)$ be a c-divisible abelian cyclically ordered group. Then its torsion subgroup 
is c-isomorphic to $T(\UU)$. Hence 
$T(\UU)$ embeds in $(G,R)$, and this embedding is elementary, since $TcD$ is model complete. \\
\indent (5) Clearly, every model of $TcD$ is an abelian cyclically ordered group. Since $TcD$ is 
model-complete, to prove that it is the model completion of the theory of abelian cyclically ordered groups, 
it remains to prove that every abelian cyclically ordered group embeds in a unique way in a minimal c-divisible one. 
Let $(G,R)$ be an abelian cyclically ordered group, and $\overline{{\rm uw}(G)}$ be the divisible hull 
of ${\rm uw}(G)$. Then $(G,R)\simeq \mbox{uw}(G)/\langle z_G\rangle$ c-embeds in the c-divisible abelian cyclically ordered group 
$\overline{{\rm uw}(G)}/\langle z_G\rangle$. Now, if $(G,R)$ c-embeds in a c-divisible abelian cyclically ordered group 
$(G',R')$,  then $({\rm uw}(G),\leq_R)$ embeds in the divisible linearly ordered group ${(\rm uw}(G'),\leq_{R'})$. Hence 
$\overline{{\rm uw}(G)}$ embeds in ${\rm uw}(G')$. 
Therefore $\overline{{\rm uw}(G)}/\langle z_G\rangle$ c-embeds 
in ${\rm uw}(G')/\langle z_{G'}\rangle\simeq G'$. It follows that $\overline{{\rm uw}(G)}/\langle z_G\rangle$ 
is the c-divisible hull of $(G,R)$. 
\hfill $\qed$
\section{General case.}\label{n4}

\indent 
By \cite[Theorem 4.3.2]{Ro56}, every linearly ordered divisible abelian group is elementarily equivalent to 
the ordered group $(\QQ,+)$. Therefore, 
if $(G,R)$ is a torsion-free divisible abelian linearly cyclically ordered group, then 
it is elementarily equivalent to the linearly cyclically ordered group $(\QQ,+)$. In 
this section we look at the nonlinear cyclically ordered case. 
\subsection{The preordered family $\mathcal{CD}(G)$.}
We recall that if $(G,R)$ is a divisible abelian nonlinear cyclically ordered group, then 
for every prime $p$ such that $G$ is not c-$p$-divisible, $H_p$ denotes either the greatest $p$-divisible convex 
subgroup of $l(G)$. If $G$ is c-$p$-divisible, then $H_p=G$. 
The family $\{\{0_G\},\; G,\; H_p\;:\; p \;{\rm prime}\}$ is denoted by 
$\mathcal{CD}(G)$. 
\begin{Remark}\label{rkcorresp} \begin{enumerate}
\item 
The group $H_p$ is a $p$-divisible c-convex proper subgroup of $(G,R)$, and the ordered  
convex subgroup $H_p$ of $l(G)$ is isomorphic to the proper convex 
subgroup $(H_p)_{uw}$ of ${\rm uw}(G)$. 
\item
The subgroup $(H_p)_{uw}$ is the greatest $p$-divisible convex subgroup of ${\rm uw}(G)$, 
since if $H_p\neq G$, then $z_G$ is not divisible by any integer within ${\rm uw}(G)$ (Lemma \ref{cdiv}). 
\item The quotient cyclically ordered group $G/H_p$ is dense (resp.\ discrete) if, and only if, \\ 
${\rm uw}(G)/(H_p)_{uw}$ is a dense (resp.\ discrete) ordered group (this follows from Lemma \ref{lm12}). 
\end{enumerate}
\end{Remark}
\indent We now prove a lemma which will also be crucial in the study of weakly cyclically minimal 
cyclically ordered groups, for replacing a lemma of Lucas that was false (it brought a contradiction). 
For $g$ in a cyclically ordered group, we let $\theta(g)$ be the element of $[0,2\pi [$ such that 
$U(g)={\rm exp}(i\theta(g))$. 
\begin{Lemma}\label{ouf} Let $(G,R)$ be a divisible abelian nonlinear cyclically ordered group and 
$p$ be a prime such that $G$ does not contain any $p$-torsion element. 
Then the sets 
$$\{U(g) \; :\; g\in G,\; \exists h\in G,\; (R(0_G,h,g,-g) \mbox{ or } 
R(-g,g,h,0_G)) \mbox{ and } ph=g\}$$ 
$$\mbox{and } \{U(g) \; :\; g\in G,\; \forall h\in G,\; (R(0_G,h,g,-g) 
\mbox{ or } R(-g,g,h,0_G)) \Rightarrow ph\neq g\}$$ 
are both dense in $\UU$. 
\end{Lemma}
\begin{proof} By Remark \ref{r}, $U(G)$ is divisible, hence it is dense in $\UU$. Let 
$A=\{ h\in G \; :\; 0<\theta(h)<\frac{\pi}{p}\}$. Then $U(A)$ is dense in the the subset 
$\left\{e^{i\theta}\; :\; 0<\theta< \frac{\pi}{p}\right\}$ of $\UU$, hence $U(pA)$ is dense in 
$\left\{e^{i\theta}\; :\; 0<\theta< \pi\right\}$. Let $h\in A$ and $g=ph$. 
Since $0<\theta(h)<p\theta(h)=\theta(g)<\pi$, both of $g$ and $h$ belong to the positive cone of $(G,R)$, and 
$g$ satisfies $\exists h\in G,\; R(0_G,h,g,-g) \mbox{ and } ph=g$ (recall that $R(0_G,h,g,-g)$ 
stands for $R(0_G,h,g)$ and $R(0_G,g,-g)$, see Subsection \ref{subsec11}). In the same way, 
$U(p(-A))$ is dense in $\left\{e^{i\theta}\; :\; \pi<\theta< 2\pi\right\}$ and 
every $g\in p(-A)$ satisfies $\exists h\in G,\; R(-g,g,h,0_G)\mbox{ and } ph=g$. Hence 
$\{U(g) \; :\; g\in G,\; \exists h\in G,\; (R(0_G,h,g,-g)$ 
or $R(-g,g,h,0_G))$ and $ph=g\}$ is dense in $\UU$. \\
\indent Let $A'=\{h\in G\; :\; 2\pi-\frac{2\pi}{p}<\theta(h)<2\pi-\frac{\pi}{p}\}$. 
Then $U(A')$ is dense in the the set 
$\left\{e^{i\theta'}\; :\; \frac{2(p-1)\pi}{p}<\theta< \frac{(2p-1)\pi}{p}\right\}$. Now, 
$2(p-1)\pi$ is congruent to $0$ modulo $2\pi$, and $(2p-1)\pi$ is congruent to $\pi$. 
Hence $U(pA')$ is dense in 
$\left\{e^{i\theta}\; :\; 0<\theta< \pi\right\}$. Let $g\in pA'$ and $h\in A'$ be such that 
$g=ph$. Since $g$ belongs to the positive cone of $G$, and $h$ does not, we have $R(0,g,-g)$, 
and $R(0,g,h)$.  
Since $G$ does not contain $p$-torsion elements, this element $h$ such that $g=ph$ is unique. 
Hence $g$ belongs to the set 
$\{U(g) \; :\; g\in G,\; \forall h\in G,\; (R(0_G,h,g,-g) \mbox{ or } R(-g,g,h,0_G)) \Rightarrow ph\neq g\}$. 
In the same way, $U(p(-A'))$ is a subset of 
$\{U(g) \; :\; g\in G,\; \forall h\in G,\; (R(0_G,h,g,-g) \mbox{ or } R(-g,g,h,0_G)) \Rightarrow ph\neq g\}$, and 
$U(p(-A'))$ is dense in $\displaystyle{\left\{e^{i\theta}\; :\; \pi<\theta< 2\pi\right\}}$.
Hence  $\{U(g) \; :\; g\in G,\; \forall h\in G,\; (R(0_G,h,g,-g) \mbox{ or } R(-g,g,h,0_G)) \Rightarrow ph\neq g\}$ 
is dense in $\UU$. 
\end{proof}
\indent In order to show that the family $\mathcal{CD}(G)$ depends on the first-order 
theory of $G$, we start with a lemma and we prove Proposition \ref{n128}. 
\begin{Lemma}\label{propHp} Let $(G,R)$ be a divisible abelian nonlinear cyclically ordered group.  
\begin{enumerate} 
\item Let $p$ be a prime such that $G$ is not c-$p$-divisible. The subgroup $H_p$ is definable by 
the formula 
$g=0_G$ 
$$\mbox{or }(R(0_G,g,-g)\mbox{ and }\forall g'\in G \; 
(g'=g \mbox{ or } R(0_G,g',g))\Rightarrow \exists h'\; R(0_G,h',g')\mbox{ and }g'=ph')$$ 
$$\mbox{or }(R(-g,g,0_G)\mbox{ and }\forall g'\in G \;
(g'=g \mbox{ or } R(g,g',0_G))\Rightarrow \exists h'\; R(g',h',0_G)\mbox{ and }g'=ph').$$ 
\item The inclusion $H_p\subseteq H_q$ depends on the first-order theory of $(G,R)$. 
\item The cyclically ordered group $G/H_p$ being discrete is determined by the 
first-order theory of $(G,R)$. 
\end{enumerate}
\end{Lemma}
\begin{proof}(1) Since $G$ is not c-$p$-divisible, it does not contain $p$-torsion element. Hence 
for every $g$ in $G$ there is only one $h$ in $G$ such that $ph=g$. 
The formula of this statement can be reformulated as $g=0_G$ or (*) $g$ belongs to the 
positive cone $P$ of $(G,R)$ and every $g'$ in $I(0_G,g)$ is $p$-divisible within $I(0_G,g')$, or 
(**) $g\in -P$ and every $g'$ in $I(g,0_G)$ is $p$-divisible within $I(g',0_G)$. 
Let $H$ be the set of elements which satisfy this formula. For $g\neq 0_G$ in $H$, we have 
either $g\in P$ and $I(0_G,g)\subseteq H$, or $g\in -P$ and $I(g,0_G)\subseteq H$, and clearly, 
$g\in H\Leftrightarrow -g\in H$. Therefore, $H$ is equal to the union of the c-convex subsets 
$\{-g\}\cup I(-g,g)\cup\{g\}$ where $g\in (P\cap H)\backslash\{0_G\}$. It follows that $H$ 
is a c-convex subset of $G$. \\
\indent We assume that $G$ does not contain $p$ torsion elements, and we prove that $H\subseteq l(G)$. 
It follows that, for every $g$ in $G$ there is only one $h$ in $G$ such that $ph=g$. \\
\indent 
Let $g\notin l(G)$, then $U(g)\neq 1$. By Lemma 
\ref{ouf}, the set $U(G\backslash H)$ is dense in $U(G)$. 
Hence there exist 
$h$, $h'$ in $G\backslash H$ such that $R(1,U(h),U(g),U(h'))$ holds in $\UU$. 
Therefore, we have $R(0_G,h,g,h')$. So $I(0_G,g)\nsubseteq H$ and $I(g,0_G)\nsubseteq H$, which proves that 
$g\notin H$. Therefore $H\subseteq l(G)$. \\
\indent Now, it follows from (*) and (**) that $H$ is the greatest  $p$-divisible c-convex subset of $l(G)$. 
In particular, $H_p\subseteq H$. Therefore, 
in order to prove that $H=H_p$, it is sufficient to prove that $H$ is a subgroup of $l(G)$. 
The set $H$ is nonempty since it contains $0_G$. 
Assume that $0_G<g\in l(G)$ 
and every $h\in l(G)$ such that $0_G<h\leq g$ is $p$-divisible within $l(G)$. Then every $h\in l(G)$ such 
that $-g<h<0_G$ also is $p$-divisible within $l(G)$, so $-g\in H$. In the same way, if $0_G>g\in l(G)$, then $-g\in l(G)$. 
Let $g$, $g'$ in $H$. If $g'<0_G<g$, then $g'<g+g'<g$, so $g+g'\in H$, since $H$ is convex. 
If $0_G<g'<g$, then every $h\in l(G)$ such that $g<h<g+g'$ can be written as $h=g'+(h-g')$, where 
$h-g'$ is $p$-divisible within $l(G)$ since $0_G<h-g'<g$. Therefore, $h$ is $p$-divisible within $l(G)$, 
which proves that $g+g'\in H$. 
In the same way, if $g<g'<0_G$, then $g+g'\in H$. So, we proved that $H$ is a subgroup of $l(G)$. 
Consequently, $H=H_p$. \\
\indent (2)  Since $H_p$ and $H_q$ are definable, the inclusion $H_p\subseteq H_q$ also is definable.  \\
\indent (3) 
The quotient group $G/H_p$ being discrete is definable by the formula 
$\exists g\in G\; R(0_G,g,-g)\; g\notin H_p\; {\rm and}\; \forall h\in G\; 
R(0_G,h,g)\Rightarrow (h\in H_p\; {\rm or}\; g-h\in H_p)$. 
\end{proof}
\indent For each positive integer $n$, the formula: 
$${\rm argbound}_n(g)=R(0_G,g,2g, \dots,ng)\;\&\;\neg R(0_G,g,2g, \dots,2ng,(n+1)g)$$ 
is satisfied in $(G,R)$ exactly by those elements $g$ such that either 
$\displaystyle{\theta(g)=\frac{2\pi}{n+1}\;\&\; 0_G<(n+1)g}$ (in $l(G)$), 
or $\displaystyle{\frac{2\pi}{n+1}<\theta(g)<\frac{2\pi}{n}}$, or 
$\displaystyle{\theta(g)=\frac{2\pi}{n}\;\&\; 0_G>ng}$ (in $l(G)$) 
(see \cite[Lemmas 6.9 and 6.10]{GiLL}). \\[2mm]
\noindent 
{\it Proof of Proposition \ref{n128}}. 
The first part of this proposition has been proved in (1) of Lemma 
\ref{rk240}. It remains to prove that the function $f_{G,p}$ is determined by the first-order theory of 
$(G,R)$. 
By Proposition \ref{propHp}, the quotient group $G/H_p$ being discrete is definable. \\
\indent  If $p=2$ and $r=1$, then $f_{G,2}(1)=1$. In the following, we assume that $p>2$, or $p=2$ and $r>1$. 
By (2) of Lemma \ref{rk240}, $p$ and  $f_{G,p}(r)$ are coprime. Hence 
$f_{G,p}(r)$ has a unique inverse modulo $p^r$ in $\{1,\dots,p^{r}-1\}$. 
Now, the first-order theory of $(G,R)$ determines $f_{G,p}(r)$ if, and only if, it determines its inverse 
modulo $p^r$. 
Let $l$ in $\{1,\dots,p^{r}-1\}$. Note that $l$ being the inverse of  $f_{G,p}(r)$ modulo $p^r$ is 
equivalent to $\displaystyle{U(lh)=e^{\frac{2i\pi}{p^{r}}}}$.
We prove that this holds if, and only if, ${\rm argbound}_{p^{r}-1}(lh)$ holds. With above notations, 
${\rm argbound}_{p^{r}-1}(lh)$ holds if, and only if,  
$$\displaystyle{\mbox{either }\theta(lh)=\frac{2\pi}{p^{r}}\;\&\; 0_G<p^{r}lh\; (\mbox{in }l(G)), 
\mbox{ or }\frac{2\pi}{p^{r}}<\theta(lh)<\frac{2\pi}{p^{r}-1},}$$  
$$\displaystyle{\mbox{ or }\theta(lh)=\frac{2\pi}{p^{r}-1}\;\&\; 0_G>(p^{r}-1)lh\; (\mbox{in }l(G)).}$$ 
\indent Now, $\theta(lh)$ belongs 
to the set $\left\{0,\frac{2\pi}{p^{r}},\dots,\frac{2(p^{r}-1)\pi}{p^{r}}\right\}$, and since $p>2$ we have 
$\frac{4\pi}{p^{r}}>\frac{2\pi}{p^{r}-1}$. Hence if ${\rm argbound}_{p^{r}-1}(lh)$ holds, then 
$\displaystyle{\theta(lh)=\frac{2\pi}{p^{r}}\;\&\; 0_G<p^{r}lh}$. Conversely, assume that 
$\displaystyle{\theta(lh)=\frac{2\pi}{p^{r}}}$. Then $2\pi$ divides $\displaystyle{\theta(p^{r}lh)}$, so 
$p^{r}lh$ belongs to $l(G)$, and $p^{r}lh=lg>0_G$. Hence $p^{r}lh$ is positive in 
$l(G)$, so ${\rm argbound}_{p^{r}-1}(lh)$ holds. 
\hfill $\qed$ 
\begin{Proposition}\label{propunsens}
Let $(G,R)$ and $(G',R')$ be divisible abelian cyclically 
ordered groups such that $(G',R')\equiv (G,R)$. Then the 
mapping $H_p\mapsto H_p'$ defines an isomorphism from $\mathcal{CD}(G)$ onto $\mathcal{CD}(G')$, for 
every prime $p$ the cyclically ordered group $G/H_p$ is discrete if, and only if, $G'/H_p'$ is discrete, 
and if this holds, then $f_{G,p}=f_{G',p}$. 
\end{Proposition}
\begin{proof} This follows from Lemma \ref{propHp} and Proposition \ref{n128}. 
\end{proof}
\indent The converse of this proposition will be proved in Subsection \ref{ssssec233}. First we turn 
to reminders about model theory of linearly ordered abelian groups. 
\subsection{The spine of Schmitt of a linearly ordered abelian group.}\label{nsec32} 
By \cite[Theorem 4.1]{GiLL}, any two abelian cyclically ordered groups $(G,R)$ 
and $(G',R')$ are elementary e\-qui\-va\-lent if, and only if, their Rieger unwounds 
$({\rm uw}(G),\leq_R,z_G)$ and $({\rm uw}(G'),\leq_{R'},z_{G'})$ are elementarily 
e\-qui\-va\-lent in the language of ordered groups together with a specified element interpreted by $z_G$: 
$$(G,R)\equiv (G',R')\mbox{ if, and only if, }({\rm uw}(G),\leq_R,z_G)\equiv 
({\rm uw}(G'),\leq_{R'},z_{G'}).$$
\indent So, determining  
the first-order theory of $(G,R)$ is equivalent to determined the first-order theory of 
$({\rm uw}(G),z_G)$. Lucas based its study of the first-order theory of 
divisible abelian cyclically ordered groups on Schmitt papers about model theory of 
linearly ordered groups 
(cf.\  \cite[p.\ 15 and the following ones]{Sc1}). \\
\indent We let $(\Gamma,\leq)$ be a linearly ordered abelian group. \\
\indent For every positive integer $n$, the $n$-spine of $(\Gamma,\leq)$ is a family of convex subgroups of $\Gamma$: 
$$Sp_n(\Gamma)=
\{A_n(x)\;:\; x\in \Gamma\backslash \{0_{\Gamma}\}\}\cup \{ F_n(x)\;:\; x\in \Gamma\backslash n\Gamma \},$$
where the convex subgroups $A_n(x)$ and $F_n(x)$ are defined as follows. \\
\indent 
For every $x\in \Gamma$, $F_n(x)$ is the greatest convex subgroup of $(\Gamma,\leq)$ such that $F_n(x)\cap 
(x+n \Gamma)=\emptyset$ if any, otherwise $F_n(x)=\emptyset$. By \cite[Lemma 2.8, p.\ 25]{Sc1},  $F_n(x)=\emptyset$ 
is equivalent to $x\in n\Gamma$. \\ 
\indent  We define $A_n(0_{\Gamma})=\emptyset$. For $x\in \Gamma\backslash \{0_{\Gamma}\}$ we denote by 
$A(x)$ the greatest convex subgroup of $\Gamma$ which doesn't contain $x$, and by $B(x)$ the 
convex subgroup 
generated by $x$. We know that $B(x)/A(x)$ is an archimedean linearly ordered abelian group, so it is regular 
(regular ordered groups have been defined in Subsection \ref{n3}). \\
\indent 
For $n\in \NN\backslash \{0\}$ we let $A_n(x)$ be the smallest convex subgroup of $(\Gamma,\leq)$ such 
that $B(x)/A_n(x)$ is $n$-regular. \\[2mm]
\indent For $x$, $y$ in $\Gamma$, $A_n(y)\subsetneq A_n(x)$ is definable by the first-order formula 
$$|y|<|x|\;\&\; \exists u\;(|y|\leq u\leq n|x|)\;\&\; \forall v(|v|<n|y|\Rightarrow v-u\notin n\Gamma)
\mbox{ (\cite[Lemma 2.6, p.\ 21]{Sc1})}.$$
\indent (Here, $|y|\leq u$ stands for: $0_{\Gamma}\leq y\leq u$ or $0_{\Gamma}<-y\leq u$.) \\
\indent For $x\in \Gamma\backslash\{0_{\Gamma}\}$ we also define $B_n(x)$ to be the greatest convex subgroup of $\Gamma$ such that 
$B_n(x)/A_n(x)$ is $n$-regular, and 
$C_n(x)=B_n(x)/A_n(x)$ (so, $C_n(x)$ is $n$-regular). \\[2mm]
\indent The language of $Sp_n(\Gamma)$ consists in a binary predicate $<_S$ and unary predicates $A$, $F$, $Dk$, 
$\{ \beta_{p,m} \; :\; p \mbox{ prime, } m\in \NN\}$, $\{ \alpha_{p,k,m} \; :\; p \mbox{ prime, } k\in \NN,\; m\in \NN\}$. 
These predicates are interpreted as follows. \\
\indent $<_S$: $C_1<_SC_2\Leftrightarrow C_2\subsetneq C_1$. \\
\indent $A(C)$: $\exists x\in \Gamma,\; C=A_n(x)$.\\
\indent $F(C)$: $\exists x\in \Gamma,\; C=F_n(x)$. \\
\indent $Dk(C)$: $\exists x\in \Gamma,\; C=A_n(x)\; \&\; C_n(x)$ is discrete.\\
\indent If $C$ is an abelian group, then for every prime $p$ and non-negative integer 
$r$ the quotient group $p^rC/p^{r+1}C$ is a $\ZZ/p\ZZ$-vector space. In the following two formulas, 
dim denotes the dimension of this vector space. Furthermore, ${\rm Tor}_p(C)$ denotes the 
subgroup of elements of $C$ which torsion is an exponent of $p$. 
One can check that 
the sets $\{y\in \Gamma\;:\; F_n(y)\subseteq F_n(x)\}$ and $\{y\in \Gamma\;:\; F_n(y)
\subsetneq F_n(x)\}$ are subgroups of $\Gamma$. \\
\indent $\beta_{p,m}(C)$: $\exists x\in \Gamma,\; C=A_n(x)$ and $\limss{r}{+\infty}{\rm dim}
\displaystyle{\left(p^rC_n(x)/p^{r+1}C_n(x)\right)\geq m }$.\\
\indent $\alpha_{p,k,m}(C)$: $\exists x\in \Gamma,\; C=A_n(x)$ and $m\leq {\rm dim}
\displaystyle{\left({\rm Tor}_p (p^kF_n^*(x))/{\rm Tor}_p (p^{k+1}F_n^*(x))\right)}$, \\
\indent where 
$\displaystyle{F_n^*(x)=\frac{\{y\in \Gamma\;:\; F_n(y)\subseteq F_n(x)\}}{\{y\in \Gamma\;:\; F_n(y)
\subsetneq F_n(x)\}} }$, 
or $F_n^*(x)=\{0_{\Gamma}\}$ if $\{y\in \Gamma\;:\; F_n(y)\subsetneq F_n(x)\}=\emptyset$. 
\begin{Theorem}(\cite[Corollary 3.2, p.\ 33]{Sc1}).\label{thspine} 
If $\Gamma'$ is an other abelian linearly ordered group, 
then $\Gamma \equiv \Gamma'$ if, and only if, for every integer $n$ we have 
$Sp_n(\Gamma)\equiv Sp_n(\Gamma')$. \end{Theorem}
\indent There is also a relative quantifiers elimination when adding the predicates 
$M_{n,k}(x)$, $D_{p,r,i}(x)$ and $E_{p,r,i}(x)$ ($i<r$) interpreted as follows. \\
\indent 
$M_{n,k}(x)$ holds if, and only if, $C_n(x)$ is discrete and if the class $x_0+A_n(x)$ of some 
$x_0\in \Gamma$ is the 
smallest positive element of $C_n(x)$, then $x+A_n(x)=k(x_0+A_n(x))$.\\
\indent 
$D_{p,r,i}(x)$ holds if, and only if, either $x\in p^r\Gamma$ or there is some $y\in \Gamma$ 
such that $F_{p^r}(x)\subsetneq 
F_{p^r}(y)$, $x-y\in p^i\Gamma$ and  $F_{p^r}\left(\frac{x-y}{p^i}\right)\subseteq F_{p^r}(x)$.\\
\indent 
$E_{p,r,k}(x)$ holds if, and only if, there exists $y\in \Gamma$ such that $F_{p^r}(x) = A_{p^r}(y)$ and 
$C_{p^r}(y)$ is discrete with smallest positive element the class of $y$ and $F_{p^r}(x-ky)
\subsetneq F_{p^r}(x)$.\\
\indent 
The relative quantifiers elimination is the following. 
\begin{Theorem}(\cite[Theorem 4.5, p.\ 66]{Sc1})\label{relelimquant} 
Let $\Psi(Z_1,\dots,Z_m)$ be a formula in the language of 
ordered groups. Then there exist an integer $n$, a formula $\varphi_0(X_1,\dots,X_k,Y_1,\dots,Y_k)$ 
in the language $Sp_n$, some terms $t_1,\dots,t_k$ with variables $Z_1,\dots,Z_m$ and a 
quantifier-free formula 
$\varphi_1(Z_1,\dots,Z_m)$ in above language such that for every abelian ordered group 
$(\Gamma,\leq)$ and every 
sequence $\overrightarrow{x}=(x_1,\dots,x_m)$ of elements of $\Gamma$ we have 
$(\Gamma,\leq)\models \Psi(\overrightarrow{x})$ if, and only if, 
$$Sp_n(\Gamma)\models \varphi_0(A_n(t_1(\overrightarrow{x})),\dots,
A_n(t_k(\overrightarrow{x})),F_n(t_1(\overrightarrow{x})),\dots,F_n(t_k(\overrightarrow{x}))))\; 
{\rm and}\; (\Gamma,\leq)\models \varphi_1(\overrightarrow{x}).$$ 
\end{Theorem}
\subsection{Spines of ``almost divisible'' abelian groups.}\label{nsec33} 
\indent If $(\Gamma,\leq)$ is an $n$-divisible linearly ordered abelian group, then for every positive integer 
$n$ and every $x\in \Gamma$ we have $A_n(x)=\{ 0_{\Gamma}\}$, $B_n(x)=\Gamma$ and $F_n(x)=\emptyset$. It follows that 
the predicates $\alpha_{p,k,m}$ and $\beta_{p,m}$ don't make sense. Now we turn to a less trivial case. \\
\indent By Lemma \ref{pclass}, if $G$ is $p$-divisible without $p$-torsion element, 
then $[p]{\rm uw}(G)=p$, and if $G$ is c-$p$-divisible, then $[p]{\rm uw}(G)=1$. 
So in this subsection we focus on the case where for every prime $p$ we have $[p]\Gamma\in\{1,p\}$. \\
\indent For every linearly ordered abelian group $(\Gamma,\leq)$ and every prime $p$, we denote by $\Gamma_p$ the greatest 
$p$-divisible convex subgroup of $(\Gamma,\leq)$. 
The group 
$\Gamma_p$ is definable by the formula 
$$x\in \Gamma_p\Leftrightarrow \forall y\in \Gamma \; 
(0_{\Gamma}\leq y\leq x\mbox{ or } x\leq y\leq 0_{\Gamma}) \Rightarrow\exists z\in \Gamma \; y=pz.$$
\indent Indeed, Let $\Gamma_p'$ be the subset defined by this formula. Then, $0_{\Gamma}\in \Gamma_p'$ and 
for every $x\in \Gamma$ we have $x\in \Gamma_p'\Leftrightarrow -x\in \Gamma_p'$. So, in the same way as 
in the proof of (1) of Lemma \ref{propHp}, $\Gamma_p'$ is the greatest $p$-divisible convex subset of $\Gamma$. 
Finally, the proof of (1) of Lemma \ref{propHp} shows that $\Gamma_p'$ is a subgroup (in Lemma \ref{propHp}, 
If $H_p\neq G$, then 
$H_p\subseteq l(G)$, hence we considered a convex subset of a linearly ordered group). \\[2mm]
\indent This subsection is dedicated to prove the following proposition. 
\begin{Proposition}\label{corr222a} Let $(\Gamma,\leq)$ be a linearly ordered abelian group 
such that, for every prime $p$, $[p]\Gamma\in\{1,p\}$. Then, the spines of the $(\Gamma,\leq)$ are 
determined by the ordered set of $\Gamma_p$'s ($p$ prime), and if this set admits a maximal element $\Gamma_p$, 
the property $\Gamma/\Gamma_p$ being discrete or dense. 
\end{Proposition}
\indent In the remainder of this subsection, $(\Gamma,\leq)$ is a linearly ordered abelian group 
such that, for every prime $p$, $[p]\Gamma\in\{1,p\}$. 
\begin{Lemma}\label{lmajoutea} Let  
$C\subsetneq C'$ be two convex subgroups of $(\Gamma,\leq)$ and $p$ be a prime. 
Then $C'/C$ is $p$-divisible if, and only if, either $\Gamma_p\subsetneq C$, or $C'\subseteq \Gamma_p$. 
\end{Lemma}
\begin{proof} Clearly, if $C'\subseteq \Gamma_p$, then $C'/C$ is $p$-divisible. 
This holds if $\Gamma_p=\Gamma$, i.e.\ $[p]\Gamma=1$. Now, we assume $[p]\Gamma=p$. \\
\indent 
We have: $[p]C=1$ if, and only if, $C\subseteq \Gamma_p$. 
Otherwise, $[p]C=p$ (the same holds with $C'$). \\
\indent  Assume that $\Gamma_p\subsetneq C$. 
Then, there exists $x\in C\backslash pC$ such that the classes 
of $C$ modulo $p$ are $pC,x+pC,\dots, (p-1)x+pC$. 
Since $C$ is a convex subgroup of $C'$, $x$ is not 
$p$-divisible within $C'$ (if $x=py$, then $|y|\leq |x|$, hence $y\in C$). Now, $[p]C=[p]C'=p$, so 
the classes of $C'$ modulo $p$ are $pC',x+pC',\dots, (p-1)x+pC'$. 
Let $x'\in C'$ and $j\in \{0,\dots, p-1\}$ such that $x'\in jx+pC'$. Then, there is $y'\in C'$ such that $x'-py'=jx$. 
Since $jx\in C$, we have  
$x'+C=py'+C$, which proves that $C'/C$ is $p$-divisible. \\ 
\indent Assume that $C\subseteq \Gamma_p\subsetneq C'$. Let $x'\in C'$ such that the class of $x'$ modulo $C$ 
is $p$-divisible. Hence, there exists $y'\in C'$ such that $x'-py'\in C$. Now, $C=pC$, hence 
$x'-py'\in pC\subseteq pC'$, which proves that $x'\in pC'$. Since $C'\neq pC'$, we see that $C'/C$ is not $p$-divisible. 
\end{proof}
\begin{Corollary}\label{corr227} If there is some $C\subsetneq C'$ in the set 
$\{ \Gamma,\; \Gamma_p \; :\; p \mbox{ prime }\}$ such that $C'/C$ is discrete, 
then $C'=\Gamma$, and $C=\max \{ \Gamma_p\; :\; p \mbox{ prime }\}$  
\end{Corollary}
\begin{proof} If there exists  $p_1$ and $p_2$ such that $\{0_G\}\subsetneq \Gamma_{p_1}\subsetneq \Gamma_{p_2}$, 
then $\Gamma_{p_2}/\Gamma_{p_1}$ is $p_2$-divisible. It follows that it is dense, and that $\Gamma/\Gamma_{p_1}$ 
is dense. Consequently, if some group $C'/C$ is discrete, then $C=\max \{ \Gamma_p\; :\; p \mbox{ prime }\}$. 
Since $C'\in \{ \Gamma,\; \Gamma_p \; :\; p \mbox{ prime }\}$, we have $C'=\Gamma$. 
\end{proof}
\indent Note that if $C=\{0_{\Gamma}\}$ and $C'/C$ is discrete, then $C'$ is discrete. 
Hence $\Gamma$ is discrete and for every prime $p$ we have $\Gamma_p=\{ 0_{\Gamma}\}$. \\[2mm]
\indent Before to prove the next Proposition, we review some properties of the sets $A_n(x)$ and 
$F_n(x)$.  
\begin{Lemma}\label{npropav411}
Assume that $n=p_1^{r_1}\cdots p_k^{r_k}$, where $p_1,\dots,p_k$ 
are the primes which divide $n$ and let $x\in \Gamma$.
\begin{enumerate}[label={(\alph*)}]
\item\label{pta} 
We have $F_n(x)\subseteq A_n(x)$ and $(F_n(x)=\emptyset 
\Leftrightarrow x\in n\Gamma)$ (\cite[Lemma 2.8, p.\ 25]{Sc1}). One can check that if $x\notin n\Gamma$, then 
$F_n(x)$ contains the greatest $n$-divisible convex subgroup of $\Gamma$. 
\item\label{ptb}  The set 
$A_n(x)$ is the maximum of the $A_{p_j}(x)$ (\cite[Lemma 2.3 p.\ 19]{Sc1}). 
\item\label{ptc} The set $F_n(x)$ is the maximum of the 
$F_{p_j^{r_j}}(x)$'s, where $1\leq j\leq k$ (\cite[Lemma 2.3 p.\ 19]{Sc1}). 
If $m$ is an other integer, then 
$F_{mn}(mx)=F_n(x)$, and if $(m,n)=1$, then $F_n(x)=F_n(mx)$ (\cite[3 and 4 of Lemma 2.10, p.\ 26]{Sc1}). 
\end{enumerate}
\end{Lemma}
\begin{Proposition}\label{prop222a} Let $n=p_1^{r_1}\cdots p_k^{r_k}$ be a positive integer, 
where $p_1,\dots,p_k$ are the primes which divide $n$. 
\begin{enumerate}
\item 
$Sp_n(\Gamma)=\{\Gamma_{p_i} \; :\; 1\leq i\leq k \}\cup \{ \,\{0_{\Gamma}\}\}$. We denote by  
$\{0_{\Gamma}\}=C_1\subsetneq C_2\subsetneq \cdots\subsetneq C_l\subsetneq C_{l+1}=\Gamma$, the elements of 
$Sp_n(\Gamma)\cup \{\Gamma\}$. \\ 
Let $x\in \Gamma\backslash \{0_{\Gamma}\}$ and $j$ be the integer such that $x\in C_{j+1}\backslash C_j$.  
Then $A_n(x)=C_j$ and $B_n(x)=C_{j+1}$. In particular, for every $\alpha\in \ZZ\backslash\{0\}$ we have 
$A_n(\alpha x)=A_n(x)$ and $B_n(\alpha x)=B_n(x)$. \\
If $x\notin n\, {\rm uw}(G)$, then 
$F_n(x)=\max \{\Gamma_{p_i} \; :\; 1\leq i\leq k,\; p_i^{r_i}\nmid x\}$. 
\item Let $C$, $C'$ in $Sp_n(\Gamma)$. Then $A(C)$, $F(C')$ hold if, and only if, $C$ 
and $C'$ belong to $\{\Gamma_{p_1},\dots ,\Gamma_{p_k}\}$. Furthermore, there exists $x$ such that 
$C=A_n(x)$ and $C'=F_n(x)$ if, and only if, $C'\subseteq C$. 
\item Let $C$, in $Sp_n(\Gamma)$. If $Dk(C)$ holds, then $C=C_l$, and for any $x$ such that $A_n(x)=C$ we have $B_n(x)=\Gamma$. 
\item Let $j\in \{2,\dots,l\}$. 
Then $\beta_{p,0}(C_j)$ holds, for $m\geq 2$, $\beta_{p,m}(C_j)$ does not hold, and $\beta_{p,1}(C_j)$ holds if, and 
only if, $C_j\subseteq \Gamma_p\subsetneq C_{j+1}$. 
\item Let $C\in Sp_n(\Gamma)\backslash \{\;\{0_{\Gamma}\}\;\}$. 
If $m>1$, then $\alpha_{p,k,m}(C)$ does'nt hold. If 
$m=0$ then it holds. In the case $m=1$ it holds if, and only if, $p^{k+1}$ divides $n$ and 
$\Gamma_p\subseteq C$. 
\end{enumerate}
\end{Proposition}
\begin{proof} 
(1) Let $x\in \Gamma\backslash \{0_{\Gamma}\}$ and $p$ be a prime. \\
\indent Assume that $x\in \Gamma_p$. Then $B(x)\subseteq \Gamma_p$. Hence $A_p(x)=\{0_{\Gamma}\}$. 
By (a) of Lemma \ref{npropav411} we have $F_p(x)=\emptyset$. If $\Gamma_p\subsetneq C$ for some convex subgroup $C$, 
then by Lemma \ref{lmajoutea} 
$C/\Gamma_p$ is not $p$-divisible, so $C$ is not $p$-regular. Hence $C/A_p(x)= C/\{0_{\Gamma}\}$ is not $p$-regular. Therefore, 
$B_p(x)=\Gamma_p$. \\
\indent Assume that $x\notin \Gamma_p$. 
Then, $\Gamma_p\subsetneq B(x)$. The group $B(x)/C$ is $p$-regular if, and 
only if, for every $C'$ such that $C\subsetneq C'\subsetneq B(x)$, the group 
$\left(B(x)/C\right)/\left(C'/C\right)$ 
is $p$-divisible, i.e.\ $B(x)/C'$ is $p$-divisible. By Lemma \ref{lmajoutea}, this is e\-qui\-va\-lent to $\Gamma_p
\subsetneq C'$. It follows that $A_p(x)=\Gamma_p$, and $B_p(x)=\Gamma$. 
We deduce from (b) of Lemma \ref{npropav411} that 
$A_n(x)$ is equal to the maximum of the $\Gamma_{p_i}$ such that $x\notin \Gamma_{p_i}$, $1\leq i\leq k$ 
(or to $\{ 0_{\Gamma}\}$ if the preceding set is empty). \\
\indent Assume that $x=p_i^dx'$, where $x'\notin p_i\,\Gamma$, and let $r$ be a positive integer. 
By (a) of Lemma \ref{npropav411}, 
$\Gamma_{p_i}\subseteq F_{p_i^r}(x')\subseteq A_{p_i^r}(x')=\Gamma_{p_i}$. Hence  
$F_{p_i^r}(x')=\Gamma_{p_i}$. 
If $p_i^r$ divides $x$, then, by (a) of Lemma \ref{npropav411}, $F_{p_i^r}(x)=\emptyset$. 
Now, assume that $p_i^r$ doesn't divide $x$ (i.e.\ $r>d$). 
By (c) of Lemma \ref{npropav411}, $F_{p_i^r}(x)=F_{p_i^{r-d}p_i^d}(p_i^dx')=F_{p^{r-d}}(x')=\Gamma_{p_i}$. Therefore, 
if $x\notin n\, \Gamma$, then, by (c) of Lemma \ref{npropav411},  
$F_n(x)=\max \{\Gamma_{p_i} \; :\; 1\leq i\leq k,\; p_i^{r_i}\nmid x\}$. 
Hence $Sp_n(\Gamma)=\{\Gamma_{p_i} \; :\; 1\leq i\leq k \}\cup \{ \,\{0_{\Gamma}\}\}$, that we can write as  
$\{\,\{0_{\Gamma}\}=C_1\subsetneq C_2\subsetneq \cdots\subsetneq C_l\}$. The hypothesis $[p]\Gamma=p$, for every prime, 
implies that $\Gamma_p\neq \Gamma$. So 
$\Gamma\notin Sp_n(\Gamma)$, for convenience we let $C_{l+1}=\Gamma$.\\
\indent We let $C_j=A_n(x)$. Since $C_j$ is the maximum of the $\Gamma_{p_i}$ such that $x\notin \Gamma_{p_i}$, $1\leq i\leq k$ 
(or to $\{ 0_{\Gamma}\}$ if the preceding set is empty), we have $x\in C_{j+1}\backslash C_j$. 
Let $C$, $C'$ be convex subgroups such that $C_j\subsetneq C\subsetneq C'$, 
and $i\in \{1,\dots,k\}$. By Lemma \ref{lmajoutea}, $C'/C$ is $p_i$-divisible if, and only if, either 
$\Gamma_{p_i}\subsetneq C$ or $C'\subseteq \Gamma_{p_i}$. Hence $C'/C_j$ is $p_i$-regular if, and only if, either 
$\Gamma_i\subseteq C_j$ or $C'\subseteq \Gamma_{p_i}$. Consequently, $C'/C_j$ is $n$-regular if, and only if, 
$C'\subseteq C_{j+1}$. Therefore, $B_n(x)=C_{j+1}$.  This proves (1). \\ 
\indent (2) Predicates $A$ and $F$. Since $\emptyset \notin Sp_n(\Gamma)$, 
it follows from (1) that $A(C)$, $F(C')$ hold if, and only if, 
$C$ and $C'$ belong to $\{\Gamma_{p_1},\dots ,\Gamma_{p_k}\}$.
By (a) of Lemma \ref{npropav411}, if there exists $x$ such that $A_n(x)=C$ and $F_n(x)=C'$, 
then $C'\subseteq C$. In order to 
prove the converse, we assume that 
$\Gamma_{p_1}\subseteq \cdots \subseteq\Gamma_{p_{s-1}}\subsetneq \Gamma_{p_s}=\cdots =\Gamma_{p_t}
\subsetneq \Gamma_{p_{t+1}}\subseteq \cdots \subseteq \Gamma_{p_k}$ and $C=\Gamma_{p_s}$, we denote 
by $B$ the set $B_n(x)$ for any $x\in \Gamma$ such that $A_n(x)=C$ (this is equivalent to $B=\Gamma_{p_{t+1}}$ 
if $t<k$ and to $B=\Gamma$ if $t=k$). By the definition of the sets $\Gamma_p$, for $i\in \{1,\dots, t\}$ 
there exists an element 
$x_i \in B\backslash \Gamma_{p_i}=B\backslash C$ which is not divisible by $p_i$. We let 
$x=p_2\cdots p_t x_1+p_1p_3\cdots p_tx_2+\cdots + p_1\cdots p_{t-1}x_t$.  
Then $x$ is not divisible by any of $p_1,\dots, p_t$. In particular,  
$x\in B\backslash C$. Hence $A_n(x)=C$ and $F_n(x)=C$. The element 
$p_s^{r_s}\cdots p_t^{r_t}x$ belongs to $B\backslash C$. Hence $A_n(p_s^{r_s}\cdots p_t^{r_t}x)=C$, 
and $F_n(p_s^{r_s}\cdots p_t^{r_t}x)=
\Gamma_{p_{s-1}}$. In the same way, by multiplying by the other $p_i^{r_i}$'s, $i<s$, we get elements $y_i$ such 
$A_n(y_i)=C$ and $F_n(y_i)=\Gamma_{p_i}$. \\ 
\indent (3) Predicate $Dk$. 
If $C$ is not the maximum of the groups $\Gamma_p$, 
then $Dk(C)$ does not hold  (Corollary \ref{corr227}). 
Assume that $C$ is the maximum of the $\Gamma_p$'s. For every $x\in \Gamma$ such that $C=A_n(x)$, $x$ does not 
belong to any $\Gamma_p$.  Therefore, $B_n(x)=\Gamma$. \\
\indent (4) Predicates $\beta_{p,m}(C_j)$. Let $x\in \Gamma$ such that $C_j=A_n(x)$ 
(i.e.  $x\in C_{j+1}\backslash C_j$). 
If $C_n(x)$ is not $p$-divisible, 
then, $[p]C_n(x)=p$. Hence, for every $r$, ${\rm dim}(p^rC_n(x)/p^{r+1}C_n(x))=1$. Otherwise, 
${\rm dim}(p^rC_n(x)/p^{r+1}C_n(x))=0$. 
It follows that if $m\geq 2$, then $\beta_{p,m}(C_j)$ does not hold and 
if $m=0$, then it holds. If $m=1$, then $\beta_{p,1}(C_j)$ holds if, and only if, $C_n(x)$ is not 
$p$ divisible. 
By Lemma \ref{lmajoutea}, $C_n(x)$ is not $p$-divisible if, 
and only if, $A_n(x)\subseteq \Gamma_p\subsetneq B_n(x)$, that is, 
$C_i\subseteq \Gamma_p\subsetneq C_{i+1}$\\
\indent (5) Predicates $\alpha_{p,k,m}(C)$. 
We let $x\in \Gamma$ such that $A_n(x)=C$. We assume that 
$\Gamma_{p_1}\subseteq \cdots \subseteq \Gamma_{p_{s-1}}\subsetneq \Gamma_{p_s}=\cdots =\Gamma_{p_t}
\subsetneq \Gamma_{p_{t+1}}\subseteq \cdots \subseteq \Gamma_{p_k}$ and $F_n(x)=\Gamma_{p_s}$. 
By (1), for $i>t$, $p_i^{r_i}$ divides 
$x$, and there is some $i\in \{s,\dots,t\}$ such that $p_i^{r_i}$ doesn't divide $x$. 
Let $y\in \Gamma$. Then 
$F_n(y)\subsetneq F_n(x)$ if, and only if, for every $i\geq s$, $p_i^{r_i}$ divides $y$. 
We let $n_1=p_{t+1}^{r_{t+1}}\cdots p_k^{r_k}$ and $n_2=p_s^{r_s}\cdots p_t^{r_t}$. Then 
$\{y \in \Gamma \; :\; F_n(y)\subsetneq F_n(x)\}=n_1n_2\Gamma$. In the same way, 
$\{y \in \Gamma \; :\; F_n(y)\subseteq F_n(x)\}=n_1\Gamma$. 

So $F_n^*(x)=n_1\Gamma/n_1n_2\Gamma$. 
Let $y_0\in n_1\Gamma\backslash n_1n_2\Gamma$, $p$ be a prime, $d$ be the greatest integer  such that 
$p^d$ divides $y_0$, and $k$ be a po\-si\-ti\-ve integer. 
Then $p^ky_0\in n_1n_2\Gamma$ if, and only if, there is $i\in 
\{s,\dots,t\}$ such that $p=p_i$, 
$k\geq r_i-d$, and for $i'\neq i$ in $\{s,\dots,t\}$, $p_{i'}^{r_{i'}}$ divides $y$. 
Consequently, if $p\notin \{p_s,\dots,p_t\}$, then $F_n^*(x)$ doesn't contain any nontrivial 
$p$-torsion element. 
If this holds, then ${\rm dim}
\displaystyle{\left({\rm Tor}_p (p^kF_n^*(x))/{\rm Tor}_p (p^{k+1}F_n^*(x))\right)}=0$. \\
\indent Assume that $p=p_i$ with $i\in \{s,\dots,t\}$. Then ${\rm Tor}_{p_i} (F_n^*(x))=
\left(n_1\frac{n_2}{p_i^{r_i}}\Gamma\right)/\left(n_1n_2\Gamma\right)$, and 
${\rm Tor}_{p_i} (p_iF_n^*(x))=\left(n_1\frac{n_2}{p_i^{r_i-1}}\Gamma\right)/\left(n_1n_2\Gamma\right)$. \\ 
Hence 
${\rm dim}\displaystyle{\left({\rm Tor}_{p_i} (F_n^*(x))/{\rm Tor}_{p_i} (p_iF_n^*(x))\right)}=1$. \\
\indent More generally, if $k<r_i$, then 
${\rm dim}
\displaystyle{\left({\rm Tor}_{p_i} (p_i^kF_n^*(x))/{\rm Tor}_{p_i} (p_i^{k+1}F_n^*(x))\right)}=1$. \\ 
\indent If $k\geq r_i$, then both of ${\rm Tor}_{p_i} (p_i^kF_n^*(x))$ and ${\rm Tor}_{p_i} (p_i^{k+1}F_n^*(x))$ are trivial,  
hence \\
${\rm dim}
\displaystyle{\left({\rm Tor}_{p_i} (p_i^kF_n^*(x))/{\rm Tor}_{p_i} (p_i^{k+1}F_n^*(x))\right)}=0$. \\ 
\indent By (1), $F_n(x)$ can be any element of $Sp_n(\Gamma)$ which is a subgroup of $C$, i.e.\ 
any $\Gamma_{p_i}$ where $i\in\{1,\dots,k\}$ and $\Gamma_{p_i}\subset C$. 
\end{proof}
\indent Now, Proposition \ref{corr222a} follows from Proposition \ref{prop222a}. 
\subsection{Theories of divisible abelian cyclically ordered groups.}\label{ssssec233}
Let $(G,R)$ be a cyclically ordered group. 
It follows from Theorems \ref{thspine} and \ref{relelimquant} that determining the first-order theory of 
$({\rm uw}(G),z_G)$ is also equivalent 
to determining  $Sp_n({\rm uw}(G))$'s and the type of $z_G$. 
So, Fran\c{c}ois Lucas proved that if $G$ is divisible, then the type of $z_G$ is determined 
by the $Sp_n({\rm uw}(G))$'s and the functions $f_{G,p}$.  
%
%
Our approach here is more explicit and we give the links with the theory of $(G,R)$. We first state a lemma. 
\begin{Lemma}(Lucas)\label{lem413} Let $(G,R)$ be a nonlinear divisible cyclically ordered abelian group. 
If $C_n(z_G)$ is discrete, then it is not archimedean.
\end{Lemma}
\begin{proof} Since $z_G$ is cofinal in $({\rm uw}(G),\leq_R)$ and $z_G\in B(z_G)\subseteq B_n(z_G)$, 
it follows that $B_n(z_G)={\rm uw}(G)$. Hence, the condition ``$C_n(z_G)$ archimedean'' is e\-qui\-va\-lent 
to $A_n(z_G)=A(z_G)=l(G)_{uw}$. Hence $C_n(z_G)=  {\rm uw}(G)/l(G)_{uw}$. 
Since $(G,R)$ is nonlinear, we have $U(G)\neq \{1\}$. Now, $(G,R)$ is divisible, hence $U(G)$ is divisible. 
In particular it is infinite, so it is dense. By Lemma \ref{lm12}, ${\rm uw}(G)/l(G)_{uw}$ is dense. 
\end{proof}
\begin{Proposition}\label{prop222} 
Let $(G,R)$ be a divisible nonlinear cyclically ordered abelian group. 
The spines of Schmitt of $({\rm uw}(G),\leq_R)$ are determined by $\mathcal{CD}(G)$. 
\end{Proposition}
\begin{proof} It follows from Lemma \ref{pclass} that, for every prime $p$, 
$[p]{\rm uw}(G)\in\{1,p\}$. If $[p]{\rm uw}(G)=1$, then ${\rm uw}(G)$ is $p$-divisible and 
${\rm uw}(G)_p=G_{uw}=(H_p)_{uw}$. Assume that $[p]{\rm uw}(G)=p$. 
Since $l(G)_{uw}$ is the greatest convex subgroup of ${\rm uw}(G)$ and 
this group is not $p$-divisible, we have ${\rm uw}(G)_p\subseteq l(G)_{uw}$. Therefore 
${\rm uw}(G)_p=(H_p)_{uw}$. 
Hence By Proposition \ref{corr222a}, the spines of $({\rm uw}(G),\leq_R)$ are 
determined by the set of $(H_p)_{uw}$'s, and if this set admits a maximal element $(H_p)_{uw}$, 
the property ${\rm uw}(G)/(H_p)_{uw}$ being discrete or dense. The preordered subset of the 
$(H_p)_{uw}$'s and $\mathcal{CD}(G)$ are isomorphic, and ${\rm uw}(G)/(H_p)_{uw}$ is discrete if, and 
only if, $G/H_p$ is (Remark \ref{rkcorresp}), and this is first-order definable 
(Lemma \ref{propHp} (3)). \\
\indent By Lemma \ref{propHp} (2), $H_p\subsetneq H_q$ depends on the first-order theory of $(G,R)$, hence 
the same holds for the predicate $<_S$ of $Sp_n({\rm uw}(G))$. 
\end{proof}
\begin{Lemma}\label{lm235} 
Let $(G,R)$ be a nonlinear cyclically ordered divisible abelian group and $p$ be a prime such that 
$G/H_p$ is discrete. Then for every $x\in l(G)_{uw}$ such that the class $x+(H_p)_{uw}$ is the smallest 
positive element of ${\rm uw}(G)/(H_p)_{uw}$ and every $r\in \NN\backslash \{0\}$, 
there exists $y\in {\rm uw}(G)$ such that $p^ry=x+f_{G,p}(r)z_G$. Furthermore, if there exists 
$y'\in {\rm uw}(G)$ and $k\in \ZZ$ such that $p^ry'=x+kz_{G}$, then $k$ is congruent to $f_{G,p}(r)$ modulo 
$p^r$. 
\end{Lemma}
\begin{proof} Let $r\in \NN\backslash \{0\}$ and $x\in l(G)_{uw}$ such that the class $x+(H_p)_{uw}$ is the smallest 
positive element of ${\rm uw}(G)/(H_p)_{uw}$. 
Since $G$ is divisible, its unwound ${\rm uw}(G)$ is divisible modulo $\langle z_G\rangle$. 
Hence there e\-xis\-ts $y'\in {\rm uw}(G)$ and an integer $k$ such that $p^ry'=x+kz_G$. 
Let $k=p^rk'+k''$, with $k''\in \{0,\dots,p^r-1\}$ be the euclidean division, and $y=y'-k'z_G$. Then 
$p^ry=x+ k''z_G$. 
Let $h=\bar{y}$ and $g=\bar{x}$. Then $g=p^rh$, hence, by the definition of $f_{G,p}(r)$, we have 
$\displaystyle{U(h)=e^{\frac{2if_{G,p}(r)\pi}{p^r}}}$. 
In the divisible hull of ${\rm uw}(G)$, we have 
$\displaystyle{y=\frac{1}{p^r}x+\frac{k''}{p^r}z_G}$. Hence 
$\displaystyle{U(h)=U\left(\overline{\frac{k''}{p^r}z_G}\right)=e^{\frac{2ik''\pi}{p^r}}}$. 
Consequently, $k''=f_{G,p}(r)$, and $k$ is congruent to $f_{G,p}(r)$ modulo $p^r$. 
\end{proof}
\indent The proof of the following theorem is similar to that of an analogous Lucas result, but much read in detail here. 
\begin{Theorem}\label{autresens}
Let $(G,R)$ be a divisible nonlinear cyclically ordered abelian group. 
If every $G/H_p$ is dense, then the type of $z_G$ in the language of 
Schmitt is entirely determined by $\mathcal{CD}(G)$. 
Otherwise, it is entirely determined by $\mathcal{CD}(G)$ and the mappings $f_{G,p}$, where $H_p$ is the 
maximum of the set $\{H_q\; :\; q\mbox{ prime}\}$. 
\end{Theorem}
\begin{proof}  
Let $\Psi(Z)$ be a formula in the language of ordered groups. By Theorem \ref{relelimquant}, there 
exist an integer $n$, a formula $\varphi_0(X_1,\dots,X_k,Y_1,\dots,Y_k)$ 
in the language $Sp_n$, some terms $t_1,\dots,t_k$ with variables $Z$ and a 
quantifier-free formula 
$\varphi_1(Z)$ in the language defined before Theorem \ref{relelimquant} such that for 
every abelian ordered group $(\Gamma,\leq)$ and every $x$ in $\Gamma$ we have $(\Gamma,\leq)\models 
\Psi(\overrightarrow{x})$ if, and only if, 
$$Sp_n(\Gamma)\models \varphi_0(A_n(t_1(x)),\dots,A_n(t_k(x)),
F_n(t_1(x)),\dots,F_n(t_k(x)))) 
\mbox{ and }(\Gamma,\leq ) \models \varphi_1(\overrightarrow{x}).$$
\indent 
The only terms with variable $z_G$ look like $\alpha z_G$, where $\alpha\in \ZZ\backslash \{0\}$. 
Set $t_1(z_G)=\alpha_1z_G,\dots,t_k(z_G)=\alpha_kz_G$.  
Let $(n,\alpha_i)$ be the gcd of $n$ and $\alpha_i$. By (1) of Proposition \ref{prop222a}, we have 
$A_n(\alpha_i z_G)=A_n(z_G)$. By (c) of Lemma \ref{npropav411}, 
$F_n(\alpha_i z_G)=F_{\frac{n}{(n,\alpha_i)}}\left(\frac{\alpha_i}{(n,\alpha_i)}z_G\right)=
F_{\frac{n}{(n,\alpha_i)}}(z_G)$. 
Therefore, $\varphi_0$ can be written as 
$$\displaystyle{\varphi_0\left(A_n(z_G), F_{\frac{n}{(n,\alpha_1)}}(z_G), \dots,
F_\frac{n}{(n,\alpha_k)}(z_G)\right)}.$$ 
\indent Since $z_G$ is cofinal in ${\rm uw}(G)$, it follows from (1) of Proposition \ref{prop222a} 
that $A_n(z_G)$ is the greatest $(H_p)_{uw}$ such that $p$ divides $n$. 
If $z_G$ is $n$-divisible, then $\displaystyle{F_{\frac{n}{(n,\alpha_i)}}(z_G)}=\emptyset$.  
If $z_G$ is not divisible $n$-divisible, then by 
(1) of Proposition \ref{prop222a} $\displaystyle{F_{\frac{n}{(n,\alpha_i)}}(z_G)}$ 
is the greatest $(H_p)_{uw}$ such that $p$ divides $\displaystyle{\frac{n}{(n,\alpha_i)}}$.  
So all these subsets are determined by the preordered set of the $(H_p)_{uw}$'s. \\
\indent 
Now, we look at the quantifier-free formula $\varphi_1(z_G)$ in the language 
$(0,+,\leq_R,M_{n,k},D_{p,r,i},E_{p,r,i})$, where $n$, $k$, $r$, $i$ belong to $\NN$ and $p$ is a prime. \\ 
\indent (1) The formulas which contain only the symbols $0$, $+$, $\leq_R$ can be seen as formulas of $\ZZ$. 
Hence they do not depend on $z_G$ (for example, $\alpha z_G\geq_R (0,0_G)$ is e\-qui\-va\-lent to $\alpha \geq 0$). \\ 
\indent (2) 
By Lemma \ref{lem413}, if $C_n(z_G)$ is discrete, then it is not archimedean. Since the class of $z_G$ is 
cofinal, $z_G+A_n(z_G)$ does not belong to the smallest convex subgroup of $C_n(G)$ (which contains the 
smallest positive element). Hence $M_{n,k}(k'z_G)$ nether holds.  \\ 
\indent (3) 
Predicates $D_{p,r,i}(\alpha z_G)$. Assume that $z_G\notin p\, {\rm uw}(G)$. Hence 
$\alpha z_G\in p^r{\rm uw}(G)\Leftrightarrow 
\alpha\in p^r \ZZ$. Assume that $\alpha\notin p^r\ZZ$, say $\alpha=p^s\alpha'$, with 
$0\leq s<r$ and $(\alpha',p)=1$. By (c) of Lemma \ref{npropav411}, $F_{p^r}(\alpha z_G)=F_{p^{r-s}}(z_G)$. 
By (1) of Proposition \ref{prop222a}, $F_{p^{r-s}}(z_G)=(H_p)_{uw}$. 
Hence $F_{p^r}(x)\subsetneq F_{p^r}(z_G)\Rightarrow F_{p^r}(x)=\emptyset$, i.e.\ $x\in p^r{\rm uw}(G)$. 
Let $x\in p^r {\rm uw}(G)$. For $i\leq r$ we have: 
$$\alpha z_G-x\in p^i{\rm uw}(G)\Leftrightarrow \alpha z_G\in p^i{\rm uw}(G) \Leftrightarrow 
\alpha\in p^i\ZZ\Leftrightarrow i\leq s,$$
which does not depend on $(G,R)$.\\ 
\indent If this holds, then by (1) of Proposition \ref{prop222a} we have 
$$\displaystyle{F_{p^r}\left(\frac{\alpha z_G-x}{p^i}\right)=\emptyset}\mbox{ or }
\displaystyle{F_{p^r}\left(\frac{\alpha z_G-x}{p^i}\right)=(H_p)_{uw}}.$$
\indent In any case, 
$\displaystyle{F_{p^r}\left(\frac{\alpha z_G-x}{p^i}\right)\subseteq F_{p^r}(z_G)}$. 
So, $D_{p,r,i}(\alpha z_G)$ holds. \\
\indent Therefore, the formula $D_{p,r,i}(\alpha z_G)$ is determined by $\mathcal{CD}(G)$. \\ 
\indent (4) 
Predicates $E_{p,r,k}(\alpha z_G)$ (there exists some $x\in {\rm uw}(G)$ such that $A_{p^r}(x)=
F_{p^r}(\alpha z_G)$, 
$C_{p^r}(x)$ is discrete, the class of $x$ is the smallest positive element, and 
$F_{p^r}(\alpha z_G-kx)\subsetneq F_{p^r}(\alpha z_G)$). 
By Corollary \ref{corr227}, if $C_{p^r}(x)$ is discrete, then 
$B_{p^r}(x)={\rm uw}(G)$ and $A_{p^r}(x)$ is the maximum of the set of all the $(H_p)_{uw}$'s. 
This is determined by $\mathcal{CD}(G)$ and the set of all primes $p$ such that 
$G/H_p$ is discrete. \\
\indent We assume that ${\rm uw}(G)/(H_p)_{uw}$ is discrete. Then ${\rm uw}(G)$ is not $p$ divisible, 
so $G$ is not c-$p$-divisible. In particular $z_G$ is not $p$-divisible (Lemma \ref{cdiv}). 
We let $x$ with class in the smallest positive element. \\
\indent If $\alpha z_G\in p^r{\rm uw}(G)$, then we already saw that $F_{p^r}(\alpha z_G)=\emptyset$. It 
follows that $E_{p,r,k}(\alpha z_G)$ doesn't hold, since the sets $A_{p^r}(x)$ are nonempty. This is determined by $\mathcal{CD}(G)$. \\
\indent Otherwise, we have $F_{p^r}(\alpha z_G)=(H_p)_{uw}$, and the last condition is e\-qui\-va\-lent to 
$F_{p^r}(\alpha z_G-kx)=\emptyset$. This in turn is e\-qui\-va\-lent to $\alpha z_G-kx\in p^r{\rm uw}(G)$. 
By Lemma \ref{lm235}, there exists $y\in {\rm uw}(G)$ such that 
$p^ry=x+f_{G,p}(r)z_G$. Since $z_G$ is not $p$-divisible, we have: 
$$\alpha z_G-kx\in p^r{\rm uw}(G)\Leftrightarrow \alpha z_G-k f_{G,p}(r)z_G\in p^r{\rm uw}(G)
\Leftrightarrow \alpha -k f_{G,p}(r)\in p^r\ZZ, $$
\indent which is entirely determined by the function $f_{G,p}$. 
\end{proof}
\indent Now, Theorem \ref{thm221} follows from Proposition \ref{propunsens} and Theorem \ref{autresens}. \\[2mm]
\indent Note that a divisible abelian cyclically ordered group $G$ is c-regular if, and only if, 
for every prime $p$ we have either $H_p=\{0_G\}$ or $H_p=G$. 
If $H_p=G$ for every $p$, then it is c-divisible, and if $H_p=\{0_G\}$ 
for every $p$, then it is torsion-free c-regular. 
\subsection{Cyclic orders on the additive group $\QQ$ of rational numbers.}\label{sec4}
We describe the cyclic orders on $\QQ$, we deduce a characterization of the families of functions 
$f_{G,p}$, and we prove Theorem \ref{uncount}. Note that there is only one linear order on $\mathbb{Q}$. \\ 
\indent The description of the cyclic orders on $\QQ$ will give rise to the characterization of all 
functions $f_{G,p}$ because of the following proposition. 
\begin{Proposition}\label{prop215}(\cite[Proposition 4.38]{Le})  
Let $(G,R)$ be an abelian cyclically ordered group which is not c-archimedean. Then $(G,R)$ is discrete and 
divisible if, and only if, there exists a discrete cyclic order $R'$ on the group $\QQ$ such that $(\QQ,R')$ 
is an elementary substructure of $(G,R)$. 
\end{Proposition}
\indent Note that in the statement of \cite[Proposition 4.38]{Le} the word ``c-regular'' 
is redundant, since by Proposition \ref{prop211} 
every discrete divisible abelian cyclically ordered group is c-regular. \\
\indent The description of the cyclic orders on $\QQ$ is based on the following mappings. 
\begin{Fact}\label{fact417} 
For every prime number $p$, let $f_p$ be a mapping from $\NN$ in 
$\{0,\dots,p-1\}$. Then one can define in a unique way a mapping from $\NN\backslash \{0\}$ to $\NN$ 
in the following way. 
Set $f(1)=0$. For $p$ prime and $r\in \NN\backslash \{0\}$, 
$f(p^r)=f_p(1)+pf_p(2)+\cdots+p^{r-1}f_p(r)$. For $n=p_1^{r_1}\cdots p_k^{r_k}$ with 
$r_1,\dots, r_k$ in $\NN\backslash \{0\}$, let $f(n)$ be the unique 
integer in $\{0,\dots,n-1\}$ which is congruent to every $f\left(p_i^{r_i}\right)$ modulo $p_i^{r_i}$ 
(the existence of $f(n)$ follows from the Chinese remainder theorem). 
\end{Fact}
\begin{Proposition}\label{prop418} Let $R$ be a ternary relation on $\QQ$. Then $R$ is a 
cyclic order if, and only if, there exist  mapping from $\NN\backslash \{0\}$ to $\NN$, defined as in 
Fact \ref{fact417}, $\theta\in \; [0,2\pi[$ and a non-negative $a\in\RR$ such that 
$(\QQ,R)$ is c-isomorphic to the subgroup of the lexicographic product $\UU\overrightarrow{\times} \QQ$ 
generated by 
$$\displaystyle{\left\{\left({\rm exp}\left(i\frac{\theta+2f(n)\pi}{n}\right),\frac{a}{n}\right)\; :\;  n\in 
\NN\backslash \{0\}\right\}}.$$
\end{Proposition}
\indent The proof of this proposition consists of Lemmas \ref{lem418}, \ref{a236}, \ref{lm238}, \ref{lm239}. 
\begin{Lemma}(Lucas)\label{lem418} Let $R$ be a nonlinear cyclic order on $\QQ$. 
The cyclically ordered group $(\QQ,R)$ embeds in the lexicographic product $\UU\overrightarrow{\times} \QQ$. 
\end{Lemma}
\begin{proof} By \cite[p.\ 161]{Sw}, there is a linearly ordered group $L$ such that $(\QQ,R)$ embeds in 
the lexicographic product $\UU\overrightarrow{\times} L$. 
Let $x\mapsto (\varphi_1(x),\varphi_2(x))$ be this embedding. Then 
$\varphi_1$ and $\varphi_2$ are group isomorphisms. It follows that they are $\QQ$-linear. In particular,  the image of 
$\varphi_2$ is $\QQ\!\cdot\!\varphi_2(1)$. So either $\QQ\!\cdot\!\varphi_2(1)$ is trivial, or it is isomorphic to $\QQ$. 
In any case $(\QQ,R)$ embeds in $\UU\overrightarrow{\times} \QQ$. 
\end{proof}
\begin{Lemma}\label{a236} 
Let $G$ be a subgroup of $\UU\overrightarrow{\times} \QQ$ which is isomorphic to $\QQ$ in the language of groups. 
There exist $\theta\in \RR$, $0\leq a\in \QQ$, with $a=0\Rightarrow \theta\notin \QQ\!\cdot\!\pi$, such that 
$G$ is generated by 
$$\displaystyle{\left\{\left({\rm exp}\left(i\frac{\theta+2f(n)\pi}{n}\right),\frac{a}{n}\right)\; :\;  n\in 
\NN\backslash \{0\}\right\}},$$  
where $f$ is a mapping from $\NN\backslash \{ 0\}$ in $\NN$ which satisfies: \\
(*) for every $n\in \NN\backslash \{0\}$, $f(n)\in\{0,\dots,n-1\}$, \\
(**) for every $m$, $n$ in $\NN\backslash \{0\}$, $f(n)$ is the remainder of the euclidean division of $f(mn)$ by $n$. 
\end{Lemma}
\begin{proof} Let $x\mapsto (\varphi_1(x),\varphi_2(x))$ be the group isomorphism between $\QQ$ and $G$. 
We pick an $x_0\neq 0$ in $\QQ$ and we let 
$({\rm exp}(i\theta),a)=(\varphi_1(x_0),\varphi_2(x_0))$. Then $G$ is generated by the set 
$$\displaystyle{\left\{ \left(\varphi_1\left(\frac{x_0}{n}\right),\varphi_2\left(\frac{x_0}{n}\right) \right) 
\; :\; n\in \NN\backslash \{0\}\right\}}.$$ 
\indent Now, $\varphi_2\left(\frac{x_0}{n}\right)=\frac{a}{n}$, and 
there exists a unique $f(n)\in \{0,\dots, n-1\}$ such that 
$$\varphi_1\left(\frac{x_0}{n}\right)=
{\rm exp}\left(i\frac{\theta+2f(n)\pi}{n}\right).$$ 
\indent Consequently, $f$ satisfies (*). \\
\indent If $a<0$, then we can take $-x_0$ instead 
of $x_0$, so we can assume that $a\geq 0$. Now, $G$ is torsion-free, so if $a=0$, then,  
for every $m\in \ZZ\backslash \{0\}$ and every $n\in \NN\backslash \{0\}$, 
${\rm exp}\left(im\frac{\theta+2f(n)\pi}{n}\right)\neq 1$. This implies that $\theta \notin \QQ\!\cdot\!\pi$. Conversely, 
if $a\neq 0$ or $\theta \notin \QQ\!\cdot\!\pi$, then 
$\left({\rm exp}\left(im\frac{\theta+2f(n)\pi}{n}\right),\frac{m}{n}a\right)=(1,0)\Rightarrow m=0$. \\
\indent Let $m$, $n$ in $\NN\backslash \{0\}$. Then $\varphi_1\left(\frac{x_0}{n}\right)=m\varphi_1\left(\frac{x_0}{mn}\right)$. 
Hence $\frac{\theta+2f(n)\pi}{n}-m\frac{\theta+2f(mn)\pi}{mn}\in 2\pi\ZZ$. So 
$f(n)-f(mn)\in n\ZZ$, where $0\leq f(n)<n$ and $0\leq f(mn)<mn$. Hence $f(n)$ is the remainder of the euclidean 
division of $f(mn)$ by $n$. This proves that $f$ satisfies (**). 
\end{proof}
\begin{Lemma}\label{lm238} 
For every $\theta\in \RR$, $0\leq a\in \QQ$, with $a=0\Rightarrow \theta\notin \QQ\!\cdot\!\pi$, 
and every mapping $f$ from $\NN\backslash \{ 0\}$ in $\NN$ which satisfies 
(*) and (**), the set 
$$\displaystyle{G=
\left\{\left({\rm exp}\left(im\frac{\theta+2f(n)\pi}{n}\right),\frac{m}{n}a\right)\; :\; m\in \ZZ,\; n\in 
\NN\backslash \{0\}\right\}}$$
is a subgroup of $\UU\overrightarrow{\times} \QQ$, which is 
isomorphic to $\QQ$ in the language of groups. 
\end{Lemma}
\begin{proof} Let $m$, $m'$ in $\ZZ$, $n$, $n'$ in $\NN\backslash\{0\}$ be such that 
$$\displaystyle{\left({\rm exp}\left(im\frac{\theta+2f(n)\pi}{n}\right),\frac{m}{n} a\right)=
\left({\rm exp}\left(im'\frac{\theta+2f(n')\pi}{n'}\right),\frac{m'}{n'}a\right)}.$$ 
Then $\left(\frac{m}{n}-\frac{m'}{n'}\right)\theta+2\left(\frac{mf(n)}{n}-\frac{m'f(n')}{n'}\right)\pi
\in 2\pi\ZZ$ and 
$\left(\frac{m}{n}-\frac{m'}{n'}\right)a=0$. If $a\neq 0$, then $\frac{m}{n}=\frac{m'}{n'}$. 
If $a=0$, then by hypothesis $\theta\notin \QQ\pi$. Hence 
$\left(\frac{m}{n}-\frac{m'}{n'}\right)\theta=0$, and $\frac{m}{n}=\frac{m'}{n'}$. \\
\indent Therefore, there is a one-to-one mapping between $\QQ$ and $G$: for $m\in \ZZ$ and 
$n\in \NN\backslash \{ 0\}$, 
$$\displaystyle{\varphi\left(\frac{m}{n}\right)= 
\left({\rm exp}\left(im\frac{\theta+2f(n)\pi}{n}\right),\frac{m}{n} a\right)}.$$
\indent Since $\displaystyle{\frac{m}{n}+\frac{m'}{n'}=\frac{mn'+m'n}{nn'}}$, to prove that 
$\varphi$ is an isomorphism in the language of groups, it suffices to prove that 
$\displaystyle{\frac{(mn'+m'n)f(nn')}{nn'}\!\cdot\!2\pi}$ is congruent to 
$\displaystyle{\frac{mn'f(n)+m'nf(n')}{nn'}\!\cdot\! 2\pi}$ modulo $2\pi$. 
By (**), the euclidean divisions of $f(nn')$ by $n$ and by $n'$ can be written as 
$f(nn')=qn+f(n)$ and $f(nn')=q'n'+f(n')$. Therefore 
$$\displaystyle{\begin{array}{rcl}
\frac{mn'f(n)+m'nf(n')}{nn'}\!\cdot\! 2\pi&=&\frac{(mn'+m'n)f(nn')-nn'(mq+m'q')}{nn'}\!\cdot\! 2\pi\\
&=&\frac{(mn'+m'n)f(nn')}{nn'}\!\cdot\! 2\pi-2(mq+m'q')\pi.
\end{array}}$$
It follows that $\varphi$ is an isomorphism in the language of groups, and so its image $G$ is a subgroup of 
$\UU\overrightarrow{\times} \QQ$. 
\end{proof}
\begin{Lemma}\label{lm239} 
There is a one-to-one correspondence between the set of mappings from $\NN\backslash \{ 0\}$ in $\NN$ 
which satisfies (*) and (**), and the set of families $\{ f_p\; : \; p \mbox{ a prime}\}$, where the $f_p$'s are 
mappings from $\NN\backslash \{0\}$ to $\{0,\dots,p-1\}$. 
\end{Lemma}
\begin{proof} 
For every prime number $p$, let $f_p$ be a mapping from $\NN$ in 
$\{0,\dots,p-1\}$, and let $f$ be the mapping defined in Fact \ref{fact417}. Then $f$ satisfies (*). 
Let $m=p_1^{r_1}\cdots p_k^{r_k}$ and $n=p_1^{s_1}\cdots p_k^{s_k}$ (where one of $r_i$ or $s_i$ can be 
equal to $0$). We let $i\in \{1,\dots k\}$. By the construction of $f\left(p_i^{r_i+s_i}\right)$ and 
$f\left(p_i^{s_i}\right)$, it follows 
that $p_i^{s_i}$ divides $f\left(p_i^{r_i+s_i}\right)-f\left(p_i^{s_i}\right)$. Now, since 
$f\left(p_i^{r_i+s_i}\right)$ is congruent to 
$f(mn)$ modulo $p_i^{r_i+s_i}$ and $f\left(p_i^{s_i}\right)$ is congruent to $f(n)$ modulo $p_i^{s_i}$, we deduce that 
$f(mn)-f(n)$ is congruent to $f\left(p_i^{r_i+s_i}\right)-f\left(p_i^{s_i}\right)$ modulo $p_i^{s_i}$. Therefore 
$p_i^{s_i}$ divides $f(mn)-f(n)$. Consequently, $n$ divides $f(mn)-f(n)$. Since $f(n)<n$, we conclude that 
$f(n)$ is the remainder of the euclidean division of $f(mn)$ by $n$. This proves that $f$ satisfies (**). \\ 
\indent Conversely, let $f$ be a mapping from $\NN\backslash \{ 0\}$ in $\NN$ which satisfies (*) and (**). 
For every prime $p$ we set $f_p(1)=f(p)$, and for $r\in \NN\backslash \{0\}$ we set 
$\displaystyle{f_p(r)=\frac{f\left(p^r\right)-f\left(p^{r-1}\right)}{p^{r-1}}}$. By (**), $f_p(r)$ is an integer 
(it is the quotient of the euclidean division of $f(p^r)$ by $p^{r-1}$). 
Now, $f\left(p^r\right)<p^r$, hence $\displaystyle{\frac{f\left(p^r\right)-f\left(p^{r-1}\right)}{p^{r-1}}<p}$, so it 
belongs to $\{0,\dots,p-1\}$. 
By construction, $f\left(p^r\right)=f_p(1)+pf_p(2)+\cdots+p^{r-1}f_p(r)$. We show that 
$f\left(p_1^{r_1}\cdots p_k^{r_k}\right)$ is the unique integer in $\{0,\dots,p_1^{r_1}\cdots p_k^{r_k}-1\}$ 
which is congruent to every $f\left(p_i^{r_i}\right)$ modulo $p_i^{r_i}$. Since $0\leq f\left(p_1^{r_1}\cdots p_k^{r_k}\right)
<p_1^{r_1}\cdots p_k^{r_k}$, it suffices to prove that every $p_i^{r_i}$ divides 
$f\left(p_1^{r_1}\cdots p_k^{r_k}\right)-f\left(p_i^{r_i}\right)$. Now, this follows from (**). 
\end{proof} 
\begin{Remark}\label{rk423} Let $(G,R)$ be a cyclically ordered group and $p$ be a prime such that 
$G/H_p$ is discrete. By Lemma \ref{rk240}, for every positive integers $m$, $n$, $f_{G,p}(n)$ is the 
remainder of the euclidean division of $f_{G,p}(m+n)$ by $p^n$ and 
$p$ doesn't divide $f_{G,p}(n)$. Hence in the same way as in the proof of Lemma \ref{lm238} there is a 
function $f_p: \; \NN\backslash\{0\} \rightarrow \{0,\dots, p-1\}$ such that 
for every $n\in \NN\backslash \{0\}$, 
$f_{G,p}(n)=f_p(1)+pf_p(2)+\cdots+p^{n-1}f_p(n)$. Since $p$ doesn't divide $f_{G,p}(n)$, we have $f_p(1)\neq 0$. 
\end{Remark}
\indent Note that the linear part of the cyclically ordered group $\UU\overrightarrow{\times} \QQ$
is $\{1\}\overrightarrow{\times} \QQ$, so $U(\UU\overrightarrow{\times} \QQ)$ is 
isomorphic to $\UU$. 
The unwound of $\UU\overrightarrow{\times} \QQ$ is isomorphic 
to $\RR\overrightarrow{\times} \QQ$. The convex subgroups of this unwound are $\{(0,0)\}$, 
$\{0\}\overrightarrow{\times} \QQ$, which is isomorphic to $l(\UU\overrightarrow{\times} \QQ)_{uw}$, 
and $\RR\overrightarrow{\times} \QQ$. So, if $G$ is a subgroup of $\UU\overrightarrow{\times} \QQ$, 
then $\mathcal{CD}(G)$ contains at most three elements. \\[2mm]
\indent In the following 
we let $G$ be a nonlinear subgroup of $\UU\overrightarrow{\times} \QQ$, which 
is isomorphic to $\QQ$ in the language of groups, and $\theta\in \RR$, $0\leq a\in \QQ$, $f\; :\; 
\NN\backslash\{0\}\rightarrow \NN$ as in Lemma \ref{a236}, and $f_p$ ($p$ prime) be the family of functions defined 
in Lemma \ref{lm239}. 
\begin{Remark}\label{rek423} If $\theta\in \QQ\pi$, then there is $\frac{m}{n}>0$ in $\QQ$ such that 
$\varphi_1\left(\frac{m}{n}x_0\right)\in 2\pi\ZZ$. So, by taking $\frac{m}{n}x_0$ instead of $x_0$ in the 
proof of Lemma \ref{a236} we can assume that $\theta=0$. 
\end{Remark}
\begin{Proposition}\label{lem423} The cyclically ordered group $G$ is c-archimedean if, and only if,  
$\theta\notin \QQ\pi$. If this hods, then it is c-isomorphic to a subgroup $G'$ of $\UU\overrightarrow{\times} \QQ$ 
such that $\theta'=\theta$, $f'=f$ and $a'=0$. 
\end{Proposition}
\begin{proof} 
Assume that $\theta\notin \QQ\!\cdot\!\pi$. Then, 
$\displaystyle{{\rm exp}\left(im\frac{\theta+2f(n)\pi}{n}\right)=1\Leftrightarrow m\frac{\theta+2f(n)\pi}{n}\in 2\pi\ZZ}$. 
This implies $m\theta \in \QQ\pi$, hence $m=0$. It follows that 
$\displaystyle{{\rm exp}\left(im\frac{\theta+2f(n)\pi}{n}\right)=1}$ if, and only if, $m=0$. Hence $l(G)=\{(1,0)\}$ and 
$(G,R)$ is c-archimedean. Since it is infinite, it is dense. \\
\indent Assume that $\theta \in \QQ\!\cdot\!\pi$. Recall that this implies $a>0$. 
By Remark \ref{rek423} we can assume that 
$\theta=0$. Then $l(G)$ contains the $\left(1,\frac{ma}{n}\right)$'s where 
$m\frac{f(n)}{n}\in \ZZ$. In particular, $l(G)\neq \{(1,0)\}$, so 
$G$ is not c-archimedean. \\
\indent Now, assume that $\theta\notin \QQ\pi$ and let $G'$ be the subgroup of $\UU\overrightarrow{\times} \{0\}$ 
generated by the set $$\displaystyle{\left\{ \left(\varphi_1\left(\frac{x_0}{n}\right),0 \right) 
\; :\; n\in \NN\backslash \{0\}\right\}}.$$ 
\indent Then $\theta'=\theta$ and $f'=f$. We proved above that $\varphi_1$ is 
one-to-one, hence the groups $G$ and $G'$ are isomorphic. Furthermore, since both of $\varphi_1$ and $\varphi_2$ 
are one-to-one, for any $x_1$ $x_2$ in $\QQ$ we have $(\varphi_1(x_1),\varphi_2(x_1))\neq (\varphi_1(x_2),\varphi_2(x_2))
\Leftrightarrow x_1\neq x_2$.  
Now, let $x_1$, $x_2$, $x_3$ in $\QQ$ such that 
$(\varphi_1(x_1),\varphi_2(x_1))\neq (\varphi_1(x_2),\varphi_2(x_2))\neq (\varphi_1(x_3),\varphi_2(x_3))
\neq (\varphi_1(x_1),\varphi_2(x_1))$. This equivalent to $\varphi_1(x_1)\neq \varphi_1(x_2)\neq \varphi_1(x_3)
\neq \varphi_1(x_1)$. Hence $R((\varphi_1(x_1),\varphi_2(x_1)),(\varphi_1(x_2),\varphi_2(x_2)),(\varphi_1(x_3),\varphi_2(x_3)))$ 
holds in $\UU\overrightarrow{\times} \QQ$ if, and only if, $R(\varphi_1(x_1),\varphi_1(x_2),\varphi_1(x_3))$ hods 
in $\UU$. Therefore, $G$ and $G'$ are c-isomorphic. 
\end{proof}
\indent By \cite[Lemma 5.1, Theorem 5.3]{JPR}, there is one and only one c-embedding of each c-archimedean cyclically 
ordered group in $\UU$. Hence if $\QQ \frac{\theta}{\pi}\cap \QQ\frac{\theta'}{\pi}=\emptyset$, 
then $(G,R)$ and $(G',R')$ are not c-isomorphic. Hence there are $2^{\aleph_0}$ non isomorphic c-archimedean cyclic orders 
on $\QQ$. Now, since any c-archimedean cyclically ordered group is c-regular, by Theorem \ref{thm215}, 
all c-archimedean dense torsion-free divisible abelian cyclically ordered 
groups are elementarily e\-qui\-va\-lent. \\[2mm] 
\indent Now, we turn to the non-c-archimedean case. By Lemma \ref{a236} and Remark \ref{rek423} 
we can assume that $\theta=0$ and $a>0$. 
\begin{Lemma}\label{lmav424} Assume that $G$ is not c-archimedean, and let $p$ be a prime. 
Then $H_p=l(G)$ if, and only if, $f_p$ is the $0$ mapping. Otherwise, $H_p=\{(1,0)\}$. 
\end{Lemma}
\begin{proof} Saying that $H_p=l(G)$ is equivalent to saying that $l(G)$ is a $p$-divisible group. 
This is turn is equivalent to saying that for every $r\in \NN$ the element $(1,a)$ is $p^r$-divisible within $l(G)$, that is, 
there is $\frac{m}{n}\in \QQ$ such that ${\rm exp}\left(im\frac{2f(n)\pi}{n}\right)=1$ and $p^r\frac{m}{n}a=a$. 
Now, ${\rm exp}\left(im\frac{2f(n)\pi}{n}\right)=1$ and $p^r\frac{m}{n}a=a\Leftrightarrow \frac{m}{n}=\frac{1}{p^r}$ 
and ${\rm exp}\left(i\frac{2f(p ^r)\pi}{p^r}\right)=1$. Since $0\leq f(p^r)\leq p^r-1$, this in turn is equivalent 
to $\frac{m}{n}=\frac{1}{p^r}$ and $f(p^r)=0$. One can prove by induction that $\forall r\in \NN\backslash \{0\}$ 
$f(p^r)=0$ is equivalent to $\forall n\in \NN\backslash \{0\}\;f_p(n)=0$. \\
\indent Since $l(G)$ is isomorphic to $\QQ$, its only proper subgroup is $\{(1,0)\}$, hence $H_p\neq l(G)\Leftrightarrow 
H_p=\{(1,0)\}$. 
\end{proof} 
\begin{Corollary}\label{prop424} For every subset $\mathcal{P}'$ of the set $\mathcal{P}$ 
of primes we can chose $f$ so that, 
for every $p\in\mathcal{P}'$,  
$H_p=l(G)$, and, for every $p\in \mathcal{P}\backslash \mathcal{P}'$, $H_p=\{(1,0)\}$. 
In particular, there are $2^{\aleph_0}$ non elementarily equivalent dense nonlinear cyclic orders on $\QQ$. 
\end{Corollary}
\begin{proof} For every  $p\in \mathcal{P}'$ we let $f_p$ be the $0$ mapping and for every 
$p\in \mathcal{P}\backslash \mathcal{P}'$ we let $f_p$ be a function from $\NN\backslash \{0\}$ to 
$\{0,\dots,p-1\}$ such that $f_p(1)\neq 0$. We let $f$ be defined in the same way as 
in Lemma \ref{lm239}. By Lemma \ref{lmav424}  
for every $p\in \mathcal{P}$ we have $H_p=l(G)$ and for every $p\in \mathcal{P}\backslash \mathcal{P}'$ 
we have $H_p=\{(1,0)\}$. \\
\indent Therefore, if $\mathcal{P}'$ is nonempty, then there is $p$ such that $l(G)$ is $p$-divisible, 
so it it dense. So the cyclic order is dense. Since there are $2^{\aleph_0}$ nonempty subsets of 
$\mathcal{P}$, by Theorem \ref{thm221} there $2^{\aleph_0}$ pairwise  non elementarily equivalent 
dense nonlinear cyclic orders on $\QQ$. 
\end{proof}
\indent The following provides a necessary and sufficient condition for $G$ being discrete. 
\begin{Proposition}\label{lm241} \begin{enumerate}
\item $G$ is discrete if, and only if, for every prime $p$ the mapping $f_p$ is not the zero mapping and there is 
only a finite number of primes $p$ such that $f_p(1)=0$. If this holds, then we can choose the real number $a$ such that 
for every prime $p$ we have $f_p(1)\neq 0$. 
\item Assume that for every prime $p$ we have $f_p(1)\neq 0$.
Then, for every prime $p$ and every $n\in \NN$, $H_p=\{0\}$ and $f(p^n)=f_{G,p}(n)$. 
\item For every family of functions $f_p: \; \NN\ \backslash \{0\}\rightarrow \{0,\dots,p-1\}$ 
with $f_p(1)\neq 0$ ($p$ prime), there is a discrete cyclically ordered group isomorphic to $\QQ$ 
such that the family of $f_{G,p}$ is constructed as in Remark \ref{rk423}. Therefore, 
there $2^{\aleph_0}$ pairwise  non elementarily equivalent discrete cyclic orders on $\QQ$. 
\end{enumerate}
\end{Proposition}
\begin{proof} 
(1) Assume that for every prime $p$ we have $f_p(1)\neq 0$. 
Since for every prime $p$ we have $f_p(1)\neq 0$, the integers $f(p)=f_p(1)$ and $p$ are coprime. 
For every $r\geq 2$, $f(p ^r)$ and $p$ are coprime, since 
$f(p)$ is the remainder of the euclidean division of $f(p^r)$ by $p^{r-1}$. 
Let $p_1<\cdots<p_k$ be primes, $r_1,\dots ,r_k$ be positive integers, and $n=p_1^{r_1}\cdots p_k^{r_k}$. Further, let  
$u$, $v$ in $\ZZ$ such that $uf(p_1^{r_1})+vp_1=1$, and $q\in \NN$ such that $f(n)=qp_1^{r_1}+f(p_1^{r_1})$. Then 
$1=uf(n)+p_1\left(v-qup_1^{r_1-1}\right)$. Therefore $f(n)$ and $p_1$ are coprime. The same holds with $p_2,\dots, p_k$, hence 
$f(n)$ and $n$ are coprime. Therefore, 
if $\displaystyle{m\frac{2\pi f(n)}{n}\in 2\pi\ZZ}$, then $n$ divides $m$. Consequently, either 
$\displaystyle{\frac{m}{n}a\leq 0}$, or 
$\displaystyle{\frac{m}{n}a\geq a}$. It follows that $(G,R)$ is discrete, 
with smallest positive element $(1,a)$. Conversely, assume that $(1,a)$ is the smallest element of $l(G)$, 
and let $p$ be a prime. Since $\left(1,\frac{a}{p}\right)<(1,a)$, $l(G)$ doesn't contain the element $\left(1,\frac{a}{p}\right)$ 
of $\UU\overrightarrow{\times}\QQ$. Therefore, $\displaystyle{e^{i\frac{2\pi f(p)}{p}}\neq 1}$, which is equivalent to 
$f(p)\neq 0$.  \\ 
\indent 
Assume that $G$ is discrete. Then for every prime $p$ we have $H_p=\{(1,0)\}$. By Lemma \ref{lmav424}, 
the functions $f_p$ are different from the zero mapping. \\
\indent Assume that no function $f_p$ is the zero 
mapping, but that there are infinitely many primes $p$ such that $f_p(1)=0$. Let $n$ be a positive integer such that 
$\left(1,\frac{a}{n}\right)$ belongs to the positive cone of $l(G)$. This is equivalent to $f(n)=0$. Indeed, since 
$0\leq f(n)\leq n-1$, $\frac{2f(n)\pi}{n}\in2 \ZZ\pi\Leftrightarrow f(n)=0$. Now, we let $p$ be a prime 
which does not divide $n$ and such that $f_p(1)=0$. Since $f(np)$ is congruent to $f(n)$ modulo $n$ and to 
$f(p)=f_p(1)$ modulo $p$, we have $f(np)=0$. Hence $(1,\frac{a}{np})\in l(G)$, with $(1,0)<(1,\frac{a}{np})<
(1,\frac{a}{n})$. Therefore $l(G)$ has no smallest positive element, and $l(G)$ is dense. This proves that 
$G$ is dense. \\
\indent It remains the case where there is only finitely many primes $p_1,\dots, p_k$ such that, 
for $i\in\{1,\dots,k\}$, $f_{p_i}(1)=0$. Since $f_{p_i}$ is not the zero mapping, there is a greatest positive 
integer $r_i$ such that $f_{p_i}(r_i)=0$. Let $p$ be a prime and $n$, $r$ be positive integers such that 
$p^r$ divides $n$. Since $f(p^r)$ is the remainder of the euclidean division of $f(n)$ by $p^r$, if $f(p^r)\neq 0$, 
then $f(n)\neq 0$. Set $n_0=p_1^{r_1}\cdots p_k^{r_k}$. By construction, we have 
$f(n_0)=0$, hence $\left(1,\frac{a}{n_0}\right)\in l(G)$. 
If $\frac{m}{n}$ is a positive rational number such that $\frac{ma}{n}<\frac{a}{n_0}$, then 
$n>n_0$. Let $n>n_0$. If the primes which divide $n$ belong 
to $\{p_1,\dots, p_k\}$, then there is $i\in\{1,\dots,k\}$ such $p_i^{r_i+1}$ divides $n$. Therefore, 
$f(n)\neq 0$. If some $p$ not in $\{p_1,\dots, p_k\}$ divides $n$, then $f(n)\neq 0$, since $f(p)\neq 0$. 
So $(1,\frac{ma}{n})\notin l(G)$. Consequently, $\left(1,\frac{a}{n_0}\right)$ is the 
smallest element of $l(G)$. By taking $\frac{m}{n_0}$ instead of $a$, in the same way as above we get a 
new function $f$ such that, for every prime $p$, $f(p)\neq 0$. \\ 
\indent (2) By Lemma \ref{lmav424}, for every prime $p$ we have $H_p=\{0\}$. 
Recall that $f_{G,p}(n)$ is the element of $\{1,\dots, p^n-1\}$ such that the element $h$ 
such that $p^nh=1_{G/H_p}$ satisfies $\displaystyle{U(h)={\rm exp}\left(im\frac{2f_{G,p}(n)\pi}{p^n}\right)}$. 
From the definition of $f$, we have $f_{G,p}(n)=f(p^n)$. \\
\indent (3) We consider a family of functions $f_p: \; \NN\backslash \{0\}\rightarrow \{0,\dots,p-1\}$ 
with $f_p(1)\neq 0$ ($p$ prime). By Lemma \ref{lm239} we can assume that the function $f$ is generated by the 
family of $f_p$'s. Hence for every $n\in \NN\backslash\{0\}\; f(p^n)=f_{G,p}(n)$. 
Consequently, the family of $f_{G,p}$ is constructed as in Remark \ref{rk423}. 
This implies that there are $2^{\aleph_0}$ distinct families of functions $f_{G,p}$. \\
\indent Let $(G,R)$, $(G',R')$ be discrete cyclically ordered groups isomorphic to $\QQ$ such that for 
some prime $p$ we have $f_{G,p}\neq f_{G',p}$. 
Then, by Theorem \ref{thm221}, $(G,R)$ and $(G',R')$ are not elementarily equivalent. 
Hence there are $2^{\aleph_0}$ pairwise non elementarily equivalent discrete cyclic orders on $\QQ$. 
\end{proof}
\begin{Fact}\label{fctdd}
Let $\Gamma$ be a linearly ordered divisible abelian group and $(G,R)$ be a discrete cyclically ordered group 
which is isomorphic to $\QQ$. Let $(G',R')$ be the 
lexicographic product $G\overrightarrow{\times} \Gamma$. Then $(G',R')$ is not discrete, for every 
prime $p$ we have $H_p'=\{0_G\}\times \Gamma$, $G'/H_p'\simeq G\simeq G/H_p$  is discrete 
and $f_{G',p}=f_{G,p}$. Hence, if some family 
of functions is the family of $f_{G,p}$'s for some discrete divisible abelian cyclically ordered group, then it is 
also the family of $f_{G,p}$'s for some dense one. 
\end{Fact}
\noindent 
{\it Proof of Theorem \ref{uncount}}. (1) By \cite[Theorem 5.10]{Zh}, $\Gamma$ is 
cyclically orderable if, and only if $T(\Gamma)$ embeds 
in $T(\UU)$ and $\Gamma/T(\Gamma)$ is orderable. 
By \cite[Corollary 5 on p.\ 36]{F}, any abelian torsion-free group admits a linear order. 
By Remark \ref{r}, $T(\Gamma)$ is divisible. \\ 
\indent (2) Assume that $\Gamma$ is divisible and not trivial. 
By \cite[Remark 5.4]{GiLL}, every c-archimedean cyclically ordered group 
c-embeds in a unique way in $\UU$. Hence all cyclic orders on $\Gamma$ are c-isomorphic. 
Without loss of generality, we can assume that $\Gamma$ is a subgroup of $T(\UU)$. Then, there is a prime $p$ 
such that $e^{2i\pi/p}$ belongs to $\Gamma$. We prove by induction that it contains $e^{2i\pi/p^n}$ 
for every positive integer $n$. Assume that $e^{2i\pi/p^n}\in \Gamma$. 
Since $\Gamma$ is divisible, it contains a $p$-root of $e^{2i\pi/p^n}$, say $e^{2i\pi/p^{n+1}}e^{2ik\pi/p}$, 
with $k\in \{0,1,\dots, p-1\}$. Assume that $k\neq 0$. Then $e^{2i\pi/p^{n+1}}e^{2ik\pi/p}(e^{2i\pi/p})^{p-k}=
e^{2i\pi/p^{n+1}}$ belongs to $\Gamma$. 
We can construct by induction infinitely many group 
isomorphisms between $T(\UU)$ and $\Gamma$. For every prime $p$ and positive integer $n$, 
we send $e^{2i\pi/p^n}$ to any of the 
$p-1$ primitive $p$-th root of the image of $e^{2i\pi/p^{n-1}}$ in $\Gamma$. 
So, this gives that 
$2^{\aleph_0}$ isomorphisms between $T(\UU)$ and $\Gamma$, each one gives rise to 
a cyclic order on $\Gamma$.  \\
\indent (4) Assume that  $\Gamma$ is torsion-free. 
Hence it is a $\QQ$-vector space, and there are two divisible abelian subgroups $\Gamma'$ and 
$\Gamma''$ such that $\Gamma=\Gamma'\oplus\Gamma''$ and $\Gamma''$ is isomorphic to $\QQ$. 
Recall that any abelian torsion-free group admits a linear order. Hence we can 
assume that $\Gamma'$ is linearly ordered. Now, let $(G,R)$ be a cyclically ordered group which is isomorphic to 
$\QQ$, then $\Gamma$ is isomorphic to the cyclically ordered group $G'=G\overrightarrow{\times} \Gamma'$. 
In the same way as in Fact \ref{fctdd}, if $(G,R)$ is discretely cyclically ordered, 
then for every prime $p$ we have $f_{G',p}=f_{G,p}$. Since there are $2^{\aleph_0}$ pairwise 
different families of $f_{G,p}$'s, 
there are $2^{\aleph_0}$ pairwise non elementarily equivalent cyclic orders on $\Gamma$. \\ 
\indent Note that $H_p'=H_p\overrightarrow{\times}\Gamma'$, hence by Corollary \ref{prop424} we can have  
$2^{\aleph_0}$ non isomorphic families of $\mathcal{CD}(G')$. Hence we can have $2^{\aleph_0}$ non pairwise 
elementarily equivalent nonlinear cyclic orders such that all the $G'/G_p'$ are dense. \\ 
\indent Above construction is not the only possible. For example, we can also send $1$ to $e^{i\theta}$, where 
$\theta\notin 2\pi\ZZ$. For $\theta$ and $\theta'$ which are not congruent modulo $2\pi\ZZ$, 
we get non-c-isomorphic cyclic orders. So there uncountably many non-isomorphic such cyclic orders. 
If $\Gamma$ is isomorphic to 
a countable product $\prod_{n\in\NN}\Gamma_n$ of groups isomorphic to $\QQ$, we can take embeddings 
of the $\Gamma_n$'s in $\UU\overrightarrow{\times} \QQ$ such that their images have trivial intersections. Then 
$\Gamma$ is isomorphic to the subgroup of $\UU\overrightarrow{\times} \QQ$ generated by these images.  \\
\indent (3) Since $\Gamma/T(\Gamma)$ is abelian torsion-free and divisible, it is a $\QQ$-vector space. 
Furthermore, its dimension is positive, since $\Gamma$ contains non-torsion elements. 
We let $\mathcal{B}$ be a basis of $\Gamma/T(\Gamma)$ and $\mathcal{B}'$ be a subset of 
$\Gamma$ such that the restriction of the canonical epimorphism $\Gamma\twoheadrightarrow \Gamma/T(\Gamma)$ 
induces a one-to-one mapping between $\mathcal{B}'$ and $\mathcal{B}$. Then, the subspace $\Gamma'$ generated by 
$\mathcal{B}'$ is a divisible abelian subgroup of $\Gamma$ which is isomorphic to $\Gamma/T(\Gamma)$. 
One can check that $\Gamma=\Gamma'\oplus T(\Gamma)$. 
We let $\Gamma'=\Gamma_1'\oplus \Gamma_2'$, where ${\rm dim}(\Gamma_2')=1$ (hence $\Gamma_2'\simeq \QQ$). 
We embed $\Gamma_2'$ in $\UU$. Since $\Gamma_2'\cap T(\Gamma)=\{0_{\Gamma}\}$, this 
gives rise to a cyclic order on $\Gamma_2'\oplus T(\Gamma)$. So the lexicographic product 
$(\Gamma_2'\oplus T(\Gamma))\overrightarrow{\times}\Gamma_1'$ is a cyclically group which is isomorphic to $\Gamma$. 
Now, we saw after Proposition \ref{lem423} that we can embed $\Gamma_2'$ in uncountably many in ways $\UU$ 
so to get non-isomorphic 
cyclic orders. Hence this gives rise to $2^{\aleph_0}$ non-isomorphic cyclic orders on $\Gamma$. 
Now, if $T(\Gamma)\simeq T(\UU)$, then $\Gamma$ is c-divisible.  
Hence, by Theorem \ref{thm215}, all these cyclic orders are elementarily equivalent. 
\hfill $\qed$ 
\subsection{The families $\mathcal{CD}(G)$.}
Note that $\mathcal{CD}(G)$ induces an equivalence relation on the set of all primes, where $p$ and 
$q$ are e\-qui\-va\-lent if, and only if, $H_p=H_q$. Hence it induces a partition of this set. This partition 
is linearly ordered in the following way. We set 
$\alpha<\beta $ if, and only if, for $p\in\alpha$ and $q\in\beta$ we have $H_p\subsetneq H_q$. \\ 
\indent If $H_p=\{0_G\}$ for every prime $p$, then the partition is trivial; this holds if $(G,R)$ is discrete.  
\begin{Proposition} 
\begin{enumerate} 
\item Every linearly ordered partition of the set of all prime numbers is 
induced by some $\mathcal{CD}(G)$, where $(G,R)$ is a dense divisible abelian nonlinear cyclically ordered group. \\
If this chain has a smallest element $\alpha_1$, then we can assume that, for $q\in \alpha_1$, 
$H_q=\{0_G\}$ or not. If this partition has a greatest element $\alpha_0$, then we can assume that for 
$p\in\alpha_0$ we have $H_p=l(G)$, so $G/H_p$ is dense. We can also assume that $G$ is c-$p$-divisible for 
every $p\in \alpha_0$. 
\item For every linearly ordered partition $A$ of the set of all prime numbers which has a greatest element $\alpha_0$ 
and for every family of functions $f_p$ from $\NN\backslash\{ 0\}$ to $\{ 0, \dots,p-1\}$ ($p\in \alpha_0$) 
with $f_p(1)\neq 0$, 
there is a divisible abelian nonlinear cyclically ordered group $(G,R)$ such that $A$ is the partition induced 
by the $H_p$'s, for every $p\in \alpha_0$ $G/H_p$ is discrete and for every $p\in \alpha_0$, $r\in \NN\backslash \{0\}$ we have 
$f_{G,p}(r)=f_p(1)+pf_p(2)+\cdots+p^{r-1}f_p(r)$. 
\end{enumerate}
\end{Proposition}
\begin{proof} In Proposition \ref{lm241} we constructed examples where the partition contains 
only one class, $H_p=\{0_G\}$ for every prime $p$, and $(G,R)$ either dense or discrete. In the case 
where it is discrete, we also proved that the family of functions $f_{G,p}$ can be any family 
satisfying the required conditions. Corollary \ref{prop424}, gives an example where the partition contains 
only one class, $H_p=l(G)$ for every prime $p$. 
In the following we assume that the partition contains at least two classes. \\[2mm]
\indent (1) 
Let $(A,\leq)$ be an ordered partition of the set of all primes, which contains at least two classes. \\[2mm]
\indent (a) Construction of the unwound $(\Gamma,z)$. 
For $\alpha\in A$ we 
denote by $\QQ_{\alpha}$ the subgroup of $\QQ$ generated by $\displaystyle{\left\{\frac{1}{p^n}\;:\; 
p\notin \alpha,\; n\in \NN\right\}}$. The group $\QQ_{\alpha}$ is $p$ divisible if, and only if, 
$p\notin \alpha$, and if $p\in \alpha$, then $1$ is not $p$-divisible. For every prime $p$ we let 
$\alpha(p)$ denote the unique $\alpha \in A$ such that $p\in \alpha$. \\
\indent We denote by $\overleftarrow{\prod}_{\alpha\in A}\QQ_{\alpha}$ the additive group 
$\prod_{\alpha\in A}\QQ_{\alpha}$ together with the inverse lexicographic order.  
The nontrivial convex subgroups of 
$\overleftarrow{\prod}_{\alpha\in A}\QQ_{\alpha}$ have the form 
$$\overleftarrow{\prod}_{\alpha\in A_1}\QQ_{\alpha}\overleftarrow{\times}
\overleftarrow{\prod}_{\alpha\in A_2}\{0\},$$ 
where $A_1<A_2$ is an ordered partition of $A$. 
Now, such a group is $p$ divisible if, and only if, $\alpha(p)\in A_2$. So the family of 
maximal $p$-divisible convex subgroups of $\overleftarrow{\prod}_{\alpha\in A}\QQ_{\alpha}$ 
induces the ordered partition $(A,\leq)$. \\
\indent Let 
$z=(z_{\alpha})_{\alpha\in A}$ be the element of $\overleftarrow{\prod}_{\alpha\in A}\QQ_{\alpha}$ 
such that for every $\alpha$ we have $z_{\alpha}=1$. 
Let $\alpha_0\in A$. The element $x=(x_{\alpha})_{\alpha\in A}$ such that $x_{\alpha_0}=1$ and, 
for $\alpha\neq \alpha_0$, $x_{\alpha}=0$ is denoted by $1_{\alpha_0}$.\\
\indent The {\it support} of an element $x=(x_{\alpha})_{\alpha\in A}$ is the set of $\alpha\in A$ such that 
$x_{\alpha}\neq 0$ (the support of $0$ is $\emptyset$). \\
\indent For $\alpha\in A$, we let $\Delta_{\alpha}$ be the subroup of elements $x$ of 
$\overleftarrow{\prod}_{\alpha\in A}\QQ_{\alpha}$ such that for every $\alpha'\neq \alpha$ we have 
$x_{\alpha'}=0$. 
Then $\sum_{\alpha\in A}\Delta_{\alpha}$ is the 
subgroup of all elements of $\prod_{\alpha\in A}\QQ_{\alpha}$ with finite support. We denote by 
$\Gamma_0$ the subgroup generated by $z$ and the 
$\displaystyle{\frac{1}{p^n}\left(z-1_{\alpha(p)}\right)}$'s, 
where $p$ is prime, $n\in \NN\backslash \{0\}$. We set $\displaystyle{\Gamma=\Gamma_0+\sum_{\alpha\in A}\Delta_{\alpha} }$. 
Note that $z$ is cofinal in $\Gamma$. We show that it is not divisible by any prime. \\
\indent The group $\sum_{\alpha\in A}\Delta_{\alpha}$ is generated by the $\frac{1}{m}1_{\alpha}$'s, where, 
for $p\in \alpha$, the integers  $p$ and $m$ are coprime. So, if $x=\frac{1}{m}1_{\alpha(p)}$, then the denominator of 
$x_{\alpha(p)}$ is not divisible by $p$. If $x=\frac{1}{p^n}(z-1_{\alpha(p)})$, then $x_{\alpha(p)}=0$, and $x_{\alpha}=
\frac{1}{p^n}$ for $\alpha\neq \alpha(p)$. Now, the group $\Gamma$ is generated by $z$, the 
$\frac{1}{p^n}(z-1_{\alpha(p)})$'s and the $\frac{1}{m}1_{\alpha}$'s, where $m$ and $p$ are coprime. 
Hence, for any $x\in\Gamma$ and $p$ prime, if $x_{\alpha(p)}\neq 0$, then its 
denominator is not divisible by $p$. In particular, $z$ is not divisible by any prime. \\[2mm]
\indent (b) We let $G=\Gamma/\langle z\rangle$. This group is torsion-free, since $z$ is 
not divisible by any prime. We show that the cyclically ordered group 
$G$ is nonlinear. We saw after the definition of the wound-round 
(Subsection \ref{subsect21}) that if $G$ is linearly cyclically ordered, then it is isomorphic to 
$(\ZZ\overrightarrow{\times} G)/\langle (1,0_G)\rangle$, where $\overrightarrow{\times}$ denotes the 
lexicographic product of linearly ordered groups. By uniqueness of the unwound, it follows that $G$ is 
linearly cyclically ordered if, and only if, $\Gamma$ is order isomorphic to $\ZZ\overrightarrow{\times} 
l(G)_{uw}$, where $l(G)_{uw}$ is the greatest proper convex subgroup of $\Gamma$. Therefore, $G$ is linearly ordered 
if, and only if, for every cofinal $x\in \Gamma$ there exists a positive integer $n$ such that $x-nz$ belongs 
to the greatest proper convex subgroup of $\Gamma$. If $A$ is infinite, then 
the greatest 
convex subgroup of $\Gamma$ is $\sum_{\alpha\in A}\Delta_{\alpha}$. For every prime $p$, $\frac{1}{p}(z-1_{\alpha(p)})$ 
is cofinal in $\Gamma$. Now, for every positive integer $n$, 
$\frac{1}{p}(z-1_{\alpha(p)})-nz$ does not belong to $\sum_{\alpha\in A}\Delta_{\alpha}$, since it has an infinite support. 
Therefore $G$ is nonlinear. 
If $A$ is finite, then it has a greatest element $\alpha_0$. Then $\Gamma=\bigoplus_{\alpha\in A}\Delta_{\alpha}$ 
and its greatest convex subgroup is $\bigoplus_{\alpha\neq \alpha_0}\Delta_{\alpha}$. We let $p$ be a prime not in 
$\alpha_0$. So $\frac{1}{p}(z-1_{\alpha(p)})$ is cofinal in $\Gamma$. Now, for every positive integer $n$, 
$\frac{1}{p}(z-1_{\alpha(p)})-nz$ does not belong to $\bigoplus_{\alpha\neq \alpha_0}\Delta_{\alpha}$. 
Consequently, $G$ is nonlinear. \\[2mm]
\indent (c) We show that $G$ is divisible by proving that for every prime $p$ and 
every $n\in\NN\backslash \{0\}$, each generator of $\Gamma$ is divisible 
by $\frac{1}{p^n}$ modulo $\langle z\rangle$. We start with $\frac{1}{m}1_{\alpha}$ where, for every $p\in\alpha$, 
$m$ and $p$ are coprime. Clearly, 
if $p\notin \alpha$, then $\frac{1}{m}1_{\alpha}$ is divisible by $p^n$. 
Assume that $p\in\alpha$. We have 
$$\displaystyle{1_{\alpha}=p^n\left(1_{\alpha}+\left(p^n-1\right)\frac{1}{p^n}\left(z-1_\alpha\right)\right)
+\left(1-p^n\right)z}.$$ 
Therefore $1_{\alpha}$ is divisible by $p^n$ modulo $\langle z\rangle$. 
Now, since $m$ and $p^n$ are coprime, there exist integers $u$ and $v$ such that 
$up ^n+vm=1$. Then 
$\displaystyle{\frac{1}{m}1_{\alpha}=v1_{\alpha}+p^n\frac{u}{m}1_{\alpha}}$. 
Hence $\frac{1}{m}1_\alpha$ is divisible by $p^n$ modulo $\langle z\rangle$. \\
\indent We turn to $\displaystyle{\frac{1}{q^m}\left( z-1_{\alpha(q)}\right) }$. If $p=q$, then 
$\displaystyle{\frac{1}{q^m}\left( z-1_{\alpha(q)}\right)=p^ n\frac{1}{p^{m+n}}\left( z-1_{\alpha(p)}\right)}$. 
Assume that $p$, $q$ are coprime and let $u$, $v$ be integers such that $up^n+vq^m=1$. Then, 
$$p^n\left(\frac{u}{q^m}\left( z-1_{\alpha(q)}\right)+\frac{v}{p^n}\left( z-1_{\alpha(p)}\right)\right)
+v(1_{\alpha(p)}-1_{\alpha(q)})=$$
$$=\frac{1}{q^m}\left(\left(up^n+vq^m\right)z-up^n1_{\alpha(q)}-vq^m1_{\alpha(p)}\right)
+\frac{1}{q^m}\left(vq^m1_{\alpha(p)}-vq^m1_{\alpha(q)}\right)=$$
$$=\frac{1}{q^m}\left(up^n+vq^m\right)(z-1_{\alpha(q)})=\frac{1}{q^m}(z-1_{\alpha(q)}).$$
\indent Since $1_{\alpha(p)}$ and $1_{\alpha(q)}$ are divisible by $p^n$ modulo $\langle z\rangle$, 
this proves that $\frac{1}{q^m}(z-1_{\alpha(q)})$ is divisible by $p^n$ modulo $\langle z\rangle$. \\
\indent Since the unwound $\Gamma$ of $G$ is dense, $G$ is dense. \\[2mm]
\indent (d) Now, we look at the subgroups $(H_p)_{uw}$ of $\Gamma$. Since $\Gamma$ 
is a subgroup of $\overleftarrow{\prod}_{\alpha\in A}\QQ_{\alpha}$, $(H_p)_{uw}$ is the set of $x\in \Gamma$ such that 
$\alpha\geq \alpha(p)\Rightarrow x_{\alpha}=0$. \\
\indent Since $\mathcal{CD}(G)$ is isomorphic to the preordered family 
$\{\{(0,0_G)\},\; {\rm uw}(G),\; (H_p)_{uw}\; : \; p \mbox{ prime }\}$, it defines the ordered partition $A$. \\[2mm]
\indent (e) If the partition has a smallest element $\alpha_1$, then for $p\in \alpha_1$ we have $H_p=\{0_G\}$. Then we can take the lexicographical product 
$G'=G\overrightarrow{\times}\QQ$. The partition is the same, but it contains a divisible c-convex proper subgroup. Hence 
for every prime $p$, $H_p'\neq\{ 0\}$. \\
\indent If the partition has no greatest element, then every $G/H_p$ is dense, since there is 
no maximal $H_p$.  Assume that the partition has a greatest element $\alpha_0$. 
Then, for $p\in \alpha_0$ the greatest $p$-divisible convex subgroup of $\Gamma$ 
is equal to the subset of elements $x$ such that $x_{\alpha_0}=0$. 
Now, this is also the greatest convex subgroup of $\Gamma$, and we know that it is isomorphic to $l(G)$. Therefore, 
$H_p=l(G)$, hence $G/H_p\simeq U(G)$ is divisible, so it is dense. In order to get a group $G'$ with torsion elements, 
in above construction  take $\overleftarrow{\prod}_{\alpha\in A\backslash\{\alpha_0\}}\QQ_{\alpha}$ instead of 
$\overleftarrow{\prod}_{\alpha\in A}\QQ_{\alpha}$, and let $\Gamma'$ be the group obtained in the same way as 
$\Gamma$. Then $\Gamma'$ is $p$-divisible if, and only if, $p\in \alpha_0$. So $G'=\Gamma'/\langle z\rangle$ 
is c-$p$-divisible if, and only if, $p\in \alpha_0$. \\[2mm]
\indent (2) Let $A$ be a linearly ordered partition of the set of all prime numbers which has a greatest element $\alpha_0$ 
and a family of functions $f_p$ from $\NN\backslash\{ 0\}$ to $\{ 0, \dots,p-1\}$ ($p\in \alpha_0$) with $f_p(1)\neq 0$.  
For every prime $p\notin \alpha_0$, we set $f_p(1)=1$, and for every $r>1$ we set $f_p(r)=0$. 
We let $f$ be the function defined in the same way as in the proof of Lemma \ref{lm239}. 
Denote by $G_0$ the cyclically ordered group 
$$\displaystyle{
\left\{\left({\rm exp}\left(im\frac{2f(n)\pi}{n}\right),\frac{m}{n}\right)\; :\; m\in \ZZ,\; n\in 
\NN\backslash \{0\}\right\}}$$
defined in Lemma \ref{lm238}. By Proposition \ref{lm241}, $G_0$ is discrete. By Proposition \ref{lem423}, it is not c-arcimedean. 
Since it embeds in $\UU\overrightarrow{\times}\QQ$, its linear part is isomorphic to a discrete subgroup of 
$\QQ$. Hence $l(G_0)\simeq \ZZ$. \\ 
\indent  (a) Construction of the unwound. 
We embed the group ${\rm uw}(G_0)$ in its 
divisible closure, and we let $\Delta_{\alpha_0}'$ be the subgroup generated by ${\rm uw}(G_0)$ and 
the $\frac{1}{p^n}z_{G_0}$'s, where $p\notin \alpha_0$ is a prime and $n\in \NN\backslash \{0\}$. 
The greatest convex subgroup $l(G_0)_{uw}$ of ${\rm uw}(G_0)$ is isomorphic to $l(G_0)$, so to $\ZZ$. Now, $z_{G_0}$ 
does not belong to this subgroup. Hence the greatest convex subgroup of $\Delta_{\alpha_0}'$ is isomorphic to $\ZZ$. 
It follows that $\Delta_{\alpha_0}'$ is discrete. 
We take $\overleftarrow{\prod}_{\alpha \in A\backslash\{\alpha_0\}}\QQ_{\alpha}\overleftarrow{\times}\Delta_{\alpha_0}'$, 
instead of $\overleftarrow{\prod}_{\alpha\in A}\QQ_{\alpha}$. \\
\indent For every $\alpha \in A$, we let $\Delta_{\alpha}$ the subgroup of elements $x$ such that 
$x_{\alpha'}=0$ for every $\alpha'\neq \alpha$. If $\alpha\neq \alpha_0$, then $1_{\alpha}$ is defined as in the 
first part of this proof. We let $1_{\alpha_0}$ be 
the element $x$ such that $\alpha\neq\alpha_0$ implies $x_{\alpha}=0$, and $x_{\alpha_0}=z_{G_0}$. 
We have $\Delta_{\alpha_0}\simeq\Delta_{\alpha_0}'$ and, for $\alpha\neq \alpha_0$, $\Delta_{\alpha}\simeq \QQ_{\alpha}$. \\
\indent Let 
$z=(z_{\alpha})_{\alpha\in A}$ the element such that for every $\alpha\neq\alpha_0$ we have $z_{\alpha}=1$, 
and $z_{\alpha_0}=z_{G_0}$. Note that in $\Delta_{\alpha_0}'$ the element $z_{G_0}$ is not divisible by any 
$p\in \alpha_0$, and it is divisible by every $p^n$ such that $p\notin \alpha$, and $n\in \NN\backslash \{0\}$. 
We denote by 
$\Gamma_0$ the subgroup generated by $z$ and the 
$\displaystyle{\frac{1}{p^n}\left(z-1_{\alpha(p)}\right)}$'s, 
where $p$ is prime, $n\in \NN\backslash \{0\}$, and $\alpha(p)$ denotes the unique $\alpha \in A$ such that 
$p\in \alpha(p)$. We set $\displaystyle{\Gamma=\Gamma_0+\sum_{\alpha\in A}\Delta_{\alpha} }$. \\
\indent (b) Since $z$ is cofinal in $\Gamma$, we can consider the cyclically ordered group $G=\Gamma/\langle z\rangle$. 
In the same way as in (1) (b), $G$ is nonlinear.  
We show that $G$ is torsion-free and divisible. \\
\indent In the same way as in (1) (a), $z$ is not divisible by any integer in $\Gamma$, so $G$ is torsion-free. 
Furthermore,  for 
$\alpha\in A$, the elements $\frac{1}{m}1_{\alpha}$ and $\frac{1}{q^m}(z-1_{\alpha(q)})$ are 
divisible by every $\frac{1}{p^n}$ modulo $\langle z\rangle$. In order to prove that $G$ is divisible, 
it remains to look at the  elements of $\Delta_{\alpha_0}$. Note that $\Delta_{\alpha_0}$ is generated 
by the $\frac{1}{m}1_{\alpha_0}$ (where the gcd of $m$ and every $p\in\alpha_0$ is $1$) and the elements 
$x$ such that $x_{\alpha_0}\in {\rm uw}(G_0)$ and for $\alpha\neq \alpha_0$ we have 
$x_{\alpha}=0$. So it is sufficient to focus on those $x$. Since ${\rm uw}(G_0)$ is divisible modulo 
$\langle z_{G_0}\rangle$, there is $y\in \Delta_{\alpha_0}$ and $k\in \ZZ$ such that 
$p^ny_{\alpha_0}=x_{\alpha_0}+kz_{G_0}$. Then $p^ny=x+k\! \cdot\! 1_{\alpha_0}$. 
Since $1_{\alpha_0}$ is divisible by $p^n$ modulo $\langle z\rangle$, so is $x$. \\
\indent (c) We look at the maximal $p$-divisible convex subgroups $\Gamma_p$ of $\Gamma$. 
If $p\notin \alpha_0$, then they are constructed in the same way as in (1). 
If $p\in \alpha_0$, then in the same way as in (1) we see that $\Gamma_p\supseteq \Gamma\cap 
 \overleftarrow{\prod}_{\alpha \in\backslash\{\alpha_0\}}\QQ_{\alpha}\overleftarrow{\times}\{0_G\}$. We prove that 
this inclusion is an equality. 
Since $G_0$ is not c-archimedean, ${\rm uw}(G_0)$ is not archimedean. Since $z_{G_0}$ 
is cofinal in  ${\rm uw}(G_0)$, so are the $\frac{1}{p^n}z_{G_0}$. Therefore, the greatest proper convex subgroup of 
$\Delta_{\alpha_0}'$ is $l(G_0)_{uw}$, which is isomorphic to $\ZZ$. 
It follows that the greatest proper convex subgroup of $\Gamma$ is $\Gamma\cap 
 \overleftarrow{\prod}_{\alpha \in\backslash\{\alpha_0\}}\QQ_{\alpha}\overleftarrow{\times}l(G_0)_{uw}$, 
and there is no convex subgroup between $\Gamma\cap 
 \overleftarrow{\prod}_{\alpha \in\backslash\{\alpha_0\}}\QQ_{\alpha}\overleftarrow{\times}\{0_G\}$ and $\Gamma\cap 
 \overleftarrow{\prod}_{\alpha \in\backslash\{\alpha_0\}}\QQ_{\alpha}\overleftarrow{\times}l(G_0)_{uw}$. 
Since $l(G_0)$ is discrete, 
the greatest proper convex subgroup is not divisible by any prime. So above inequality is an equality. \\
\indent (d) Now, $\Gamma={\rm uw}(G)$, hence the greatest proper convex subgroup $l(G)_{uw}$ of $\Gamma$ 
is isomorphic to $l(G)$. Consequently, 
the partition induced by $G$ on the set of prime numbers is $A$. Let $p\in \alpha_0$, $q\notin \alpha_0$ 
and $n\in \NN\backslash\{0\}$. Then $\frac{1}{p^n}(z-1_{\alpha_0})\in \Gamma_p$ and 
$\frac{1}{q^n}(z-1_{\alpha_(q)})-\frac{1}{q^n}\!\cdot\!1_{\alpha_0}\in \Gamma_p$. 
Clearly, $z$ and the elements of $\displaystyle{\sum_{\alpha\in A}\Delta_{\alpha}}$ are congruent modulo 
$\Gamma_p$ to an element of $\Delta_{\alpha_0}$. 
It follows that every element of $\Gamma$ is congruent modulo 
$\Gamma_p$ to an element of $\Delta_{\alpha_0}$. Therefore, $\Gamma/\Gamma_p\simeq \Delta_{\alpha_0}'$ 
is discrete. \\
\indent (e) Finally, we turn to the functions $f_{G,p}$, where $p\in \alpha_0$. By the definition 
of the group $G_0$, we have $f_{G_0,p}=f_p$. 
Let $x$ be the element of  $\Delta_{\alpha_0}$ such that $x_{\alpha_0}$ is the smallest positive element 
of $\Delta_{\alpha_0}'$. Then the class $x+\Gamma_p$ is the 
smallest positive element of $\Gamma/\Gamma_p$. Let $r\in \NN\backslash\{ 0\}$. 
By Lemma \ref{lm235} there is $y'\in {\rm uw}(G_0)$ such that 
$x_{\alpha_0}=p^ry'+f_p(r)z_{G_0}$. We denote by $y$ the element of $\Delta_{\alpha_0}$ such that $y_{\alpha_0}=y'$. 
Then $x=p^ry+f_p(r)1_{\alpha_0}$. So $x=p^ry-f_p(r)(z-1_{\alpha_0})+f_p(r)z$. Since $p\in \alpha_0$,  
$t=\frac{1}{p^r}(z-1_{\alpha_0})$ belongs to $\Gamma$. Therefore $x=p^r(y-f_p(r)t)+f_p(r)z$. By Lemma \ref{lm235} 
again, this shows that $f_p(r)$ is congruent to $f_{G,p}(r)$ modulo $p^r$. Now, since both of $f_p(r)$ and 
$f_{G,p}(r)$ belong to $\{0,\dots, p^r-1\}$, we have $f_p(r)=f_{G,p}(r)$. 
This proves that the function $f_{G,p}$ is equal to $f_p$. 
\end{proof}
\section{cyclically minimal cyclically ordered groups.}\label{sec5}
\indent 
In the case of real numbers, each definable subset is a finite union of intervals, it is definable by a 
quantifier-free formula in the language of order. The general study of linearly ordered algebraic 
structures having such a property has been done by A.\ Pillay and C.\ Steinhorn as o-minimal structures 
(\cite{PiSt}). Afterwards M.\ Dickmann introduced the notion of weakly o-minimal structures (\cite{Di}). 
In a very large context D.\ Macpherson and S.\ Steinhorn looked at analogues (\cite{MSt}). In the case of 
cyclically ordered structures the analogue is the notion of cyclically minimal structures. They proved that a 
cyclically ordered group $(G,R)$ is cyclically minimal if, and only if, it is abelian and  its unwound is 
divisible. Lucas proved independently this theorem. He deduced from Section \ref{sec3} that this condition is sufficient. 
He proved the converse in several lemmas, which also characterized the weakly cyclically minimal structures. 
Some of these lemmas contained errors, so we had to make changes in the proof of Lucas. 
These results of Lucas remained unpublished. 
Later, weakly cyclically minimal cyclically ordered groups were also studied in \cite{KV15}, where it was proved that every 
weakly minimal cyclically ordered group is abelian. \\
\indent  Since the groups are not necessarily abelian, 
we take here the multiplicative notation for the group law, however we speak of divisible group and of elements 
divisible by n (with this convention $x$ is divisible by $n$ if $\exists y\mbox{ } y^n=x$).
\begin{Proposition}\label{n51}
Each c-divisible abelian cyclically ordered group is cyclically minimal. 
\end{Proposition}
\begin{proof} First, note that a c-divisible abelian cyclically ordered group is infinite since its 
torsion subgroup is isomorphic to $T(\UU)$. 
In \cite{MSt}, D.\ Macpherson and C.\ Steinhorn asserted that this follows from the minimality of 
the divisible abelian linearly ordered groups and the interpretability of a cyclically ordered group 
in its unwound. This interpretability has been studied more in detail in 
\cite[Lemma 4.3]{GiLL}. Lucas asserted that this  result can be obtained using the elimination 
of quantifiers in c-divisible 
abelian cyclically ordered groups (Theorem \ref{thm31}). We prove this assertion, that is 
we show that any quantifier-free 
formula defines a finite union of intervals and singletons. 
First, note that a c-divisible abelian cyclically ordered group is infinite, since its torsion subgroup is.  
A quantifier-free formula is a boolean 
combination of formulas such that $ax^n=b$ and $R(ax^m,bx^n,cx^p)$. Since, $ax^n=b\Leftrightarrow 
a=bx^{-n}$, we can assume that $n\geq 1$. The formula $ax^n=b$ is 
equivalent to $x^n=ba^{-1}$, and it defines the set $n$-th roots of $ba^{-1}$, which contain $n$ 
elements, by properties of the subgroup of torsion elements of a c-divisible abelian cyclically 
ordered group. If $m=n=p$, then $R(ax^m,bx^n,cx^p)$ is equivalent to $R(a,b,c)$. If $m=n\neq p$, 
then $R(ax^m,bx^n,cx^p)$ is equivalent to $R(bc^{-1},x^{p-n},ac^{-1})$. The cases $m\neq n=p$ and 
$n\neq p=m$ are similar. If $m\neq n\neq p\neq m$, and $m=\min(m,n,p)$, then $R(ax^m,bx^n,cx^p)$ is 
equivalent to $R(ab^{-1},x^{n-m},cb^{-1}x^{p-m})$, where $n-m>0$ and $p-m>0$. The cases $n=\min(m,n,p)$ and 
$p=\min(m,n,p)$ are similar. So, it remains to prove that the formulas 
$R(a,x^n,b)$ and $R(a,x^m,bx^n)$ (where in the second case $m$, $n$ are positive integers) define unions of open intervals. \\
\indent We let $y$ be the element of 
the unwound $\Gamma={\rm uw}(G)$ such that $x$ is the image of $y$ in $G=\Gamma/\langle z_G\rangle$ 
and $e_{\Gamma}\leq y<z_G$. 
We also denote by $a'$ (resp.\ $b'$) the element of $[e_{\Gamma},z_G[$ such that its image is $a$ (resp.\ $b$). 
Assuming $n\geq 1$, we let   
$a^{1/n}$ (resp.\ $b^{1/n}$, resp.\ $\zeta_n$) be the image of $(a')^{1/n}$ (resp.\ $(b')^{1/n}$, 
resp.\ $z_G^{1/n}$). 
Assume that $n\geq 1$. By properties of the unwounds $R(a,x^n,b)$ holds if, and only if, there is $k\in \ZZ$ such that 
$e_{\Gamma}\leq a'<y^nz_G^{-k}<b'<z_G$, or $e_{\Gamma}\leq y^nz_G^{-k}<b'<a'<z_G$ or $e_{\Gamma}\leq b'<a'<y^nz_G^{-k}<z_G$. 
This in turn is 
equivalent to 
$(a')^{1/n}z_G^{k/n}<y<(b')^{1/n}z_G^{k/n}$, or 
$z_G^{k/n}\leq y<(b')^{1/n}z_G^{k/n}<(a')^{1/n}z_G^{k/n}$ or 
$(b')^{1/n}z_G^{k/n}<(a')^{1/n}z_G^{k/n}<y<z_G^{(1+k)/n}$. 
Note that $e_{\Gamma}\leq y<z_G$ and $e_{\Gamma}< y^nz_G^{-k}<z_G$ imply $0\leq k \leq n-1$. 
It follows that 
$R(a,x^n,b)$ is equivalent to $\bigvee_{0\leq k\leq n-1}R(a^{1/n}\zeta_n^k,x,b^{1/n}\zeta_n^k)$, which 
defines a finite union of intervals. \\
\indent Assume that $n<0$. Then $R(a,x^n,b)\Leftrightarrow R\left(b^{-1},(x^{-1})^{-n},a^{-1}\right)$. 
This is equivalent to saying that $x^{-1}$ belongs to a finite union of open intervals. 
Now, $x^{-1}\in I(c,d) \Leftrightarrow x\in I(d^{-1},c^{-1})$. Hence $R(a,x^n,b)$ defines a finite  
union of open intervals. \\
\indent We assume that $m>n$. The formula $R(a,x^m,bx^n)$ holds if, and only if, there are integers $k$, $l$ such 
that $e_{\Gamma}\leq a'<y^mz_G^{-k}<b'y^nz_G^{-l}<z_G$ or $e_{\Gamma}\leq y^mz_G^{-k}<b'y^nz_G^{-l}<a'<z_G$ or 
$e_{\Gamma}\leq b'y^nz_G^{-l}<a'<y^mz_G^{-k}<z_G$. \\
\indent In the same way as above, this implies $0\leq k\leq m-1$, and 
$e_{\Gamma}\leq y<z_G$ and $e_{\Gamma}< b'y^nz_G^{-l}<z_G$ imply $0\leq l \leq n$. \\
\indent $\bullet$ $e_{\Gamma}\leq a'<y^mz_G^{-k}<b'y^nz_G^{-l}<z_G$ is equivalent to 
$$(a')^{1/m}z_G^{k/m}<y,\;y<(b')^{1/(m-n)}z_G^{(k-l)(m-n)}\mbox{ and }y<(b')^{-1/n}z_G^{(l+1)/n}.$$  
$$\mbox{We let }c_{k,l}'=\min \left((b')^{1/(m-n)}z_G^{(k-l)(m-n)},(b')^{-1/n}z_G^{(l+1)/n}\right)\mbox{ and }
c_{k,l}=c_{k,l}'\!\cdot\!\langle z_G\rangle.$$ 
If $l<n$, then $(b')^{-1/n}z_G^{(l+1)/n}<z_G$, and if 
$l=n$, then $(b')^{1/(m-n)}z_G^{(k-n)(m-n)}<z_G$. It follows that $c_{k,l}'<z_G$. Hence 
$e_{\Gamma}\leq a'<y^mz_G^{-k}<b'y^nz_G^{-l}<z_G$ is equivalent to $R(a^{1/m}\zeta_m^k,x,c_{k,l})$. \\
\indent $\bullet$ $e_{\Gamma}\leq y^mz_G^{-k}<b'y^nz_G^{-l}<a'<z_G$ is equivalent to $y=z_G^{k/m}$ or:  
$$z_G^{k/m}<y, \; y<(b')^{1/(m-n)}z_G^{(k-l)(m-n)}\mbox{ and }y<(a')^{1/n}(b')^{-1/n}z_G^{l/n}.$$ 
$$\mbox{We let }d_{k,l}'=\min \left((b')^{1/(m-n)}z_G^{(k-l)(m-n)},(a')^{1/n}(b')^{-1/n}z_G^{l/n}\right)\mbox{ and }
d_{k,l}=d_{k,l}'\!\cdot\!\langle z_G\rangle.$$ 
If $l<n$, then $(a')^{1/n}(b')^{-1/n}z_G^{l/n}<z_G$, and 
if $l=n$, then $(b')^{1/(m-n)}z_G^{(k-n)(m-n)}<z_G$. Therefore $d_{k,l}'<z_G$. Hence 
$e_{\Gamma}\leq y^mz_G^{-k}<b'y^nz_G^{-l}<a'<z_G$ is equivalent to $x=\zeta_m^k$ or 
$R(\zeta_m^k,x,d_{k,l})$. \\
\indent $\bullet$ $e_{\Gamma}\leq b'y^nz_G^{-l}<a'<y^mz_G^{-k}<z_G$ is equivalent to 
$y=(b')^{-1/n}z_G^{l/n}$ or: 
$$(b')^{-1/n}z_G^{l/n}<y,\; y<(a')^{1/n}(b')^{-1/n}z_G^{l/n},\; 
(a')^{1/m}z_G^{k/m}<y\mbox{ and }y<z_G^{(k+1)/m}.$$ 
$$\mbox{We let }
r_{k,l}'=\max\left((b')^{-1/n}z_G^{l/n},(a')^{1/m}z_G^{k/m}\right)\mbox{ and }
r_{k,l}=r_{k,l}'\!\cdot\!\langle z_G\rangle.$$  
If $k<m-1$, then 
$z_G^{(k+1)/m}<z_G$, and if $l<n$, then $(a')^{1/n}(b')^{-1/n}z_G^{l/n}<z_G$. So, if 
$k<m-1$ or $l<n$, then we set $s_{k,l}'=\min\left((a')^{1/n}(b')^{-1/n}z_G^{l/n},z_G^{(k+1)/m}\right)$, and 
$r_{k,l}=r_{k,l}'\!\cdot\!\langle z_G\rangle$. 
We have $s_{k,l}'<z_G$. Then 
$e_{\Gamma}< b'y^nz_G^{-l}<a'<y^mz_G^{-k}<z_G$ is equivalent to $R(r_{k,l},x,s_{k,l})$. 
If $k=m-1$ and $l=n$, then $z_G^{(k+1)/m}=z_G$.  
$e_{\Gamma}< b'y^nz_G^{-n}<a'<y^mz_G^{1-m}<z_G$ is equivalent to $r_{k,l}'<y<(a')^{1/n}(b')^{-1/n}z_G^{l/n}$, 
but we don't know whether $(a')^{1/n}(b')^{-1/n}z_G^{n/n}<z_G$ or not. This inequality is equivalent to $a'<b'$. 
If $a'<b'$, then we set $s_{m,n}'=(a')^{1/n}(b')^{-1/n}z_G$ and $s_{m,n}=s_{m,n}'\!\cdot\!\langle z_G\rangle$. 
Therefore $e_{\Gamma}< b'y^nz_G^{-n}<a'<y^mz_G^{1-m}<z_G$ is equivalent to $R(r_{m,n},x,s_{m,n})$. 
If $a'\geq b'$,  then $e_{\Gamma}< b'y^nz_G^{-n}<a'<y^mz_G^{1-m}<z_G$ is equivalent to 
$r_{m,n}'<y<z_G$, which in turn is equivalent to $R(r_{m,n},x,e_G)$. We set $s_{m,n}=e_G$. \\ 
\indent Consequently, $R(a,x^m,bx^n)$ is equivalent to the disjunction for $k\in \{0,\dots, m-1\}$ 
and $l\in\{0,\dots,n\}$ of: $R(a^{1/m}\zeta_m^k,x,c_{k,l})$ or $x=\zeta_m^k$ or $R(\zeta_m^k,x,d_{k,l})$ 
or $R(r_{k,l},x,s_{k,l})$. \\
\indent Assume that $m<n$. Then $R(a,x^m,bx^n)\Leftrightarrow R\left(ba^{-1},(x^{-1})^n,b(x^{-1})^m\right)$. 
We proved above that this is equivalent to saying that $x^{-1}$ belongs to a finite union of open intervals and 
singletons. In the same way as in the case of the formula $R(a,x^n,b)$, 
it follows that $R(a,x^m,bx^n)$ defines a finite  
union of open intervals and singletons. 
\end{proof}
\indent Recall that a c-convex subset is a subset $J$ such that either $J$ is a singleton or 
for every $g\neq g'$ in $J$, either $I(g,g')\subseteq J$ or $I(g',g)\subseteq J$. \\
\indent We saw in Remark \ref{rk51} that a c-convex subset is not 
necessarily a singleton or an open interval~: the linear part of a cyclically ordered group 
is a c-convex subset and if it is nontrivial, then it is 
neither open interval nor finite unions of open intervals or singleton. 
There also exist definable c-convex subset which are neither an open intervals nor a finite union of open intervals or 
singletons. For example the subgroups $H_p$, and the subsets defined by the formulas ${\rm argbound}_n$, 
before the proof of Proposition \ref{n128}. Indeed, consider the lexicographic product $\UU\overrightarrow{\times} D$, 
where $D$ is an abelian linearly ordered group 
(see Definition \ref{def124}). In this cyclically ordered group, the formula ${\rm argbound}_n$ defines the open interval 
$I\left(({\rm exp}(\frac{2\pi}{n+1}),0),
({\rm exp}(\frac{2\pi}{n}),0)\right)$. Now, if $G$ is a subgroup of $\UU\overrightarrow{\times} D$ 
such that $U(G)$ does not contain ${\rm exp}\left(\frac{2\pi}{n+1}\right)$ or 
${\rm exp}\left(\frac{2\pi}{n+1}\right)$, and $U(G)$ is infinite, then it is not a finite union 
of open intervals or singletons. 
\begin{Proposition}\label{n52}(Lucas)
For each $n\in \NN\backslash\{0\}$ and 
each nontrivial divisible abelian linearly ordered group $D$, the cyclically ordered group 
$T(\UU)_n\overrightarrow{\times}D$ is weakly 
cyclically minimal.
\end{Proposition}
\begin{proof} We denote by $(G,R)$ the cyclically ordered group $T(\UU)_n\overrightarrow{\times}D$. 
Since $D$ is nontrivial, $G$ is infinite. 
Let $E$ be a subset of $G$. Then $E$ can be written as a disjoint union $E=E_0\cup \left(e^{i\frac{2\pi}{n}},e_D\right)E_1\cup\cdots 
\cup \left(e^{i\frac{2\pi}{n}},e_D\right)^{n-1}E_{n-1}$, where $E_0,E_1,\dots ,E_{n-1}$ 
are subsets of $l(G)=\{1\}\overrightarrow{\times}D$. 
Now, in $G$, the element $\left(e^{i\frac{2\pi}{n}},e_D\right)$ is definable by the formula 
$R(e_G,x,\dots,x^{n-1})$ and $x^n=e_G$. Hence all 
the elements of $T(\UU)_n\overrightarrow{\times}\{e_D\}$ are definable. The subgroup $l(G)$ is definable by the 
formula $x=e_G$ or $R(e_G,x,x^2,\dots,x^n)$ or $R(e_G,x^{-1},x^{-2},\dots,x^{-n})$. Assume that $E$ is definable 
by a formula $\varphi(x)$ in the language $\{\cdot,R,e,^{-1}\}$. Then, for $k\in \{0,1,\dots, n-1\}$ 
the set $E_k$ is definable by the formula 
$x\in l(G)$ and $\varphi\left(\left(e^{i\frac{2\pi}{n}},e_D\right)^kx\right)$. 
Now, if, $x$, $y$, $z$ belong to $l(G)$, then $R(x,y,z)$ is equivalent 
to either $x<y<z$ or $y<z<x$ or $z<x<y$. Hence $E_k$ is definable in $l(G)$ equipped with the language of 
ordered groups. Since $l(G)$ is abelian and 
divisible, it is minimal. Therefore, $E_k$ is a finite union of open intervals. 
We saw in Remark \ref{rk51} that the bounded open intervals of $(l(G),<$) are open intervals of $(G,R)$, hence they are 
c-convex. Now, every open interval of $(l(G),<)$ is an increasing union of bounded open intervals, so it is c-convex. 
\end{proof}
\indent Now, we look at necessary conditions. 
In the remainder of this section $(G,R)$ is a cyclically ordered group which is not necessarily abelian. 
We take Lucas's approach to show that if $(G,R)$ is weakly cyclically minimal, then $l(G)$ is 
divisible and abelian. \\
\begin{defi}
If $H$ is a subset of $(G,R)$ and $h\in H$, then we define the {\it c-convex component} of $H$ which contains $h$ to 
be the greatest c-convex subset of $(G,R)$ containing $h$ and contained in $H$. It is denoted by $C(h,H)$.
\end{defi}
\begin{Fact}\label{fct53} For $g$, $h$ in $H$, we have $g\in C(h,H)\Leftrightarrow h\in C(g,H)$. It follows that 
$C(g,H)\neq C(h,H)\Leftrightarrow C(g,H)\cap C(h,H)=\emptyset$. \end{Fact}
\indent A subset $H$ of $G$ is said to be {\it symmetric} if for every $g\in G$ we have $g\in H\Leftrightarrow g^{-1}\in H$. 
\begin{Lemma}(Lucas)\label{lm55}
Let $H$ be a symmetric subset of $(G,R)$ which contains $e$. 
\begin{enumerate}
\item $C(e,H)$ is symmetric.
\item $C(e,H)=\{h\in G\; :\; I(e,h)\subseteq H\}\cup \{h\in G\; :\; I(e,h^{-1})\subseteq H\}$. 
\item If $H$ is a subgroup, then $C(e,H)$ is a subgroup of $G$, and for every $h\in H$ we have 
$C(h,H)=hC(e,H)=C(e,H)h$. 
\end{enumerate}
\end{Lemma}
\begin{proof}
(1) Let $h\in C(e,H)$ such that $h\neq h^{-1}$. We prove that $h^{-1}\in C(e,H)$. 
We have either $I(e,h)\subseteq C(e,H)$ or 
$I(h,e)\in C(e,H)$. \\
\indent (a) Assume that $I(e,h)\subseteq C(e,H)$. \\
\indent 
If $R(e,h^{-1},h)$ holds, then $h^{-1} \in C(e,H)$. Assume that $R(e,h,h^{-1})$ holds. 
Let $t\in I(h^{-1},e)$. Then we have $R(h^{-1},t,e)$ so $R(e,t^{-1},h)$. Hence $t^{-1} \in C(h,H)\subseteq H$. 
Since $H$ is symmetric, we have and $t\in H$. Consequently, $I(h^{-1},e)\subseteq C(e,H)$. 
The subset $\{h^{-1}\} \cup I(h^{-1},e) \cup \{e\}$ is c-convex. It contains $e$, and is contained in 
$H$ hence it is contained in $C(e,H)$, so $h^{-1} \in C(e,H)$. \\
\indent (b)
Assume that $I(h,e)\subseteq C(e,H)$.\\
\indent 
If $R(h,h^{-1},e)$ holds, then $h^{-1} \in C(e,H)$. Assume that $R(e,h^{-1},h)$ holds. Let $t\in I(e,h^{-1})$.  
Then $R(e,t,h^{-1})$ so $R(e,h,t^{-1})$ and $R(h,t^{-1},e)$. In the same way as in (a), 
this proves that $I(h^{-1},e)\subseteq C(e,H)$, and we conclude $h^{-1} \in C(e,H)$. \\
\indent (2) Let $h\in C(e,H)$. Since $C(e,H)$ is c-convex, we have either $I(e,h)\subseteq C(e,H)$ or 
$I(h,e)\subseteq C(e,H)$. Since $C(e,H)$ is symmetric, then $I(h,e)\subseteq C(e,H)$ implies 
$I(e,h^{-1})\subseteq C(e,H)$. Now, let $h\in H$. If $I(e,h)\subseteq H$, then 
$\{ e\}\cup I(e,h)\cup\{ h\}$ is a c-convex subset of $G$ contained in $H$, so it is contained 
in $C(e,H)$. Since $H$ is symmetric, we have $h^{-1}\in H$, hence in the same way if $I(e,h^{-1}) 
\subseteq H$, then $h^{-1}\in C(e,H)$, so $h\in C(e,H)$. \\
\indent (3) Let $h\in H$. For every c-convex subset $F$ of $G$, $hF$ is c-convex. 
So $hC(e,H)$ is c-convex and contains $h$. Hence by the maximality of $C(h,H)$ we have $hC(e,H)\subseteq C(h,H)$. 
Conversely, let $h'\in C(h,H)$. Then either 
$I(h,h')\subseteq H$ or $I(h',h)\subseteq H$. Since $H$ is a subgroup, this is equivalent to 
either $I(e,h^{-1}h')\subseteq H$ or $I(h^{-1}h',e)\subseteq H$. Since $h^{-1}h'\in H$, by (2), this 
is equivalent to $h^{-1}h'\in C(e,H)$. Therefore, $h'\in hC(e,H)$. 
Symmetrically, we can prove that $C(h,H)=C(e,H)h$.  \\
\indent Let $h$, $h'$ in $C(e,H)$. Since $C(h,H)$ is the greatest c-convex subset which contains $h$, 
we have $C(h,H)=C(e,H)=C(h',H)$. Furthermore, $e\in C(e,H)$. Hence 
$h\!\cdot\!h'\in h\!\cdot\!h'C(e,H)=hC(h',H)=hC(e,H)=C(h,H)=C(e,H)$. 
This proves that $C(e,H)$ is a subgroup of $G$. 
\end{proof}
\begin{Lemma}(Lucas)\label{wk}
Assume that $(G,R)$ is weakly cyclically minimal and let $H$ be a definable subset. 
If $H\cap l(G) $ is a subgroup of $l(G)$, then it is a convex subgroup of $l(G)$.
\end{Lemma}
\begin{proof} If $H\cap l(G)=l(G)$, then it is a convex subgroup of $l(G)$. In the following, we assume that 
$H\cap l(G)\neq l(G)$ and we let $H'=H\cap l(G)$. \\
\indent (a) We show that for every $h\in H'$, $C(h,H)=C(h,H')$. Let $h\in H'$. Since $H'\subseteq H$, 
we have $C(h,H')\subseteq C(h,H)$. Now, 
there is $e< h_1\in l(G)\backslash H$. Since $H'$ is a subgroup, we have 
$h_1^{-1}h \in l(G)\backslash H$, $h_1h\in l(G)\backslash H$, and $h_1^{-1}h<h<h_1h$. Since $C(h,H)$ is 
c-convex, we have $C(h,H)\subseteq (h_1^{-1}h,h_1h)\subseteq l(G)$. Therefore $C(h,H)\subseteq H\cap l(G)=H'$. 
By the maximality of $C(h,H')$, we have $C(h,H)=C(h,H')$. \\
\indent (b) By Lemma \ref{lm55} (3), $C(e,H)=C(e,H')$ is a subgroup of $G$. So, it is a convex subgroup of $l(G)$. \\
%
%
\indent (c) We prove that $H'=C(e,H)$, so it is a convex subgroup of $l(G)$. 
We assume that $H'\neq C(e,H)$, and we show that the number of the 
c-convex components of $H\cap l(G)$ is infinite. By (a), $C(e,H)=C(e,H')\subseteq H'$. Hence there is 
$h\in H'\backslash C(e,H)$. 
Then $C(h,H)\cap C(e,H)=\emptyset$ (Fact \ref{fct53}). Assume that $h>e$ in $l(G)$. Then 
$e<h<h^2<\cdots<h^n<\cdots$. Since $C(e,H)$ is a convex subgroup of $l(G)$, for every 
$n\in \NN\backslash\{0\}$ we have $h^n\notin C(e,H)$. Therefore, $\forall n\in \NN\backslash\{0\}$ 
$C(h^n,H)\cap C(e,H)=\emptyset$. By Lemma \ref{lm55} (3), 
for every positive integer $n$ we have $h^nC(e,H)=h^nC(e,H')=C(h^n,H')=C(h^n,H)$. 
Let $n$, $q$ in $\NN\backslash\{0\}$. 
Therefore, $C(h^{n+q},H) \cap C(h^n,H)= h^{n+q}C(e,H)\cap h^nC(e,H))=h^n(h^qC(e,H)\cap C(e,H))= 
h^n(C(h^q,H)\cap C(e,H))=\emptyset$. 
So the $C(h^n,H)$'s are pairwise disjoint.  
Consequently, $H$ has infinitely many c-convex components. Since $H$ is definable, this contradicts the fact that 
$G$ is weakly cyclically minimal. 
The case $h<e$ is similar. 
\end{proof}
\begin{Proposition}\label{wdiv}(Lucas)
If $(G,R)$ is weakly cyclically minimal, then 
$l(G)$ is abelian and divisible. 
\end{Proposition}
\begin{proof} Let $g$ in $l(G)$, and $Com(g)$ be the definable subgroup $Com(g)=\{h\; : \; hg=gh \}$. 
By Lemma \ref{wk}, $Com(g)\cap l(G)$ is a convex subgroup of $l(G)$, and we know that $g\in Com(g)\cap l(G)$. 
It follows that the convex subgroup generated by $g$ is contained in $Com(g)$. Since this holds for 
any $g\in l(G)$, we deduce that $l(G)$ is abelian. \\
\indent For each positive integer $n$ consider\\
\indent 
 $I\! D_n(G)=
\{ e \}\cup \{g\; : \; R(g^{-1},e,g)\; \&\; \exists h\; (R(e,h,g) \;\&\; h^n =g) \}$\\
\indent  
$ \cup \{ g\; : \; R(g,e,g^{-1}) \;\&\; \exists h\; (R(g,h,e) \;\&\; h^n =g) \}$.\\
\indent 
It is a definable set. Now, $I\! D_n(G)\cap l(G)=
\{ e \}\cup \{g\in l(G)\; : \; e<g\; \&\; \exists h\in l(G)\; h^n =g \}
\cup \{ g\in l(G)\; : \; g<e\;\&\; \exists h\in l(G)\;  h^n =g\}$ is the subgroup of $n$-th powers of elements of the 
abelian group $l(G)$. 
So by Lemma \ref{wk} $I\! D_n(G)\cap l(G)$ is a convex subgroup of $l(G)$. 
Moreover it is cofinal in $l(G)$, so, for each $n$, 
$I\! D_n(l(G))\cap l(G)=l(G)\subseteq I\! D_n(G)$. Hence each element of $l(G)$ is $n$-divisible within $l(G)$. 
\end{proof}
\indent We turn to the other necessary conditions, starting with a lemma. 
\begin{Lemma}\label{wBB}(Lucas)
Assume that $(G,R)$ is weakly cyclically minimal. Let $D$ be a definable subset of $(G,R)$ and $D'=G\backslash D$. 
Then $U(D)$ and $U(D')$ are not both dense in $\mathbb{U}$. 
\end{Lemma}
\begin{proof} The definable set $D$ is a finite union of c-convex subsets. Let $C$ be such a c-convex subset, 
$g$, $h$ in $C$ and assume that $U(D')$ is dense in $\UU$. We prove that $U(g)=U(h)$. Indeed, otherwise there are $h_1'$, $h_2'$ in 
$D'$ such that $R(U(g),U(h_1'),U(h))$ and $R(U(h),U(h_2'),U(g))$ hold in $U(G)$. This implies that 
$R(g,h_1',h)$ and $R(h,h_2',g)$ hold in $(G,R)$, 
which contradicts the fact that $C$ is c-convex. It follows that $U(D)$ is finite. So it is not dense in $\UU$. 
Now, $D'$ is definable by the formula $g\notin D$. So in the same way we can prove that if $U(D)$ is 
dense in $\UU$, then $U(D')$ is finite. 
\end{proof}
\begin{Corollary}\label{cor412} 
Let $(G,R)$ be a divisible abelian nonlinear cyclically ordered group which is not c-divisible. Then 
$(G,R)$ is not weakly cyclically minimal. 
\end{Corollary}
\begin{proof}  Since $G$ is not c-divisible, its torsion subgroup is not isomorphic to $T(\UU)$. 
Hence there is a prime $p$ such that $G$ does not contain any $p$-torsion element. 
Let $D =\{ g\in G,\; :\;  \exists h\in G,\; (R(e,h,g,g^{-1}) 
\mbox{ or } R(g^{-1},g,h,e)) \mbox{ and } h^p=g\}$. 
By Lemma \ref{ouf}, both of $U(D)$ and $U(G\backslash D)$ are dense in $\UU$. 
It follows that $(G,R)$ is not weakly cyclically minimal. 
\end{proof}
\indent Not that if $(G,R)$ is c-divisible, then the set $D$ defined above is equal to $G\backslash \{ e\}$. 
Hence $G\backslash D$ is finite. 
\begin{Proposition}\label{lem59}(Lucas, and independently Kulpeshov and Verbovskiy) 
If $(G,R)$ is weakly cyclically minimal, then it is abelian.
\end{Proposition}
\begin{proof} For every $g\in G$ we let $Com(g)=\{ h\in G \; : \; hg=gh\}$ and $Ncom(g)=\{ h\in G \; : \; hg\neq gh\}$. 
Note that $Com(g)=\{ h\in G \; : \; g^{-1}hg=h\}=\{ h\in G \; : \; ghg^{-1}=h\}= Com(g^{-1})$.
We start listing some properties of $Com(g)$ and $Ncom(g)$.  \\ 
\indent (1) 
Assume that $Ncom(g)\neq \emptyset$. If $h\in Com(g)$, then $h\, Ncom(g)\subseteq Ncom(g)$. 
Indeed, if $g^{-1}g'g\neq g'$, then $g^{-1}hg'g=hg^{-1}g'g\neq hg'$.  
Hence $g'\in Ncom(g)\Rightarrow hg'\in Ncom(g)$. 
\indent If $h\in Ncom(g)$, then $h\, Com(g)\subseteq Ncom(g)$. Indeed, let $g'\in Com(g)$. Then $ghg'g^{-1}=ghg^{-1}g'\neq hg'$. \\
\indent (2)  We prove that if $(G,R)$ is not abelian and $U(G)$ contains torsion-free elements, then there is 
$g_0\in G$ such that $U(g_0)$ is torsion-free and $Ncom(g_0)\neq \emptyset$. 
Let $g\in G$ such that $Ncom(g)\neq \emptyset$. 
Since $Com(g)\cup Ncom(g)=G$, we have $U(Com(g))\cup U(Ncom(g))=U(G)$. 
So $U(Com(g))$ or $U(Ncom(g))$ contains a torsion-free element. 
We prove that in any case $Ncom(g)$ contains an element $g_0$ such that $U(g_0)$ is torsion-free.
Let $g'\in Ncom(g)$ and $h\in Com(g)$ be such that $U(g')$ is a torsion element and $U(h)$ is torsion-free. Let $g_0=hg'$. 
Then, $U(g_0)$ is torsion-free. Since $h\, Ncom(g)\subseteq Ncom(g)$, $g_0\in Ncom(g)$. So, $gg_0\neq g_0g$, and 
$Ncom(g_0)\neq \emptyset$. \\
\indent (3)  Assume that $G$ is weakly cyclically minimal and $U(G)$ contains a torsion-free element. 
If $G$ is not abelian, then there is 
$g_0$ in $G$ such that $Ncom(g_0)\neq \emptyset$ and $U(g_0)$ is torsion-free. Since $Com(g_0)$ contains the subgroup 
generated by $g_0$, $U(Com(g_0))$ is dense in $\UU$. Now, for $h\in Ncom(g_0)$ we have $h \, Com(g)\subseteq Ncom(g)$. 
Hence $U(Ncom(g))$ is dense in $\UU$. This contradicts Lemma \ref{wBB}. So $G$ is abelian. \\
\indent (4)  Assume that $G$ is weakly cyclically minimal and $U(G)\subseteq T(\UU)$. By Proposition \ref{wdiv} $l(G)$ is divisible. 
Hence by \cite[Lemmas 5.1 and 6.7]{JPR}, $G$ embeds in the lexicographic product 
$U(G)\overrightarrow{\times} l(G)$. Since $l(G)$ is abelian (Proposition \ref{wdiv}), $U(G)\overrightarrow{\times} l(G)$ 
is abelian. Therefore $G$ also is abelian. 
\end{proof}
\begin{Proposition}(Lucas)\label{lmfinite}
If $(G,R)$ is weakly cyclically minimal and not divisible, then $G\cong 
T(\UU)_n\overrightarrow{\times}l(G)$, for some $n\in \mathbb{N}\backslash \{0,1\}$, 
and $l(G)$ is nontrivial, abelian, divisible.
\end{Proposition}
\begin{proof} 
By Propositions \ref{wdiv} and \ref{lem59} $l(G)$ is abelian, divisible and $G$ is abelian. 
Assume that $U(G)$ is infinite. Let $n$ be a positive integer such that $(G,R)$ is not $n$-divisible, 
and $D_n$ (resp.\ $D_n'$) be 
the set of elements which are $n$-divisible (resp.\ not $n$-divisible). Since $g^n\in D_n$ for every $g\in G$, 
$U(D_n)$ is infinite. Therefore it is dense in $\mathbb{U}$. Let $g'\in D_n'$. Then $g'D_n\subseteq D_n'$. 
It follows that $U(D_n')$ is infinite, so it is dense. This contradicts Lemma \ref{wBB}. Therefore, 
$U(G)$ is finite, and it embeds in $T(\UU)$. 
Since $l(G)$ is divisible, by \cite[Lemmas 5.1 and 6.7]{JPR}, $G$ embeds in the lexicographic product 
$U(G)\overrightarrow{\times} l(G)$. We show that this embedding is onto. 
First, note that since $(G,R)$ is weakly cyclically minimal, $G$ is infinite. Hence $l(G)$ is nontrivial. 
Since $U(G)$ is a finite subgroup of $T(\UU)$, there exists a positive integer $n$ such that 
$U(G)$ is c-isomorphic to $T(\UU)_n$. Since $G$ is not divisible and 
$\{1\}\overrightarrow{\times} l(G)$ is divisible, we have $n>1$. 
We let $g\in G$ be such that $U(g)=e^{2i\pi/n}$. Then $U(g^n)=1$, hence $g^n\in l(G)$. 
Since $l(G)$ is divisible (Proposition \ref{wdiv}) 
there exists $h\in l(G)$ such that $g^n=h^n$. Now, $G$ is abelian (Proposition \ref{lem59}), so $(gh^{-1})^n=e$.
Therefore, 
$G$ contains a subgroup $H$ which is isomorphic to $U(G)\simeq G/l(G)$, hence 
$G=H\oplus l(G)$.  It follows that $G$ is 
c-isomorphic to $T(\UU)_n\overrightarrow{\times}l(G)$. 
\end{proof}
{\it Proof of Theorems \ref{th513} and \ref{th514}}. Let $(G,R)$ be a cyclically ordered group. \\
\indent Assume that $G$ is not divisible. By Propositions \ref{n52} and \ref{lmfinite}, 
$(G,R)$ is weakly cyclically minimal if, and only if $l(G)$ is nontrivial, abelian, divisible, and 
there is an integer $n>1$ such that $G\cong T(\UU)_n\overrightarrow{\times}l(G)$. 
Now, this group is not cyclically minimal. 
Indeed, in this cyclically ordered group, $l(G)$ is definable by the formula $x=e$ or $R(e,x,x^2,\dots,x^n)$ 
or $R(e,x^{-1},x^{-2},\dots,x^{-n})$. So, $l(G)$ is a definable subset which is not a finite union of 
open interval and singletons. \\ 
\indent Now, we assume that $G$ is divisible. \\
\indent If $(G,R)$ is abelian and c-divisible, then by Proposition \ref{n51}
it is cyclically minimal. In particular, it is weakly cyclically minimal. \\
\indent If $G$ is not abelian or if $(G,R)$ is abelian non linear and not c-divisible, then, by Proposition \ref{lem59} 
and by Corollary \ref{cor412}, $(G,R)$ is not 
weakly cyclically minimal. Hence it is no cyclically minimal. \\
\indent Finally, assume that $(G,R)$ is linearly cyclically ordered. Then $G\simeq l(G)\simeq T(\UU)_1\overrightarrow{\times} l(G)$. 
By Propositions \ref{n52} and \ref{wdiv}, $(G,R)$ is weakly cyclically minimal if, and only if, it is abelian and divisible. 
Now, it is not cyclically minimal. For example, the set $\{g\in G\; :\; g>e_G\}$ is defined by the formula 
$R(e_G,g,g^2)$, but it is not a finite union of bounded intervals and singletons. Note that in a c-divisible abelian 
cyclically ordered group, $R(e_G,g,g^2)\Leftrightarrow R(e_G,g,\zeta_2)$, where $\zeta_2\neq e_G$ and $\zeta_2^2=e_G$. 
\hfill $\qed$ \\[2mm]
\end{document}